\documentclass{m2an}
\usepackage{fullpage}
\usepackage{graphicx}
\usepackage[utf8]{inputenc}
\usepackage[english]{babel}

\usepackage{amsmath}
\usepackage{amssymb,amsthm,dsfont,enumerate,url,appendix,tabularx,mathtools}
\usepackage{todonotes, comment}

\usepackage{psfrag}
\usepackage{pstool}
\usepackage{graphicx}
\usepackage{epstopdf}

\theoremstyle{plain}
\numberwithin{equation}{section}
\numberwithin{figure}{section}
\numberwithin{table}{section}

\newtheorem{theorem}{Theorem}[section]
\newtheorem{lemma}[theorem]{Lemma}
\newtheorem{corollary}[theorem]{Corollary}
\newtheorem{proposition}[theorem]{Proposition}

\newtheorem{assumption}[theorem]{Assumption}

\theoremstyle{definition}
\newtheorem{definition}[theorem]{Definition}

\theoremstyle{remark}
\newtheorem{remark}[theorem]{Remark}

\newcommand{\R}{\mathbb{R}}

\newcommand{\N}{\mathbb{N}}
\newcommand{\Z}{\mathbb{Z}}
\newcommand{\bigO}{\mathcal{O}}

\newcommand{\abs}[1]{\left|#1\right|}
\newcommand{\norm}[1]{\left\|#1\right\|}

\newcommand{\dualproduct}[3][\Gamma]{\left\langle#2,#3 \right\rangle_{#1}}
\newcommand{\honehalftilde}[1][\Gamma]{\widetilde{H}^{1/2}(#1)}

\newcommand{\refspace}[2]{P^{#1}(#2)}
\newcommand{\polyspace}[2]{\SSnd^{#1}(#2)}

\newcommand{\dist}[2]{\operatorname{dist}\left(#1,#2\right)}

\def\supp{\operatorname*{supp}}

\def\allvertices{\mathcal{V}}
\def\innervertices{\allvertices^{int}}

\def\alledges{\mathcal{E}}
\def\inneredges{\alledges^{int}}

\newcommand{\fV}{\mathfrak V}

\def\H{\widetilde{H}}
\def\P{{\mathbb P}}
\def\R{{\mathbb R}}
\def\N{{\mathbb N}}

\def\A{{\mathcal A}}
\def\D{{\mathcal D}}

\def\P{{\mathcal P}}
\def\Q{{\mathcal Q}}

\def\SSnd{{\mathcal S}}
\def\SS{\widetilde {\SSnd}}
\def\TT{{\mathcal T}}

\def\diam{{\rm diam}}

\def\mesh{\TT}

\def\QQ{{\mathcal Q}_p}
\def\PP{{\mathcal P}_p}

\def\support#1{\operatorname{supp}#1}

\def\pp{\mathbf{p}}

\def\AVP{\mathcal{A}_{{\omega}_V}}
\def\AEP{\mathcal{A}_{{\omega}_e}}

\newcommand{\Tref}{\widehat T}
\newcommand{\mTref}{\widetilde T}

\newcommand{\Vref}{\widehat V}
\newcommand{\Sref}{\widehat S}
\newcommand{\Kref}{\widehat K}

\newcommand{\eremk}{\hbox{}\hfill\rule{0.8ex}{0.8ex}}

\newcommand{\colvec}[1]{\begin{pmatrix}#1\end{pmatrix}}

\iffalse
\usepackage{tikz}
\usetikzlibrary{shapes}

\usetikzlibrary{external}
\tikzexternaldisable

\else

\fi

\begin{document}
\title{Stable decompositions of $hp$-BEM spaces and an optimal Schwarz preconditioner
for the hypersingular integral operator in 3D}
\author{Michael Karkulik}\address{Departamento de Matem\'atica, Universidad T\'ecnica Federico Santa Mar\'ia, Valpara\'iso, Chile}
\author{Jens Markus Melenk}\address{Institut f\"ur Analysis und Scientific Computing, Technische Universit\"at Wien, Austria}
\author{Alexander Rieder}\address{Institut f\"ur Analysis und Scientific Computing, Technische Universit\"at Wien, Austria}
\date{\today}
\begin{abstract}  
  We consider fractional Sobolev spaces $H^\theta(\Gamma)$, $\theta \in [0,1]$, on a 2D surface $\Gamma$. 
  We show that functions in $H^\theta(\Gamma)$ can be decomposed into
  contributions with local support in a stable way. Stability of the decomposition is inherited by 
  piecewise polynomial subspaces. Applications include the analysis of additive Schwarz preconditioners 
  for discretizations of the hypersingular integral operator by the $p$-version of the 
  boundary element method with condition number bounds that are uniform in the polynomial degree $p$. 
\end{abstract}

\maketitle
\section{Introduction}
Fractional Sobolev spaces arise frequently in both analysis and numerical analysis when dealing with 
(fractional) partial differential or integral equations. We mention, for example, the 
classical boundary integral operators associated with the Laplacian that lead to the 
Sobolev space spaces $H^{1/2}(\Gamma)$.
Similarly, for example in screen problems or fractional diffusion, the natural spaces are
often given by Sobolev spaces encoding homogeneous Dirichlet conditions denoted by $\widetilde{H}^{\theta}(\Gamma)$.

When discretizing problems posed 
in such spaces, a standard ansatz space consists of globally continuous piecewise polynomials of degree $p$ 
on a mesh $\TT$ that partitions $\Gamma$. An important tool in the numerical analysis of 
such a setting are \emph{stable} decompositions of such discrete functions into local contributions. 

In this work, we consider 2D surfaces $\Gamma \subset \R^3$ and 
propose a stable decomposition procedure of functions $u \in \widetilde{H}^{\theta}(\Gamma)$ into
a global low-order contribution of piecewise linears/bilinears and functions with local support
(see Theorem~\ref{thm:main_decomposition}).  This decomposition is constructed such that, 
if $u$ is a piecewise polynomial on a mesh $\TT$, then the local contributions are also 
piecewise polynomials with the same degree distributions. Our construction accommodates 
variable polynomial degree distribution (i.e., the polynomial degree is allowed to vary from element to element) and ``mixed'' meshes consisting of triangles and quadrilaterals
(see Theorem~\ref{thm:decomposition-FEM-space}).
Similar decompositions that emphasize the $p$-dependence have already appeared 
in the context of meshes consisting of quadrilaterals only in \cite{pavarino94} for the case $\theta=1$
and~\cite{heuer99} for $\theta \in (0,1)$. For meshes consisting of triangles/tetrahedra only, 
the decomposition in the finite element case of $\theta=1$ was covered 
in \cite{schoeberl-melenk-pechstein-zaglmayr08} and recently a new decomposition
was proposed in \cite{falk-winther15} for general simplicial meshes. The decomposition of \cite{falk-winther15} 
is similar to our result in that it is also simultaneously stable in $L^2$ and $H^1$,
which implies stability for general $\theta \in (0,1)$, although this simultaneous stability is 
not emphasized in \cite{falk-winther15} and its ramifications not explored. 

An application of our stable decomposition is given by the analysis in 
Section~\ref{sec:ams} of an additive Schwarz preconditioner for the $p$-BEM applied to the 
hypersingular integral equation. The resulting condition number is shown to be uniformly bounded 
in the polynomial degree $p$ employed; here, mixed meshes consisting of triangles and quadrilaterals 
are allowed as well as a variable polynomial degree distribution. The numerical performance 
of this preconditioner is studied in \cite{fuehrer-melenk-praetorius-rieder15}. 

A second application of our decomposition result is given in Theorem~\ref{thm:interpolation_poly_l2}, 
which identifies the interpolation space between spaces of piecewise polynomials (of arbitrary degree) 
equipped with weighted $L^2$-norms. 
In fact, in a subsequent work \cite{our_interpolation_space_paper} we will use our decomposition to show
that the interpolation space obtained by interpolating (using the $K$-method) between 
a space of piecewise polynomials equipped with the $L^2$- and $H^1$-norm coincides with the same space
equipped with the appropriate Sobolev norm, i.e., 
\begin{align*}
  \left(\left(\SS^{\pp,1}(\mesh),\|\cdot\|_{L^2(\Gamma)}\right),\left(\SS^{\pp,1}(\mesh),\|\cdot\|_{\widetilde H^1(\Gamma)}\right)\right)_{\theta,2} & = 
\left(\SS^{\pp,1}(\mesh),\|\cdot\|_{\widetilde H^{\theta}(\Gamma)}\right) 
\qquad \mbox{ (equivalent norms)}, \\  
  \left(\left(\SSnd^{\pp,1}(\mesh),\|\cdot\|_{L^2(\Gamma)}\right),\left(\SSnd^{\pp,1}(\mesh),\|\cdot\|_{H^1(\Gamma)}\right)\right)_{\theta,2} & = 
\left(\SSnd^{\pp,1}(\mesh),\|\cdot\|_{H^{\theta}(\Gamma)}\right) 
\qquad \mbox{ (equivalent norms)}; \\  
\end{align*}
here, the implied norm equivalences are independent of the polynomial degree distribution. 
This result of \cite{our_interpolation_space_paper} will generalize the analysis of the single-element case in 
\cite{maday89,bernardi-dauge-maday92,bernardi-dauge-maday10,belgacem94}
and \cite[Thm.~{4.2}]{bernardi-dauge-maday07,bernardi-dauge-maday10} to general shape-regular meshes. 

\subsection*{Outline of the paper}

In view of the technical nature of the paper, we have collected the main results in 
Section~\ref{sect:main_results}. Applications such as the preconditioning
of the $p$-version discretization of the hypersingular integral operator are given
in Section~\ref{sec:applications}. The remainder of the paper is devoted to the proof
of the main result, namely, the decomposition of a function $u$ into local components. The 
decomposition is such that the local components are 
supported on the patches of a mesh, i.e., the union of elements meeting at a vertex or an edge. It is performed in several steps. In a first step, 
a piecewise linear contribution $u_1$ is determined using a Scott-Zhang interpolant. The primary purpose 
of this contribution is to remove the $h$-dependence of the decomposition, i.e., the difference $u - u_1$ 
can be localized with control of the constants uniformly in the size of the mesh patches. In a second step, 
``vertex contributions'' for each vertex $V$ of the mesh are determined with the aid of 
averaging operators. The vertex contributions have two important properties: a) the support of each vertex contribution is the 
corresponding vertex patch $\omega_V$ and b) the function value at $V$ is reproduced 
if the function is continuous at $V$. We collect these vertex contributions in the function $u_2$. 
Then the function $u - u_1 - u_2$ vanishes in the vertices of the mesh. A second averaging operator 
is applied to this function for each edge to create ``edge contributions'' supported by the patches 
$\omega_e$ associated with the edges. These are collected in a function $u_3$. 
The difference $u - u_1 - u_2 -u_3$ is then a function
vanishing on the skeleton of the mesh and therefore consists of local ``element contributions''. 

The stability properties of the averaging operators aplied in these steps of the decomposition 
have to be analyzed. In preparation to that analysis, we address in Section~\ref{sec:interpolation-argument} 
the issue that full Sobolev norms do not scale conveniently under (affine) changes of variables 
but that seminorms do. Since we define fractional Sobolev spaces by interpolation, we present 
and analyze interpolation spaces that are based on interpolating seminorms. These results come in two flavors: 
Section~\ref{sect:interpol_seminorms} focuses on the interpolation between Sobolev spaces, and 
Section~\ref{sec:interpolation-of-weighted-spaces} studies the case of interpolation between 
weighted Sobolev spaces with weight given by the distance from a point. The ensuing 
Section~\ref{sec:averaging-operators} develops averaging operators on the 
reference element $\Kref$ with the following key property: ``Vertex averaging operators'' reproduce 
the function value in one corner of $\Kref$ and ``edge averaging operators'' reproduce function values
on one edge of $\Kref$. The actual construction is performed in several steps since these operators 
should have the additional property that the ``vertex averaging operator'' should vanish on the edge
opposite the vertex and the ``edge averaging operator'' should vanish on the other edges of $\Kref$. 
Section~\ref{sec:decomposition-general-triangulations} uses the averaging operators on the reference
element $\Kref$ to generate the vertex or edge contributions mentioned above. The general procedure 
for meshes consisting of triangles only is as follows: one selects from the (vertex or edge) patch 
a certain element $K^\star$. In the discrete case of piecewise polynomials, $K^\star$ is normally chosen to have 
the lowest polynomial degree present in the patch. (In some cases the element with a slightly higher polynomial degree 
is favored if it allows one to do the averaging on a triangle instead of a square). 
The corresponding vertex or edge contribution is defined for \emph{that} element $K^\star$ by applying 
the averaging operator to the pullback to the reference element. For the \emph{remaining} elements of the patch, 
the contribution is defined by copying the values from $K^\star$. 

An important feature of the averaging operators for triangles is that polynomials of degree $p$ 
are mapped to polynomials of the same degree. In this way, the above decomposition procedure also yields 
a stable decomposition for spaces of piecewise polynomials. The presence of quadrilaterals in a mesh
introduces significant complications. We define the averaging operators on the reference square 
in terms of the averaging operators on a triangular part of the square combined with mapping this triangular
part to the full square using a Duffy transformation. When applied to polynomials of degree $p$, this 
increases the polynomial degree to $2p$. In the discrete case, we therefore introduce an additional Gauss-Lobatto 
interpolation step to get back to polynomials of degree $p$. Due to the lack of stability of the 
Gauss-Lobatto interpolation in $L^2$, the stability of the decomposition in fractional
Sobolev norms $\widetilde{H}^\theta(\Gamma)$ cannot easily be inferred  from stability 
in $L^2(\Gamma)$ and $\widetilde{H}^1(\Gamma)$ and an interpolation argument. Rather, a careful 
analysis in the local $\widetilde H^\theta$-spaces is necessary, which is done in the present paper. 
%



\section{Notation, assumptions, and main results}
\label{sect:main_results}
We will introduce the necessary notation and present the main results of the present work.
Due to the technical nature of the results, all proofs are relegated to Section~\ref{sec:decomposition-general-triangulations}.

$B_r(x)$ denotes the Euclidean ball with radius $r$ centered at $x$.
The shorthand $a \lesssim b$ expresses $a \leq C b$ for a constant $C > 0$ that does not depend on parameters of 
interest (in particular the mesh size $h$ and the polynomial degree $p$). The notation  
$a \sim b$ is short for $a \lesssim b$ in conjunction with $b \lesssim a$. 

\subsection{Geometric and functional setting}\label{subsection:setting}
Let $\Omega \subseteq \R^3$ be a bounded Lipschitz domain and let $\Gamma \subseteq \partial \Omega$ 
be an open, connected subset of $\partial \Omega$ that stems from a Lipschitz dissection as described 
in \cite[p.~{99}]{mclean}. 
The Sobolev spaces $L^{2}(\partial\Omega)$ and $H^1(\partial\Omega)$ are 
defined as in \cite[pp.~{99}]{mclean} by use of Bessel potentials on $\R^{2}$ and liftings
via the bi-Lipschitz maps that describe $\partial\Omega$.
For any relatively open $\omega \subseteq \partial\Omega$, we 
define the space $L^2(\omega)$ of square integrable functions in the standard way. The spaces 
$H^1(\omega)$ and $\widetilde H^1(\omega)$ are defined by 
\begin{align}
H^{1}(\omega) &:= \{v|_\omega \colon  v \in H^{1}(\partial\Omega)\},  
& 
\widetilde H^{1}(\omega) &:= \{v \colon  E_{0,\omega} v \in H^{1}(\partial\Omega)\}, 
\end{align}
where $E_{0,\omega}$ denotes the operator that extends a function defined on $\omega$ to a function
on $\partial\Omega$ by zero. We recall that for each $u \in H^1(\partial\Omega)$ we can define
the surface gradient $\nabla u \in L^2(\partial\Omega)$. It can be checked that 
for (relatively) open $\omega\subseteq \partial\Omega$ and $u \in H^1(\omega)$ 
or $u \in \widetilde H^1(\omega)$ 
the surface gradient $\nabla u$ is also well-defined on $\omega$
and depends only on the function values of $u$ on $\omega$.
With the surface gradient in hand, we introduce the seminorm and norm 
\begin{equation}\label{eq:semi-norm}
  |u|_{H^1(\omega)}:= \|\nabla u\|_{L^2(\omega)}, 
  \qquad \|u\|_{H^1_h(\omega)}^2 :=
  \frac{1}{\diam(\omega)^2} \|u\|_{L^2(\omega)}^2 + |u|_{H^1(\omega)}^2. 
\end{equation}
Fractional Sobolev spaces are defined by interpolation
via the $K$-method.
For two Banach spaces $(X_0,\norm{\cdot}_{0})$, $(X_1,\norm{\cdot}_{1})$ with continuous embedding
$X_1 \subseteq X_0$ and fixed $\theta \in (0,1)$, define the $K$-functional by $K^2(t,u):= 
\inf_{v \in X_1} \|u - v\|_0^2 + t^2 \|v\|^2_1$ and 
the interpolation norm by
\begin{align}
  \label{eq:def_K_functional}
  \norm{u}^2_{(X_0,X_1)_{\theta,2}}
  &:= \int_{t=0}^\infty t^{-2\theta} K^2(t,u) \frac{dt}{t}
  = \int_{t=0}^\infty t^{-2\theta} \bigg(\inf_{v \in X_1} \|u - v\|_{0}^{2} + t^{2} \|v\|_1^{2} \bigg) \frac{dt}{t}    
\end{align}
together with the convention 
that for $\theta=0$ or $\theta = 1 $ we set $\norm{u}^2_{(X_0,X_1)_{\theta,2}}=\norm{u}^2_{X_\theta}$.

We then define the interpolation space $\left(X_0,X_1\right)_{\theta,2}
:=\{ u \in X_0: \, \norm{u}_{(X_0,X_1)_{\theta,2}} < \infty \}$.
In this way, the spaces
$H^\theta(\omega) := (L^2(\omega),H^1(\omega))_{\theta,2}$,
$\widetilde H^\theta(\omega):= (L^2(\omega),\widetilde H^1(\omega))_{\theta,2}$
as well as
$H^\theta_h(\omega) := (L^2(\omega),H^1_h(\omega))_{\theta,2}$,
$\widetilde H^\theta_h(\omega):= (L^2(\omega),\widetilde H^1_h(\omega))_{\theta,2}$
and their corresponding norms are defined.
Occasionally, it will be more convenient to work with seminorms instead of full norms.
For two Banach spaces $X_0$, $X_1$ with
$\norm{\cdot}_{X_1}^2=\abs{\cdot}_{X_1}^2 + c^2\norm{\cdot}_{X_0}^2$,
we define a seminorm by
\begin{align*}
  |u|^2_{(X_0,X_1)_{\theta,2}} &= \int_{t=0}^\infty t^{-2\theta} \left(\inf_{v \in X_1} \|u - v\|_{0}^2 + t^2|v|^2_1\right) \frac{dt}{t}.
\end{align*}
For example, on the space $\H_h^\theta(\omega)$ we define in this way the seminorm
\begin{align*}
  |u|^2_{\H_h^{\theta}(\omega)}&:= \int_{t=0}^\infty t^{-2\theta}
  \bigg(\inf_{v \in \widetilde{H}^1(\omega)} \|u - v\|_{L^2(\omega)}^2 + t^2|v|_{H^1(\omega)}^2\bigg) \frac{dt}{t}=|u|^2_{\H^{\theta}(\omega)} .
\end{align*}
We refer to Section~\ref{sect:interpol_seminorms} on how these seminorms relate to the full interpolation norms.
\subsection{Meshes and polynomial spaces} 
In this section, we further restrict our assumptions on the surface $\Gamma$ and introduce the discretization into
  boundary elements which are the image of planar reference elements under suitable transformations. 

We require that $\Gamma$ admits a suitable triangulation $\TT$  into open subsets $K_1, \dots, K_{\abs{\TT}}$, 
satisfying Assumption~\ref{assumption:element_maps} below. As it is standard in FEM and BEM, each element $K_i$ is 
the image of some fixed reference element under an element map $F_{K_i}$. 
To that end, we define the reference triangle and square by 
\begin{align}
\label{eq:reference-triangle}
\Tref &:=\{(\xi,\eta)\,:\, 0 < \xi < 1, \quad 0 <  \eta < \xi \}, \\
\Sref &:=\{(\xi,\eta)\,:\, 0 < \xi < 1, \quad 0 <  \eta < 1 \}. 
\end{align}
Often, we will work with functions that are only defined on a subdomain  $\omega \subseteq \Gamma$. 
Correspondingly, we write $\TT|_{\omega}:=\left\{ K \in \TT: K \subset \omega \right\}$ for the 
subtriangulation. Throughout the article, the triangulations and the element maps are required to satisfy
the following assumption, which generalizes the usual concept of 
shape-regularity to elements on surfaces, 
where the element maps are mappings ${\mathbb R}^2\rightarrow {\mathbb R}^3$:
\begin{assumption}
  \label{assumption:element_maps}
  \begin{enumerate}
  \item For each element $K \in \TT$, there exists $\Kref \in \{\Tref,\Sref\}$ and 
   an element map $F_K:\Kref \rightarrow K$ that is $C^1$ up to the boundary $\partial\Kref$. 
  \item The element maps $F_K$ are bijections $\overline{\Kref} \rightarrow \overline{K}$. 
  \item With $h_K:= \operatorname*{diam} K$, the Gramian $G(x) := (F_K^\prime(x))^\top \cdot F_K^\prime(x)$ 
    has two eigenvalues $\lambda_{min}(x)$, $\lambda_{max}(x)$. 
    The shape regularity 
    of $\TT$ is characterized by the constant $\gamma$ 
    that satisfies, for all $K \in \TT$, 
    \begin{equation}
      \label{eq:shape-regularity}
      \sup_{x \in \Kref} \frac{h_K^2}{\lambda_{min}(x)} + \frac{\lambda_{max}(x)}{h_K^2} \leq \gamma. 
    \end{equation}
    Note that this implies $\|F_K^\prime\|_{L^\infty(\Kref)} \leq C h_K$.  
  \item 
    The triangulation is {\em regular}\/: The intersection $\overline{K}_1 \cap \overline{K}_2$ of two elements 
    $K_1 \ne K_2 \in {\mathcal T}$ is either empty, exactly one vertex or exactly one edge (including its two endpoints). 
    If the intersection $\overline{K}_1 \cap \overline{K}_2$ is an edge 
    $e = F_{K_1}(\widehat e_1) = F_{K_2}(\widehat e_2)$ (for two edges $\widehat e_1$, $\widehat e_2$ of the reference
    element $\Kref$), then $F_{K_1}^{-1} \circ F_{K_2}|_{\widehat e_1}: \widehat e_1 \rightarrow \widehat e_2$ 
    is an affine bijection.
  \item Each edge is either fully contained in $\Gamma$ or part of $\partial \Gamma$.
  \end{enumerate}
\end{assumption}

A crucial role in our construction will be played by the elements sharing a vertex or an edge. 
We introduce the following definition. 
\begin{definition}
  \label{def:patches}
  Denote the set of all vertices by $\allvertices$ and the set of all edges by $\alledges$. 
  For a vertex $V \in \allvertices$ and an edge $e \in \alledges$ we denote the vertex and edge patches by
  \begin{equation*}
    \omega_{V}:=\operatorname{interior} \left( \bigcup_{ K \in \TT: V \in \overline{K}} \overline{K} \right), \qquad
    \omega_{e}:=\operatorname{interior} \left( \bigcup_{ K \in \TT: e \subset \overline{K}} \overline{K} \right).
  \end{equation*}
  For each patch denote the local mesh size as $h_{V}:=\diam\left(\omega_{V}\right)$ and $h_{e}:=\diam{\left(\omega_{e}\right)}$. 
  The set of vertices ${\mathcal V}$ of the triangulation $\TT$ is decomposed as 
  ${\mathcal V} =  {\mathcal V}^{int} \dot \cup {\mathcal V}^{bdy}$, where 
  ${\mathcal V}^{int} = \{z \in {\mathcal V}\colon z \in \Gamma\}$ and 
  ${\mathcal V}^{bdy} = \{z \in {\mathcal V}\colon z \in \partial \Gamma\}$.  
  Analogously we decompose the edges as $\mathcal{E} = \mathcal{E}^{int} \dot\cup \mathcal{E}^{bdy}$.
\end{definition}
We are interested in the space of piecewise polynomials on the triangulation.
For $p \in \N$, we denote by ${\mathcal P}_p := \operatorname*{span} \{x^i y^j\colon 0 \leq i,j \leq p \wedge\; i+j\leq p\}$ 
the space of polynomials of (total) degree $p$ and 
by ${\mathcal Q}_p:= \operatorname*{span} \{x^i y^j\colon 0 \leq i,j \leq p\}$ the tensor-product space of polynomials
of degree $p$ in each variable.  We write 
\begin{equation}
\refspace{p}{\Kref}:= 
\begin{cases} 
{\mathcal P}_p & \mbox{ if $\Kref = \Tref$}, \\ 
{\mathcal Q}_p & \mbox{ if $\Kref = \Sref$}. 
\end{cases} 
\end{equation}
If we want to emphasize the  domain under consideration, we write $\PP(\widehat{K})$ and $\QQ(\widehat{K})$ for $\widehat{K} \in \{\Sref,\Tref\}$
for the two different polynomial spaces.
For each element $K \in \TT$ we choose a polynomial degree $p_K \in \N$ and collect them in the family
$\pp:=(p_K)_{K\in \TT}$. We define the space of piecewise polynomials as:
\begin{align*}
  \polyspace{\pp,0}{\TT}:=\left\{ u \in L^2(\Gamma) : u \circ F_K \in \refspace{p_K}{\Kref} \; \forall K \in \TT \right\}.
\end{align*}

For the discretization of the spaces $H^{\theta}(\Gamma)$ and $\widetilde{H}^{\theta}(\Gamma)$ we consider spaces of globally continuous piecewise polynomials:
\begin{align*}
  \SSnd^{\pp,1}(\TT)&:= \polyspace{\pp,0}{\TT} \cap H^1(\Gamma), \\
  \SS^{\pp,1}(\TT)&:= \SSnd^{\pp,1}(\TT) \cap \widetilde H^1(\Gamma) = \left\{ u \in \SSnd^{\pp,1}(\TT)\, : \;  u|_{\partial \Gamma} = 0 \right\}.
\end{align*}

For subtriangulations $\TT|_{\omega}$ we define $\SS^{\pp,1}(\TT|_{\omega})$ analogously,
i.e., globally continuous piecewise polynomials that vanish on $\partial \omega$.
We introduce the piecewise constant local mesh size function $h \in \SSnd^{0,0}(\TT)$ as 
the function  satisfying $h|_{K}:=\diam(K)$ for all $K \in \TT$ and the 
polynomial degree distribution $p \in \polyspace{0,0}{\TT}$ as $p|_K:= p_K$, for all $K \in \TT$.
For a set $M\subset\R^d$, we denote by $d_M(\cdot) := \dist{\cdot}{M}$ the distance to the set $M$.

  For some results, we rely on the following assumption regarding the polynomial degree distribution wherever triangles and quadrilaterals meet:
  \begin{assumption}
    \label{ass:degree_distribution}
    If a triangle $T$ and a quadrilateral $S$ of the triangulation $\TT$ share an edge $e$, the corresponding polynomial degrees $p_T$ and $p_S$
    satisfy
    \begin{align}
      p_T \leq p_S \qquad \text{ or } \qquad p_S \leq 2p_T.
    \end{align}
  \end{assumption}

\subsection{Main results} 
\label{sec:main-results}
The main result of this paper, which underlies the stability of the additive Schwarz preconditioner
discussed in Section~\ref{sec:ams} ahead, 
states that we can decompose the space $\SS^{\pp,1}(\TT)$ into 
local contributions in an $H^\theta$-stable way.
To that end, define the spaces
\begin{equation}
\begin{aligned}[c]
X_0 &= \prod_{V \in \innervertices} L^2(\omega_V),  \\
Y_0 &= \prod_{e \in \inneredges} L^2(\omega_e), \\
Z_0 &= \prod_{K \in \TT} L^2(K), \\
\end{aligned}
\qquad
\begin{aligned}[c]
X_1 &= \prod_{V \in \innervertices} \widetilde H^1_h(\omega_V), \\
Y_1 &= \prod_{V \in \inneredges} \widetilde H^1_h(\omega_e), \\
Z_1 &= \prod_{K \in \TT} \widetilde H^1_h(K), 
\end{aligned}
\label{eq:definition_xyz}
\end{equation}
equipped with norms and seminorms as described in Section~\ref{subsection:setting}.

The following theorem 
has in a similar form already appeared in \cite{heuer99} for rectangular
meshes and was presented for triangulations in \cite[Lemma 5.2]{fuehrer-melenk-praetorius-rieder15}.
\begin{theorem}
\label{thm:decomposition_coloring_estimate}
For $\theta \in (0,1)$, there exists a constant $C_{\theta}$ that depends only on $\Gamma$, $\Omega$,
the $\gamma$-shape regularity of the triangulation $\TT$, and $\theta$, 
such that for all $u \in \widetilde{H}^{\theta}(\Gamma)$,
and for all decompositions
\begin{align*}
  u=u_1 + \sum_{V \in \innervertices} {u_{V} } + \sum_{e \in \inneredges} {u_e} + \sum_{K \in \TT}{u_K}
\end{align*}
with $\support{u_V} \subseteq \omega_V$, $\support{u_e} \subseteq \omega_e$ and $\support{u_K}\subseteq K$, we can estimate:
\begin{align*}
  \norm{u}_{\widetilde{H}^{\theta}(\Gamma)}^2
  &\leq
    C_{\theta}\left( \norm{u_1}_{\widetilde{H}^{\theta}(\Gamma)}^2 + \sum_{V \in \innervertices}{ \norm{u_V}_{\widetilde{H}^{\theta}(\Gamma)}^2} + 
    \sum_{e \in \inneredges}{ \norm{u_e}_{\widetilde{H}^{\theta}(\Gamma)}^2} + 
    \sum_{K \in \TT}{ \norm{u_K}_{\widetilde{H}^{\theta}(\Gamma)}^2} \right)\\
  &\leq C_{\theta}\left( \norm{u_1}_{\widetilde{H}^{\theta}(\Gamma)}^2 + \sum_{V \in \innervertices}{ \norm{u_V}_{\widetilde{H}_h^{\theta}(\omega_V)}^2} + 
    \sum_{e \in \inneredges}{ \norm{u_e}_{\widetilde{H}_h^{\theta}(\omega_e)}^2} + 
    \sum_{K \in \TT}{ \norm{u_K}_{\widetilde{H}^{\theta}_h(K)}^2} \right).
\end{align*}
\end{theorem}
\begin{proof}
  The proof is based on a so called coloring argument, we include it for completeness.
  The main ingredient is the following estimate (see \cite[Lemma 4.1.49]{ss} or \cite[Lemma 3.2]{petersdorff_rwp_elasticity}):
  Let $w_j$, $j=1,\dots, n$ be functions in $\widetilde{H}^{\theta}(\Gamma)$ for $\theta \geq 0$ with pairwise disjoint support. Then it holds
  \begin{align}
    \label{eq:coloring_est}
    \Big\|{\sum_{i=1}^{n}{w_i}}\Big\|_{\widetilde{H}^{\theta}(\Gamma)}^2 &\lesssim \sum_{i=1}^{n}{\norm{w_i}^2_{\widetilde{H}^{\theta}(\Gamma)}}.
  \end{align}
  (The implied constant is known for the case of using the Aaronstein-Slobodeckij norm to be $5/2$, \cite[Lemma 4.1.49]{ss}).
  For notational simplicity we only consider the vertex contributions, the edge and inner parts are treated in exactly the same way.
  By $\gamma$-shape regularity, the number of vertices in a patch can be uniformly bounded by some constant
  $N_{c}$.  Thus, we can sort the vertices into sets $J_1,\dots,J_{N_{c}}$, such that
  $\bigcup_{i=1}^{N_c}{ J_i} = \innervertices$ and $ \abs{\omega_{V} \cap \omega_{V'} } = 0$ for all $V,V'$ in the same index set $J_i$.
  Repeated application of the triangle inequality and~\eqref{eq:coloring_est} then gives:
  \begin{align*}
    \norm{u}_{\widetilde{H}^{\theta}(\Gamma)}^2
    &\leq 2\norm{u_1}_{\widetilde{H}^{\theta}(\Gamma)}^2 + 2 \,\Big\|{\sum_{V \in \innervertices}{ u_{V}}}\Big\|_{\widetilde{H}^{\theta}(\Gamma)}^2 
    \leq 2 \norm{u_1}_{\widetilde{H}^{\theta}(\Gamma)}^2 + 2 N_c \sum_{i=1}^{N_c}{ \Big\|{\sum_{V \in J_i}{u_V}}\Big\|_{\widetilde{H}^{\theta}(\Gamma)}^2 }  \\
    &\leq 2 \norm{u_1}_{\honehalftilde}^2 + 2\, N_c \, C \!\!  \sum_{V \in \innervertices}{ \norm{u_{V}}_{\widetilde{H}^{\theta}(\Gamma)}^2 }.
  \end{align*}
  To prove the second estimate, see Remark~\ref{rem:norm-equivalence-extension}.
\end{proof}
The primary objective of the present work is to 
provide the following converse estimate:
\begin{theorem}[stable decomposition of $\widetilde{H}^\theta(\Gamma)$---continuous and discrete]
\label{thm:main_decomposition}
\begin{enumerate}[(i)]
\item 
\label{item:thm:main_decomposition-i}
Any function $u \in L^2(\Gamma)$ can be written in the form
\begin{equation}
\label{eq:decomposition}
u = u_1 + \sum_{V\in \innervertices} u_V  + \sum_{e\in \inneredges} u_e + \sum_{K \in \TT} u_K,
\end{equation}
where the components $u_1 \in \SS^{1,1}(\TT)$, 
$(u_V)_{V \in \innervertices} \subset X_0$, 
$(u_e)_{e \in \inneredges} \subset Y_0$, 
and $(u_K)_{K\in\TT} \subset Z_0$ of the 
decomposition (\ref{eq:decomposition}) depend linearly on $u$, 
and the corresponding linear maps have the 
following mapping properties: 
\begin{equation*}
  \begin{aligned}[c]
    u &\mapsto u_1 \\
    u &\mapsto (u_V)_V \\
    u &\mapsto (u_e)_e \\
    u &\mapsto (u_K)_K
  \end{aligned}
  \begin{aligned}[c]
    :\\:\\:\\:
  \end{aligned}
\qquad
  \begin{aligned}[c]
    L^2(\Gamma) &\rightarrow \left(\SS^{1,1}(\TT),\norm{\cdot}_{L^2}\right),  \\
    L^2(\Gamma) &\rightarrow X_0,  \\
    L^2(\Gamma) &\rightarrow Y_0,  \\
    L^2(\Gamma) &\rightarrow Z_0, 
  \end{aligned} 
  \qquad
  \begin{aligned}
    \widetilde H^1(\Gamma) &\rightarrow \left(\SS^{1,1}(\TT),\norm{\cdot}_{H^1}\right), \\
    \widetilde H^1(\Gamma) &\rightarrow X_1, \\
    \widetilde H^1(\Gamma) &\rightarrow Y_1,\\
    \widetilde H^1(\Gamma) &\rightarrow Z_1. 
  \end{aligned}
\end{equation*}
The constants of these bounded linear maps depend solely on $\Gamma$, $\Omega$, and the shape regularity constant of ${\mathcal T}$. 
Additionally, for every $\theta \in (0,1)$ there exist constants $C_\theta$ 
(depending solely on $\theta$, $\Gamma$, $\Omega$, and the shape regularity) such that:

\begin{eqnarray*}
\norm{u_1}_{\widetilde{H}^{\theta}(\Gamma)}^2  \leq C_{\theta} \norm{u}^2_{\widetilde{H}^{\theta}(\Gamma)}, \\
\sum_{V \in {\mathcal V}^{int}} |u_V|^2_{\widetilde{H}_h^{\theta}(\omega_V)} + h_V^{-2\theta} \|u_V\|^2_{L^2(\omega_V)} \leq C_{\theta} \|u\|^2_{\widetilde H^\theta(\Gamma)}, \\
\sum_{e\in{\mathcal E}^{int}} |u_e|^2_{\widetilde{H}_h^{\theta}(\omega_e)} + h_e^{-2\theta} \|u_e\|^2_{L^2(\omega_e)} \leq C_{\theta} \|u\|^2_{\widetilde H^\theta(\Gamma)}, \\
\sum_{K\in{\mathcal T}} |u_K|^2_{\widetilde{H}_h^{\theta}(K)} + h_K^{-2\theta} \|u_K\|^2_{L^2(K)} \leq C_{\theta} \|u\|^2_{\widetilde H^\theta(\Gamma)}. 
\end{eqnarray*}
\item
\label{item:thm:main_decomposition-ii}
If $\TT$ consists of triangles only and if $u \in \SS^{\pp,1}(\TT)$ 
then each of the contributions $u_V$, $u_e$, $u_K$ is in $\SS^{\pp,1}(\TT)$.  
\item
\label{item:thm:main_decomposition-iii}
If $\TT$ consists of triangles and quadrilaterals
and if $u \in \SS^{\pp,1}(\TT)$, then each of the contributions 
$u_V$, $u_e$, $u_K$ is in $\SS^{2\pp,1}(\TT)$. 
\end{enumerate}
\end{theorem}

We note that the decomposition of Theorem~\ref{thm:main_decomposition}, (\ref{item:thm:main_decomposition-iii}) 
in the general case of meshes containing both triangles and quadrilaterals does not ensure that the contributions
are in $\SS^{\pp,1}(\TT)$. This makes Theorem~\ref{thm:main_decomposition} not directly applicable for the analysis
of additive Schwarz methods on meshes consisting of triangles and quadrilaterals.
A modification, which, however, relies on the function $u$ to be discrete, rectifies this deficiency: 
\begin{theorem}[stable localization of $\SS^{\pp,1}(\TT)$]
\label{thm:decomposition-FEM-space}
If Assumption~\ref{ass:degree_distribution} holds, any function $u \in \SS^{\pp,1}(\TT)$ can be decomposed in the way described in Theorem~\ref{thm:main_decomposition}
with the additional feature that the contributions $u_V$, $u_e$, $u_K$ are in $\SS^{\pp,1}(\TT)$.
If Assumption~\ref{ass:degree_distribution} is not satisfied, the stability estimates only hold for $\theta \in \{0,1\}$.
\end{theorem} 
\begin{theorem}[stable decomposition of $H^\theta(\Gamma)$ --- continuous and discrete]
\label{rem:decomposition}
The statements of Theorems~\ref{thm:main_decomposition}, \ref{thm:decomposition-FEM-space} are also true in the spaces $H^{\theta}(\Gamma)$ instead of $\widetilde H^{\theta}(\Gamma)$. Then, the sums run over all vertices/edges instead of just the interior ones. 
\end{theorem}

  \begin{remark}
    \label{rem:why_not_interpolate_at_the_end}
  While one can prove Theorem~\ref{thm:main_decomposition} by first considering the cases $\theta=0$ and $\theta=1$ and then using an interpolation argument, the proof of Theorem~\ref{thm:decomposition-FEM-space}
  is more involved due to the fact that mixed meshes are considered. 
  Indeed, the decomposition operators that do not increase the polynomial
  degrees for quadrilaterals are only stable when applied to (piecewise)
  polynomial functions (due to the Gauss-Lobatto interpolation step).
  Performing the interpolation step at the end would therefore require a stability estimate of the
  form
  $$
  \norm{u}_{\left(\left(\SS^{\pp,1}(\mesh),\|\cdot\|_{L^2(\Gamma)}\right),\left(\SS^{\pp,1}(\mesh),\|\cdot\|_{\widetilde H^1(\Gamma)}\right)\right)_{\theta,2}}
  \lesssim \norm{u}_{\widetilde{H}^{\theta}(\Gamma)}.
  $$
  Such estimates are not available in the required generality (in fact, the present decomposition result forms the basis 
  for a proof of such estimates). Therefore we cannot rely on performing a single interpolation step at the end.
  Instead we carefully work with fractional Sobolev norms throughout our proof, making sure to 
  perform interpolation arguments only on the reference element where the necessary norm equivalences are known 
  (see Proposition~\ref{lemma:GL}).
\eremk
\end{remark}

\section{Applications}
\label{sec:applications}
\subsection{Interpolation of discrete $L^2$ spaces with weights}
Interpolation of piecewise polynomial spaces equipped with weighted $L^2$ norms
  is common when working  with a non-uniform triangulation or non-constant polynomial degree,
  and appears, e.g., in the context of inverse estimates, \cite{georgoulis08}.
  Theorem~\ref{thm:interpolation_poly_l2} below provides
  a general setting for such applications.

It is well-known that the interpolation of $L^2$ spaces with different weights gives
the $L^2$ space with corresponding interpolated weight:
\begin{proposition}[{\cite[Lemma 23.1]{tartar07}}]\label{prop:tartar}
  \label{prop:interpolation_weighted_l2}
  Let $w_0,w_1$ be positive, measurable functions, and let $\omega \subseteq \Gamma$ be open.
  Let $X_0:=L^2(\omega;w_0)$ and $X_1:=L^2(\omega;w_1)$ be the weighted Lebesgue spaces with norm $\|u\|_{L^2(\omega; w_i)} = (\int_\omega |u(x)|^2 \omega_i\,dx)^{1/2}$. 
  For $\theta \in (0,1)$ 
  we can identify the interpolation space as
  \begin{align*}
    \left(X_0,X_1\right)_{\theta,2}= L^2(\omega;w_0^{1-\theta}w_1^{\theta}) 
    \qquad \mbox{ (equivalent norms).}
  \end{align*}
  The implied constants are explicitly known and depend only on $\theta$.
\end{proposition}
In order to formulate the analog of Proposition~\ref{prop:interpolation_weighted_l2} for discrete spaces, 
we need the following definition:
\begin{definition}[locally comparable]
\label{def:locally_comparable}
  We call a measurable function $w: \Gamma \to \R$ a {\em locally comparable weight}, if $w(x)>0$ (almost) everywhere,
  and if there exists a constant $C_w>0$ such that for all 
vertex patches $\omega_V$ the following estimate holds
  \begin{align}
    \label{eq:def:locally_comparable}
    \inf_{x\in \omega_V}{w(x)} &\leq \sup_{x\in \omega_V}{w(x)} \leq C_w \inf_{x\in \omega_V}{w(x)}.
  \end{align}
\end{definition}
\begin{theorem}
\label{thm:interpolation_poly_l2}
Let $\TT$ satisfy Assumption~\ref{assumption:element_maps}. 
  Let $w_0$, $w_1 \in L^{\infty}(\Gamma)$ be locally comparable weights. 
  Then we can identify the interpolation space by
  \begin{align}
    \left(\left(\SS^{\pp,1}(\mesh),\|\cdot\|_{L^2(\Gamma;w_0)}\right),\left(\SS^{\pp,1}(\mesh),\|\cdot\|_{L^2(\Gamma;w_1)}\right)\right)_{\theta,2}
    &=
    \left( \SS^{\pp,1}(\mesh), \norm{\cdot}_{L^2(\Gamma;w_0^{1-\theta} w_1^{\theta})} \right), \label{eq:interpolation_poly_l2_1}\\
    \left(\left(\SSnd^{\pp,1}(\mesh),\|\cdot\|_{L^2(\Gamma;w_0)}\right),\left(\SSnd^{\pp,1}(\mesh),\|\cdot\|_{L^2(\Gamma;w_1)}\right)\right)_{\theta,2}
    &=
    \left( \SSnd^{\pp,1}(\mesh), \norm{\cdot}_{L^2(\Gamma;w_0^{1-\theta} w_1^{\theta})} \right),  \label{eq:interpolation_poly_l2_2}
  \end{align}
  where the  norm equivalence constants depend only on $\Gamma$, $\Omega$, $\theta$,
  the shape regularity of $\TT$, and  $C_w$ from~\eqref{eq:def:locally_comparable}.
\end{theorem}

\begin{proof}
We will only show~\eqref{eq:interpolation_poly_l2_1}. Let $u \in \SS^{\pp,1}(\TT)$.
We denote by $\norm{\cdot}_{\theta}$ the interpolation norm of the space
$$
\left((\SS^{\pp,1}(\TT),\norm{\cdot}_{L^2(\Gamma, w_0)}),(\SS^{\pp,1}(\TT),\norm{\cdot}_{L^2(\Gamma, w_1)})\right)_{\theta,2}.
$$ 
The weighted $L^2$ norm will be denoted by 
$\norm{u}_{L^2_{\theta}}:=\norm{u}_{L^2(\Gamma, w_0^{1-\theta} w_1^{\theta})}$.
By Proposition~\ref{prop:interpolation_weighted_l2} and the definition of the $K$-functional as an infimum we have the trivial estimate
\begin{align*}
  \norm{u}_{\theta} &\gtrsim \norm{u}_{L^2_{\theta}}.
\end{align*}
It remains to show the converse estimate. We proceed similarly to the proof of \cite[Lemma 23.1]{tartar07}.

Let $u_1$,$(u_V)_{V \in \innervertices}$ denote the decomposition of Theorem~\ref{thm:decomposition-FEM-space}
(for simplicity of notation we assume that the edge and element contributions are included in the vertex functions).

By the local definition of the Scott-Zhang operator, it easy to see that 
it is also stable in weighted $L^2$ norms
(with stability constant depending additionally on the constant $C_w$ 
of Def.~\ref{def:locally_comparable}).
Since $u_1$ is constructed using this operator 
(see the Section~\ref{sec:decomposition-general-triangulations}), 
we have by interpolation
\begin{align*}
  \norm{u_1}_{\theta}&\lesssim \norm{u}_{L^2_\theta}.
\end{align*}

For $V \in \innervertices$ and $j=0$, $1$ set 
$(\overline{w_j})_{V}:=\frac{1}{2} \left(\inf_{x \in \omega_V}{w_j(x)} + \sup_{x \in \omega_V}{w_j(x)}\right)$.
Since  $w_0$, $w_1$ are locally comparable weights, we have 
on each patch $(\overline{w_{j}})_V \sim w_j$.
In order to estimate the infimum in the $K$-functional, we set:
\begin{align*}
  v(x)&:=\sum_{V \in \innervertices}{\frac{(\overline{w_0})_V}{(\overline{w_0})_V + t^2 (\overline{w_1})_V} u_V(x) }
  \quad \text{and}  \quad                
  u-u_1-v=\sum_{V \in \innervertices}{\frac{t^2 (\overline{w_1})_V}{(\overline{w_0})_V + t^2 (\overline{w_1})_{V}} u_V }. 
\end{align*}
Since the coefficients are independent of $x$ we have $v \in {\SS}^{\pp,1}(\TT)$ and thus $u = (u-v)+ v$ 
is an admissible decomposition for the infimum of \eqref{eq:def_K_functional}.
Using $(\overline{w_{j}})_V \sim w_j$ we estimate
\begin{align*}
  \abs{v(x)}&\lesssim \frac{ w_0(x)}{w_0(x) + t^2 w_1(x)} \sum_{V \in \innervertices}{\abs{u_V(x)}} \; \text{ and } \;
  \abs{u - u_1-v}\lesssim \frac{t^2 w_1(x)}{w_0(x) + t^2 w_1(x)} \sum_{V \in \innervertices}{\abs{u_V}}, 
\end{align*}
where the implied constant only depends on the constant in \eqref{eq:def:locally_comparable}.
A simple calculation then shows
\begin{align*}
  K^2(t;u-u_1) &\lesssim  \int_{\Gamma}{\frac{t^2 w_0(x)w_1(x)}{w_0(x)+t^2 w_1(x)} \Big(\!\sum_{V \in \innervertices}{\abs{u_V(x)}} \Big)^2 dx }.
\end{align*}
For the interpolation norm we then get by using Fubini's theorem and the substitution ${t=s\sqrt{\frac{w_0(x)}{w_1(x)}}}$:
\begin{align*}
  \norm{u-u_1}_{\theta}^2
  &\lesssim \int_{0}^{\infty}{\int_{\Gamma}{ t^{-2\theta}\frac{t^2 w_0(x)w_1(x)}{w_0(x)+t^2 w_1(x)} \big(\sum_{V \in \innervertices}{\abs{u_V(x)}} \big)^2 dx \frac{dt}{t}}} \\
  &\lesssim \int_{\Gamma}{ w_0(x)^{1-\theta} w_1(x)^{\theta} \big(\sum_{V \in \innervertices}{\abs{u_V(x)}} \big)^2  \int_{0}^{\infty}{\frac{s^{1-2\theta}}{1+s^2} \,ds} \,dx} \\
  &\lesssim \int_{\Gamma}{ w_0(x)^{1-\theta} w_1(x)^{\theta} \big(\sum_{V \in \innervertices}{\abs{u_V(x)}} \big)^2 \;dx}.
\end{align*}
Since each function $u_V$ is supported by the single patch $\omega_V$, and since the patches only have finite overlap, a simple coloring argument implies
\begin{align*}
  \norm{u-u_1}_{\theta}^2 & \lesssim \sum_{V \in \innervertices}{ \int_{\Gamma}{ w_0(x)^{1-\theta} w_1(x)^{\theta} \abs{u_V(x)}^2 \;dx}} 
  =\sum_{V \in \innervertices}{\norm{u_V}_{L^2_{\theta}}^2}.
\end{align*}

We observe that, since the decomposition $(u_V)_{V \in \innervertices}$ was constructed in a local manner
and $\omega_0^{1-\theta} \omega_1^{\theta}$ is locally comparable, we can exchange the
$L^2$-norm in the stability result of Theorem~\ref{thm:main_decomposition} by the $L^2_{\theta}$ norm, which then concludes the proof.
\end{proof}

\subsection{Additive Schwarz preconditioning for the $p$-BEM}
\label{sec:ams}
In this section we apply the decomposition results of
Theorems~\ref{thm:decomposition_coloring_estimate}, 
\ref{thm:main_decomposition}
to the $hp$-version of 
the boundary element method. 
Our model problem is the hypersingular integral
operator $D$ for the Laplacian; for a more detailed discussion of 
boundary integral operators and their discretizations, 
we refer the reader to the monographs \cite{ss,s,hsiaowendland,mclean}.
We note that the following setting covers both, the case of closed surfaces and screen problems.
  This is because in the case of closed surfaces we have 
  $\honehalftilde=H^{1/2}(\Gamma)$ and $\SS^{\pp,1}(\TT)=\mathcal{S}^{\pp,1}(\TT)$.

The hypersingular integral operator $D: \honehalftilde \to H^{-1/2}(\Gamma)$ is defined by
\begin{align*}
  ( D\,u ) \;(x):= -\partial_{n_x}^{int} \int_{\Gamma}{ \partial_{n_y}^{int} G(x,y)\,u(y) \;dS_y}, \quad \text{for $x \in \Gamma$},
\end{align*}
where $G(x,y):=\frac{1}{4\pi} \frac{1}{\abs{x-y}}$ is the fundamental solution of the 3D-Laplacian and $\partial^{int}_{n_x}$
denotes the (interior) normal derivative with respect to $x$.

In the case of a closed, connected surface, the kernel of $D$ consists of the constant
functions. In order to get a well-posed system it is customary to introduce,
for some chosen $\alpha > 0$, the stabilized form
\begin{align}
\label{eq:def_Dtilde}
  \dualproduct{\widetilde D u}{v}:= \dualproduct{Du}{v} + \alpha \dualproduct{u}{1}\dualproduct{v}{1} \quad \quad \forall u,v \in \honehalftilde,
\end{align}
where $\dualproduct{\cdot}{\cdot}$ denotes the extension of the standard $L^2$-inner product to $H^{-1/2}(\Gamma) \times \widetilde{H}^{1/2}(\Gamma)$.
This bilinear form is known to be bounded and elliptic, i.e., 
there exist some constants $c$, $C>0$ such that
$| \dualproduct{ \widetilde D u}{v} | \leq C \norm{u}_{\honehalftilde} \norm{v}_{\honehalftilde}$ and
 $\dualproduct{ \widetilde D u}{u} \geq c \norm{u}_{\honehalftilde}^2$ for all $u,v \in \honehalftilde$.
In the case of an open surface $\Gamma \neq \partial \Omega$ the kernel 
of $D$ is trivial and already $D$ is elliptic so that we may set $\alpha=0$.

The Galerkin matrix $\widetilde {\mathbf D}_{hp}$ corresponding to the 
Galerkin discretization of this bilinear form based on the 
space $\SS^{\pp,1}(\TT)$ with chosen basis $(\varphi_i)_i$ is given by 
$\displaystyle 
  \left(\widetilde{{\mathbf D}}_{hp}\right)_{ij}:=\dualproduct{\widetilde{D} \varphi_j}{\varphi_i}. 
$

We will present a preconditioner for this matrix based on the abstract additive Schwarz framework that will
allow for $hp$-independent bounds on the condition number of the preconditioned system. It is based on the
decomposition into the vertex, edge and element patch spaces, given in \eqref{eq:decomposition}.
We briefly recall some important definitions of the additive Schwarz theory. For a detailed introduction see \cite[Chapter 2]{book_toselli_widlund}.

Let $a(\cdot,\cdot): \fV \times \fV \to \R$ be a symmetric, positive definite bilinear form 
on a finite dimensional vector space $\fV$. We will write ${\mathbf A}$ for the 
corresponding Galerkin matrix. Let $\fV_i \subseteq \fV$, $i=0, \dots, N$, be a family of subspaces and let
$R_i^T: \fV_i \to \fV$ denote the canonical embedding operators 
(we will use the symbol ${\mathbf R}_i^T$ for their matrix representation). 
Assume that for each subspace $\fV_i$ a symmetric, positive definite bilinear form 
$\widetilde{a}_i(\cdot,\cdot)$ 
is given; its Galerkin matrix is denoted $\widetilde{{\mathbf A}}_i$.  
Assume the spaces $\fV_i$ form a decomposition of $\fV$, i.e., we can write
\begin{align*}
  \fV &= R_0^T \fV_0 + \sum_{i=1}^{N}{R_i^T \fV_i}.
\end{align*}
We then define the additive Schwarz preconditioner by:
\begin{align*}
  {\mathbf B}^{-1}:= \sum_{i=0}^{N}{{\mathbf R}_i^T \, \widetilde{{\mathbf A}}_{i}^{-1} {\mathbf R}_i}.
\end{align*}

The preconditioner induced by the decomposition defined in \eqref{eq:decomposition} is optimal in $h$ and $p$. This is formalized in the following theorem: 
\begin{theorem}
\label{thm:ams-hypersingular}
  Let $a(u,v)= \langle \widetilde D u,v\rangle_\Gamma$. 
  Let $\fV_0:=\SS^{1,1}(\TT)$ be the global lowest order space and 
  $R_0^T:\fV_0 \rightarrow \SS^{\pp,1}(\TT)$ be the canonical embedding. 
  For every $V \in \innervertices$ we define the space 
  $\fV_{V}:=\{u \in \SS^{\pp,1}(\TT): \operatorname{supp}(u) \subseteq \overline{\omega_V} \}$
  with canonical embedding $R_V^T:\fV_V \rightarrow \SS^{\pp,1}(\TT)$.  
  We use ``exact local solvers'', i.e., $\widetilde{a}_V(u,v):=\dualproduct{\widetilde{D} R^T_{V} u }{ R^T_{V} v}$ with matrix representation $\widetilde {\mathbf A}_V$.
  Let Assumption~\ref{ass:degree_distribution} hold.
  Then the preconditioner defined by
  \begin{align*}
    {\mathbf B}^{-1}:= {\mathbf R}_0^T {\mathbf A}_0^{-1} {\mathbf R}_0 
   + \sum_{V \in \innervertices}{{\mathbf R}_V^T \, \widetilde{{\mathbf A}}_{V}^{-1} {\mathbf R}_V}
  \end{align*}
  is optimal in the sense that there exists a constant $C>0$ that only depends only on $\Gamma$, $\Omega$, and the $\gamma$-shape regularity of $\mathcal{T}$ such that
  \begin{align*}
    \kappa({\mathbf B}^{-1} \widetilde{{\mathbf D}}_{hp}):=\frac{\lambda_{max}({\mathbf B}^{-1} \widetilde{{\mathbf D}}_{hp})}{\lambda_{\min}({\mathbf B}^{-1} \widetilde{{\mathbf D}}_{hp})} &\leq C .
  \end{align*}
\end{theorem}
  \begin{proof}
    The abstract additive Schwarz theory gives the condition number estimate 
    $\kappa \leq C_0 C_1$ in terms of the following two constants $C_0$, $C_1$ 
    (see \cite{zhang_multilevel_schwarz,lions_schwarz_alternating_1,nepomnyaschik_asm}):
    \begin{enumerate}[(i)]
    \item 
     \label{item:schwarz-i}
    The constant $C_0 >0 $ is such that every $u \in \fV$ admits a decomposition
       $ u = \sum_{i=0}^{J}{ R_i^T \, u_i}$ with $u_i \in \fV_i$
      such that $\sum_{i=0}^{J}{\widetilde{a}_i(u_i,u_i)}\leq C_0 \; a(u,u).$
    \item 
     \label{item:schwarz-ii}
     The constant $C_1 >0$ is such that for every decomposition
      $u=\sum_{i=0}^{J}{R_i^T \,v_i}$ with $v_i \in \fV_i$ the following estimate holds:
      $a(u,u) \leq C_1 \sum_{i=0}^{J}{\widetilde{a}_i(v_i,v_i)}$.
    \end{enumerate}
    Since $\widetilde{D}$ is continuous and elliptic, we can replace $a(\cdot,\cdot)$ with  the $\honehalftilde$-norm.
    The requirement (\ref{item:schwarz-i}) then corresponds to the stability statement 
    of Theorem~\ref{thm:decomposition-FEM-space} (where we absorbed the element and edge contributions into 
    the vertex parts)
    in conjunction with Remark~\ref{rem:norm-equivalence-extension}.
    The requirement (\ref{item:schwarz-ii}) is just an application of Theorem~\ref{thm:decomposition_coloring_estimate}. 
  \end{proof} 
We refer to \cite{fuehrer-melenk-praetorius-rieder15} for studies concerning the numerical performance 
of the preconditioner of Theorem~\ref{thm:ams-hypersingular}. 
\begin{remark}
The preconditioner of Theorem~\ref{thm:ams-hypersingular} requires a solver for 
  the space $\fV_0=\SS^{1,1}(\TT)$ of piecewise linears. This solver can be replaced with a multilevel
  method as advocated, e.g., in \cite{tran_stephan_asm_h_96,amcl03,fuehrer-melenk-praetorius-rieder15}. 
\eremk
\end{remark}
\section{Interpolation of Sobolev norms} 
\label{sec:interpolation-argument}
\subsection{Interpolation of Sobolev norms and seminorms}
\label{sect:interpol_seminorms}
As is well-known, the $L^2$-norm and the $H^1$-seminorm scale differently under affine changes of variables,
and the full $H^1$-norm does not have a natural scaling property. Thus, the effect of domain scalings on 
fractional Sobolev spaces, which are defined by interpolation between $L^2$ and $H^1$ is not very clear.
One way to clarify the impact of domain scalings is to study the interpolation of seminorms.
The following lemma works out the corresponding norm equivalences. 
\begin{lemma}[interpolation of seminorms]
\label{lemma:K-vs-k}
Let $X_1 \subseteq X_0$ be two Banach spaces with norms 
$\|\cdot\|_0$ and $\|\cdot\|_1 := H^{-1} \|\cdot\|_0 + |\cdot |_1$, where $|\cdot|_1$ is a seminorm
and $H > 0$. Introduce the following two $K$-functionals: 
$$
K^{2}(u,t):= \inf_{v \in X_1} \|u - v\|_{0}^2 + t^2 \|v\|_1^2, 
\qquad k^{2}(u,t):= \inf_{v \in X_1} \|u - v\|_{0}^2 + t^ 2|v|_1^2.
$$
For $\theta \in (0,1)$ introduce the seminorm $|\cdot|_\theta$ and the norms 
$\|\cdot\|_\theta$ and $\|\cdot\|_{\tilde\theta}$ by 
\begin{eqnarray*}
|u|^2_\theta &=& \int_{t=0}^\infty t^{-2\theta} k^2(u,t) \frac{dt}{t}, \\
\|u\|^2_\theta &=& \int_{t=0}^\infty t^{-2\theta} K^2(u,t) \frac{dt}{t}, \\
\|u\|^2_{\tilde \theta} &=& H^{-2\theta}  \|u\|^2_0 + |u|^2_\theta. 
\end{eqnarray*}
Then there exists $C > 0$, which depends solely on $\theta$ (in particular, it is independent of $H$), 
such that 
$$
C^{-1} \|u\|_{\theta} \leq \|u\|_{\tilde \theta} \leq C \|u\|_{\theta}. 
$$ 
\end{lemma}
\begin{proof} 
We show $\|u\|_{\tilde \theta} \leq C \|u\|_\theta$: Obviously, $|u|_\theta \leq \|u\|_{\theta}$. 
In order to see $H^{-\theta} \|u\|_{0} \leq C \|u\|_{\theta}$, we observe that for arbitrary $v\in X_1$
$$
\|u\|_{0} \leq \|u - v\|_{0} + \|v\|_{0} \leq \|u - v\|_{0} + \frac{H}{t} t \|v\|_{1},
$$
and hence $\|u\|_{0} \leq \max\{1,H/t\}  K(u,t)$. This implies 
\begin{eqnarray*}
H^{-2\theta} \left(\frac{1}{2\theta} + \frac{1}{2(1-\theta)}\right) \|u\|^2_0 
 = \int_{t=0}^\infty t^{-2\theta-1} \min\{1,t/H\}^2 \|u\|^2_0\,dt \leq \int_{t=0}^\infty t^{-2\theta-1} K^2(u,t)\,dt
 = \|u\|^2_\theta. 
\end{eqnarray*}
Next, we show $\|u\|_{\theta} \leq C \|u\|_{\tilde \theta}$. We write 
\begin{align}\label{lemma:K-vs-k:eq1}
  \|u\|^2_\theta = \int_{t=0}^H t^{-2\theta - 1} K^2(u,t)\,dt + \int_{t=H}^\infty t^{-2\theta -1} K^2(u,t)\,dt .
\end{align}
To treat the first integral, we let $\tilde v$ be the minimizer of $k^{2}(u,t)$, i.e., 
$\tilde v = \operatorname*{argmin} \inf_{v \in X_1} \|u - v\|^{2}_{0} + t^{2} |v|_{1}^{2}$ 
(If the minimum is not attained, then select $v_\varepsilon$ with 
$\|u - v_\varepsilon\|_{0}^{2} + t^{2} |v_\varepsilon|_{1}^{2} \leq k^{2}(u,t) + \varepsilon$ and let $\varepsilon \rightarrow 0$ 
at the end.)
Then 
$$
K^{2}(u,t) \leq \|u - \tilde v\|^{2}_{0} + t^{2} \|\tilde v\|_{1}^{2} \leq 
\|u - \tilde v\|^{2}_{0} + \frac{t^{2}}{H^{2}} \|\tilde v\|^{2}_{0} + t^{2}|\tilde v|_{1}^{2}
\lesssim \max\left\{1,\frac{t^{2}}{H^{2}}\right\} k^{2}(u,t) + \frac{t^{2}}{H^{2}} \|u\|^{2}_{0}. 
$$
Therefore, 
\begin{align*}
  \int_{t=0}^H t^{-2\theta-1} K^2(u,t)\,dt \lesssim \int_{t=0}^H t^{-2\theta-1} k^2(u,t)\,dt + 
  H^{-2} \int_{t=0}^H t^{-2\theta+1} \|u\|^2_{0}\,dt
\lesssim |u|^2_\theta + H^{-2\theta} \|u\|^2_{0}. 
\end{align*}
To treat the second integral in~\eqref{lemma:K-vs-k:eq1}, we use the obvious estimate $K(u,t)\leq \|u\|_{0}$ and see
\begin{align*}
  \int_{t=H}^\infty t^{-2\theta-1} K^2(u,t)\,dt \leq \int_{t=H}^\infty t^{-2\theta-1} \|u\|^2_{0}\,dt
  \leq C H^{-2\theta} \|u\|^2_{0}. 
\end{align*}
Therefore, $\|u\|^2_\theta \lesssim H^{-2\theta}\|u\|^2_{0} + |u|^2_\theta$ with implied constants 
depending only on $\theta$. 
\end{proof}

Lemma~\ref{lemma:K-vs-k}
 gives the following norm equivalence for the standard fractional Sobolev norms and
the weighted versions:
\begin{align}
\label{eq:norm-equivalence}
  \abs{u}_{\H^\theta(\omega)} + \diam(\omega)^{-\theta} \norm{u}_{L^2(\omega)} \sim \norm{u}_{\H^{\theta}_h(\omega)}  \quad
  \text{and}\quad
  \abs{u}_{H^\theta(\omega)} + \diam(\omega)^{-\theta} \norm{u}_{L^2(\omega)} \sim \norm{u}_{H^{\theta}_h(\omega)}.
\end{align}

Concerning the norms on patches, we note the following: 
\begin{remark}
\label{rem:norm-equivalence-extension}
Since for any vertex or edge patch $\omega \subset \Gamma$ 
the operator realizing the extension by zero is bounded with constant $1$ as an operator 
$L^2(\omega) \rightarrow L^2(\Gamma)$ and 
$\widetilde {H}^1_h(\omega) \rightarrow \widetilde{H}^1(\Gamma)$ we obtain by interpolation 
\begin{equation} 
\|u\|_{\widetilde{H}^\theta(\Gamma)} \lesssim
\|u\|_{\widetilde{H}_h^\theta(\omega)} \stackrel{(\ref{eq:norm-equivalence})}{\sim}
|u|_{\widetilde{H}_h^\theta(\omega)} + \operatorname{diam}(\omega)^{-\theta} \|u\|_{L^2(\omega)} 
\qquad \forall u \in \widetilde{H}^\theta(\omega).
\end{equation}
\eremk
\end{remark}

Next, we study the scaling of interpolation seminorms. The following corollary can be seen
as a generalization of~\cite[Lemma 2.8]{heuer_sobolev_norms} and provides the
  application of Lemma~\ref{lemma:K-vs-k} to the present setting of 2D surface elements.

\begin{corollary} \label{cor:reference-patch-theta-equivalence}
  Let $K\in\TT$ be an element and $\Kref$ be the associated reference element.
  Suppose that $X_1\subset L^2(K)$ and $\widehat X_1\subset L^2(\Kref)$
  are continuously embedded Banach spaces such that for every $u\in X_1$ and its pull-back 
  $\widehat u:= u \circ F_K$ to the reference element $\widehat K$ there holds 
$\widehat u \in \widehat X_1$. 
  Let $|\cdot|_{X_1}$ and $|\cdot|_{\widehat X_1}$ be seminorms (or norms)
  on $X_1$ with $| u |_{X_1} \sim | \widehat u |_{\widehat X_1}$.  
  Define $|\cdot|_{\theta,K}$ and $|\cdot|_{\theta,\Kref}$
  as in Lemma~\ref{lemma:K-vs-k}, where $X_0 = L^2(K)$ 
  and $\widehat{X}_0 = L^2(\Kref)$.
Then it holds for $\theta\in[0,1]$
  \begin{align*}
    |u|_{\theta,K} \sim h_K^{1-\theta} |\widehat u|_{\theta,\Kref}. 
  \end{align*}
  In particular, we have 
  \begin{align*}
    \| u \|_{(L^2(K),X_1)_{\theta,2}} \sim h_K^{1-\theta} \| \widehat u \|_{(L^2(\Kref),\widehat X_1)_{\theta,2}}.
  \end{align*}
\end{corollary}
\begin{proof} 
  Write $k^2(u,t) = \inf_{v \in X_1} \|u - v\|^2_{L^2(K)} + t^2|v|_{X_1}^2$. 
  Then
  \begin{align*}
    \|\widehat u - \widehat v\|_{L^2(\Kref)}^{2} + t^{2}|\widehat v|_{\widehat X_1}^{2}
    \sim 
    h_K^{-2} \|u - v\|^{2}_{L^2(K)} + t^{2} |v|_{X_1}^{2} . 
  \end{align*}
  Therefore,
  \begin{align*}
    \inf_{w \in \widehat{X}_1}
    \|\widehat u - w\|_{L^2(\Kref)}^{2} + t^{2}|w|_{\widehat X_1}^{2}
    \lesssim h_K^{-2} k^{2}(u,t  h_K), 
  \end{align*}
  and thus 
  \begin{align*}
    |\widehat u|_{\theta,\Kref}^2
    \lesssim h_K^{-2} \int_{t=0}^\infty t^{-2\theta} k^2(u,t  h_K)\frac{dt}{t} 
    \lesssim h_K^{-2+2\theta} |u|^2_{\theta,K}. 
  \end{align*}
  The reverse direction is shown in a similar manner. The equivalence of norms then follows with 
  Lemma~\ref{lemma:K-vs-k}.
\end{proof}
\subsection{Interpolation of weighted spaces}
\label{sec:interpolation-of-weighted-spaces}
The vertex averaging operators defined in Section~\ref{sec:averaging-operators} ahead are such that 
they reproduce the value in one corner of the reference element $\Kref$ for functions that are continuous
there. For functions with less regularity near that corner, the difference is measured in a 
weighted $L^2$-norm. Thus, weighted spaces and interpolation between weighted spaces are studied
in the present section. 
\subsubsection{Trace theorems and local estimates}
We recall some simple trace estimates.
\begin{lemma}\label{lemma:weighted-norm-equivalence}
  Let $\widehat{V}$ be a vertex of $\Tref$ and $\widehat{e}$ an edge with $\widehat{V} \in \overline{\widehat{e}}$.
    \begin{enumerate}[(i)]      
    \item\label{item:lemma:weighted-norm-equivalence-i}
      Then for all $v \in H^1(\Tref)$ the following estimates hold      
      provided that the right-hand sides are finite:  
      \begin{eqnarray*}
        \|d_{\widehat{V}}^{-1} v\|_{L^2(\Tref)} &\lesssim  & 
        \|\nabla v\|_{L^2(\Tref)} + \|d_{\widehat{V}}^{-1/2} v\|_{L^2(\widehat{e})},  \\
        \|d_{\widehat{V}}^{-1/2} v\|_{L^2(\widehat{e})} &\lesssim& \|d_{\widehat{V}}^{-1} v\|_{L^2(\Tref)} + \|\nabla v\|_{L^2(\Tref)}. 
      \end{eqnarray*}
    \item\label{item:lemma:weighted-norm-equivalence-ii}
      Define, for $t > 0$, the slabs $\widetilde{A}^{{2,1}}(t):=\{ x \in \Tref: t<d_{\widehat{V}}(x) < 2t\}$.
      Then for all $v \in H^1(\Tref)$ the following estimates hold
      provided that the right-hand sides are finite:
      \begin{eqnarray*}
        \|d_{\widehat{V}}^{-1} v\|_{L^2(\widetilde{A}^{2,1}(t))} &\lesssim & 
        \|\nabla v\|_{L^2(\widetilde{A}^{{2,1}}(t))} + \|d_{\widehat{V}}^{-1/2} v\|_{L^2(\widehat{e}\, \cap \widetilde{A}^{{2,1}}(t))},  \\
        \|d_{\widehat{V}}^{-1/2} v\|_{L^2(\widehat{e}\, \cap \widetilde{A}^{{2,1}}(t))} &\lesssim& \|d_{\widehat{V}}^{-1} v\|_{L^2(\widetilde{A}^{{2,1}}(t))}
        + \|\nabla v\|_{L^2(\widetilde{A}^{{2,1}}(t))}. 
      \end{eqnarray*}
    \item\label{item:lemma:weighted-norm-equivalence-edge}
       For $v \in H^1(\Tref)$ with $v|_{\widehat{e}} \equiv 0$, it holds
      \begin{align}
          \norm{d_{\widehat{e}}^{-1} u}_{L^2(\Tref)} &\leq \norm{\nabla v}_{L^2(\Tref)}.
          \label{eq:lemma:weighted-norm-equivalence-edge}
      \end{align}
  \end{enumerate}
\end{lemma}
\begin{proof}
  The first estimate in~(\ref{item:lemma:weighted-norm-equivalence-i}) follows from standard arguments
  and an application of Hardy's inequality: To keep the notation succinct,
  consider $\widehat{V}:=(0,0)$ and $\widehat{e}:= (0,1) \times \{0\}$; note that $d_{V} \sim \xi$. 
  Then $u(\xi,\eta) = u(\xi,0) + \int_{t=0}^\eta \partial_t u(\xi,t)\,dt$ and therefore 
  $$
    \int_{\widehat T} \xi^{-2} u^2(\xi,\eta) \lesssim \int_{\xi} \xi^{-1} u(\xi,0)^2\,d\xi + 
    \int_{\xi} \int_{\eta=0}^\xi \left| \xi^{-1} \int_{t=0}^\eta u_t(\xi,t)\,dt\right|^2\,d\eta\,d\xi.
  $$
  In the last integral, we estimate $\xi^{-1} \leq \eta^{-1}$ and apply Hardy's inequality. 
  The second estimate in~(\ref{item:lemma:weighted-norm-equivalence-i})
  follows from local trace estimates near $\widehat{e}$ and a covering argument (Besicovitch).
  The estimates in (\ref{item:lemma:weighted-norm-equivalence-ii})
  follow in a similar manner using polar coordinates; alternatively, it can be
  shown by scaling arguments.
  The estimate~(\ref{item:lemma:weighted-norm-equivalence-edge}) follows along the same lines,
  using that $u(\xi,0)$ vanishes.
\end{proof}
\begin{lemma}\label{lemma:weighted-infty}
 Let $u\in L^\infty(\Tref)$ and let $\widehat{V}$ be one of the vertices of $\Tref$.
  Then,
$\displaystyle 
    \| d_{\widehat{V}}^{-1/2} u \|_{L^2(\Tref)} \lesssim \| u \|_{L^\infty(\Tref)}.
$
\end{lemma}
\begin{proof}
  For simplicity, we consider $\widehat{V}=(0,0)$ and note that $d_{\widehat{V}}\sim \xi$. We compute
  \begin{align*}
    \| d_{\widehat{V}}^{-1/2} u \|_{L^2(\Tref)}^2 &\lesssim \| u \|_{L^\infty(\Tref)}^2 \int_0^1\xi^{-1}\int_0^\xi 1\,d\eta d\xi
    = \| u \|_{L^\infty(\Tref)}^2.  \qedhere
  \end{align*}
  
\end{proof}
\subsubsection{Interpolation of weighted spaces}
We start with an explicit construction of a function that essentially
realizes the decomposition of the K-functional for a pair of Sobolev spaces: 
\begin{lemma}\label{lemma:sobolev:decomposition}
  Let $\omega\subset\R^d$ be a bounded Lipschitz domain with $\diam(\omega)\leq C_1$, $\theta\in(0,1)$,
  and fix $\beta > 0$.
  Then, for $u\in H^\theta_h(\omega)$, there is a function $w:(0,\infty)\rightarrow H^1_h(\omega)$
  such that
  \begin{align}\label{cor:sobolev:eq1}
    \int_{0}^\infty t^{-2\theta}
    \left( \| u - w(t) \|_{L^2(\omega)}^2 + t^2 \| w(t) \|_{H^1_h(\omega)}^2 \right)
    \frac{dt}{t}
    \lesssim \| u \|^2_{H^\theta_h(\omega)}.
  \end{align}
  Additionally, for all subsets $\omega'\subset\omega$ with
  $\dist{\omega'}{\partial\omega}>\beta t$ it holds
  \begin{align}\label{cor:sobolev:eq2}
    \| w(t) \|_{L^2(\omega')} &\lesssim \| u \|_{L^2(\cup_{x\in\omega'}B_{\beta t/2}(x))}.
  \end{align}
  The hidden constants depend only on $d$, $C_1$, $\beta$, and the Lipschitz constant of $\omega$.\\

  If, additionally, $\widetilde\omega\subset\omega$ is a subset with
  $\diam(\omega)\leq C_{\diam}\diam(\widetilde\omega)$ for some constant $C_{\diam}>0$,
  and $\supp(u)\subset\widetilde\omega$,
  then $w:(0,\infty)\rightarrow \widetilde H^1_h(\omega)$ and 
  the right-hand side in~\eqref{cor:sobolev:eq1} can be replaced by
  $\| u \|_{\widetilde H^\theta_h(\widetilde\omega)}$;
  furthermore, \eqref{cor:sobolev:eq2} holds for all $\omega'\subset\R^d$ with $u$ on the right implicitly
  extended by zero. The hidden constants depend additionally on $C_{\diam}$.
\end{lemma}
\begin{proof}
  Denote by $H^1_h(\R^d)$ the space $H^1(\R^d)$ with norm
  \begin{align*}
    \| \cdot \|_{H^1_h(\R^d)}^2 := \diam(\omega)^{-2}\| \cdot \|_{L^2(\R^d)}^2 + | \cdot |_{H^1(\R^d)}^2,
  \end{align*}
  and by $H^\theta_h(\R^d)$ the interpolation between $L^2(\R^d)$ and $H^1_h(\R^d)$.
  From $\diam(\omega)\leq C_1$ we conclude that Stein's extension operator $E$ (see~\cite[Sec.~{VI}]{stein70}) is bounded
  as $L^2(\omega)\rightarrow L^2(\R^d)$ and $H^1_h(\omega)\rightarrow H^1_h(\R^d)$,
  with a constant depending only on $d$, $C_1$, and the Lipschitz constant of $\omega$.
  Therefore, it suffices to show the existence of  a function $w:(0,\infty)\rightarrow H^1_h(\R^d)$
  such that
  \begin{align}\label{cor:sobolev:eq3}
    \int_{0}^\infty t^{-2\theta}
    \left( \| Eu - w(t) \|_{L^2(\R^d)}^2 + t^2 \| w(t) \|_{H^1_h(\R^d)}^2 \right)
    \frac{dt}{t}
    \lesssim \| Eu \|^2_{H^\theta_h(\R^d)}.
  \end{align}
  To that end, let $\rho\in C^\infty(\R^d)$ with $\supp \rho\subset B_1(0)\setminus B_{1/2}(0)$
  be a mollifier and $\chi_{I}$ the indicator function of an interval $I$.
  Using the techniques from the proof of~\cite[Theorem 7.47]{adams_fournier} (implications $(a) \Rightarrow (b)$
    and $(c) \Rightarrow (a)$), it follows that
  $w(t) := \chi_{[0,\diam(\omega)]}(t)\cdot (\rho_{\beta t/2} \star E u)$ fulfills~\eqref{cor:sobolev:eq3}
  where we used again $\diam(\omega)\leq C_1$. The estimate~\eqref{cor:sobolev:eq2} is clear.
  If $\widetilde\omega\subset\omega$ is a subset
  as indicated and $u\in\widetilde H^\theta_h(\widetilde\omega)$, then we can extend
  $u$ by zero to $\R^d$ instead of using the Stein extension.
\end{proof}
The following lemma identifies the interpolation space between $L^2$ and 
a weighted $H^1$-space: 
\begin{lemma}
\label{lemma:equivalence-sobolev-plus-weight}
  Let $\theta \in (0,1)$, $\widehat{K}\in\{ \widehat T,\widehat S \}$ be the reference triangle or square,
  and let $\widehat{V}$ be one of its vertices. Consider the interpolation space between $L^2(\widehat{K})$ with the standard norm and
  $H^1(\widehat{K},d_{\widehat{V}}^{-1}) :={\{ u \in H^{1}(\widehat{K}): \norm{d_{\widehat{V}}^{-1} u}_{L^2(\widehat{K})} < \infty \}}$
  with the norm $\| u \|_{H^1(\widehat{K},d_{\widehat{V}}^{-1})}^2:= \|u\|_{H^1(\widehat{K})}^2 + \| d_{\widehat{V}}^{-1} u \|_{L^2(\widehat{K})}^2$.
  Then,
  \begin{align*}
    \norm{u}_{\left(L^2(\widehat{K}), H^{1}(\widehat{K},d_{\widehat{V}}^{-1})\right)_{\theta,2}}
    &\sim \norm{u}_{H^\theta(\widehat{K})} + \| d_{\widehat{V}}^{-\theta} u \|_{L^2(\widehat{K})},
  \end{align*}  
  and the implied constants depend only on $\theta$.
\end{lemma}
\begin{proof}
  For notational simplicity, we consider the case $\widehat{K}=(0,1)^2$ and $\widehat{V}=(0,0)$;
  the case of triangles can be inferred from this one by reflection across an edge of the triangle.
  We first show the bound $\gtrsim$ in the desired equivalence. 
  Note that for $t>0$ we have 
  \begin{align*}
    &\inf_{v \in H^1(\widehat{K},d_{\widehat{V}}^{-1})}{\left(\norm{u - v}_{L^2(\widehat{K})}^2 + t^2 \norm{v}_{H^1(\widehat{K})}^2 + t^2 \norm{d^{-1}_{\widehat{V}} v}^2_{L^2(\widehat{K})}\right)}
    \gtrsim K^2_{1}(u,t) + K^2_{2}(u,t),
  \end{align*}
  where $K_1(u,t)$ is the $K$-functional for the interpolation pair $(L^2(\widehat{K}),H^1(\widehat{K}))$
  and $K_2$ is the functional for $(L^2(\widehat{K}),L^2(\widehat{K},d^{-1}_{\widehat{V}}))$.
  By Proposition~\ref{prop:tartar} we have 
  $(L^2(\widehat{K}),L^2(\widehat{K},d^{-1}_{\widehat{V}}))_{\theta,2}=L^2(\widehat{K},d^{-\theta}_{\widehat{V}})$
  and obtain the bound $\gtrsim$.
  For the reverse bound, consider $u \in H^{\theta}(\widehat{K})$ with 
  $\| d_{\widehat{V}}^{-\theta} u \|_{L^2(\widehat{K})}<\infty$.
  We extend $u$ to $\omega:=(-1,1)^2$ in two steps by extending symmetrically first across the $y$ and 
then across the $x$-axis. The extended function is again denoted by $u$.
  Then, we choose $w:(0,\infty)\rightarrow H^1(\omega)$ according to Lemma~\ref{lemma:sobolev:decomposition} using $\beta:=1/2$
  and get for $t<1/2$ 
  \begin{align}\label{lem:espw:eq1}
    \| w(t) \|_{L^2(B_t(0)\cap\widehat K)} \leq \| w(t) \|_{L^2(B_t(0))}
    \lesssim \| u \|_{L^2(B_{2t}(0))} \lesssim \| u \|_{L^2(B_{2t}(0)\cap\widehat K)},
  \end{align}
  and for $t\geq 1/2$ 
  \begin{align}\label{lem:espw:eq2}
    \| w(t) \|_{L^2(B_t(0)\cap\widehat K)} \lesssim \| Eu \|_{L^2(\R^2)} \lesssim \| u \|_{L^2(\widehat K)}.
  \end{align}
  For $t > 0$, let $\chi_t\in C_0^\infty(\R^d)$ denote a smooth cutoff function with
  $\chi_t(x) = 1$ for all $x \in B_{t/2}(0)$, $\support{\chi_t} \subseteq B_t(0)$, and
  $\norm{\nabla \chi_t}_{L^{\infty}(\R^d)} \leq C t^{-1}$
  and define $\widetilde w(t):=(1-\chi_t)w(t)$. We calculate
  \begin{align*}
    \norm{u-\widetilde{w}(t)}_{L^2(\widehat{K})}&\leq
    \norm{u-w(t)}_{L^2(\widehat{K})} + \norm{\chi_t w(t)}_{L^2(\widehat{K})}
    \leq\norm{u-w(t)}_{L^2(\widehat{K})} + \norm{w(t)}_{L^2(B_t(0)\cap\widehat{K})}, \\
    t\norm{\widetilde w(t)}_{H^1(\widehat{K})}
    &\leq t\norm{w(t)}_{H^1(\widehat{K})} + t\norm{\chi_t w(t)}_{H^1(\widehat{K})}
    \lesssim t\norm{w(t)}_{H^1(\widehat{K})} + \norm{ w(t)}_{L^2(B_t(0)\cap\widehat{K})},
  \end{align*}
  and, with the notation $B^c_r(0) = {\mathbb R}^2 \setminus B_r(0)$, 
  \begin{align*}
    t \| d_{\widehat{V}}^{-1} \widetilde w(t)\|_{L^2(\widehat{K})} 
    &\lesssim t \| d_{\widehat{V}}^{-1} w(t) \|_{L^2(B^{c}_{t/2}(0)\cap\widehat K)} 
    \lesssim t \| d_{\widehat{V}}^{-1} \left(u - w(t)\right) \|_{L^2(B^c_{t/2}(0)\cap\widehat K)}
    + t \| d_{\widehat{V}}^{-1} u \|_{L^2(B^{c}_{t/2}(0)\cap\widehat K)} \\
    &\lesssim \norm{ u - w(t) }_{L^2(\widehat{K})}  + t \| d_{\widehat{V}}^{-1} u\|_{L^2(B^{c}_{t/2}(0)\cap\widehat K)}.
  \end{align*}
  We obtain
  \begin{align}
    \norm{u}^2\!\!\!&_{(L^2(\widehat{K}), H^{1}(\widehat{K},d_{\widehat{V}}^{-1}))_{\theta,2}}  \nonumber \\
    &\lesssim\begin{multlined}[t][11.75cm]
      \int_{0}^{\infty}{ t^{-2\theta}
        \left(\norm{u-w(t)}^2_{L^2(\widehat{K})} + t^2 \norm{w(t)}^2_{H^1(\widehat{K})}\right) \frac{dt}{t}} \\
      + \int_{0}^{\infty}{t^{-2\theta}  \norm{w(t)}^2_{L^2(B_{t}(0)\cap\widehat K)} \frac{dt}{t}}      
      + \int_{0}^{\infty}{t^{2-2\theta}  \| d_{\widehat{V}}^{-1} u \|^2_{L^2(B_{t/2}^c(0)\cap\widehat K)} \frac{dt}{t}}
    \end{multlined} \nonumber \\
    &\lesssim \norm{u}^2_{H^{\theta}(\omega)} + \int_{0}^{\infty}{t^{-2\theta}
        \norm{w(t)}^2_{L^2(B_{t}(0)\cap\widehat K)} \frac{dt}{t}} +
      \int_{0}^{\infty}{t^{2-2\theta}  \| d_{\widehat{V}}^{-1} u\|^2_{L^2(B_{t/2}^c(0)\cap\widehat K)} \frac{dt}{t}},
      \label{lem:espw:eq3}      
  \end{align}
  where we have used Lemma~\ref{lemma:sobolev:decomposition} in the last step.
  Since the symmetric extension is continuous in $L^2$ and $H^1$, 
  an interpolation argument shows that
  $\| u\|_{H^\theta(\omega)}\lesssim \| u \|_{H^\theta(\widehat K)}$,
  and it remains to bound the last two integrals by $\| d^{-\theta}_{\widehat{V}} u \|_{L^2(\widehat{K})}$.
  For the first term, we use~\eqref{lem:espw:eq2} to bound
  \begin{align*}
    \int_{1/2}^{\infty}{t^{-2\theta} \| w(t) \|^2_{L^2(B_{t}(0)\cap\widehat K)} \frac{dt}{t}}
    \lesssim \norm{u}^2_{L^2(\widehat K)}.
  \end{align*}
  We use \eqref{lem:espw:eq1} and polar coordinates to bound
  \begin{align*}
    \int_{0}^{1/2}{t^{-2\theta}  \| w(t) \|^2_{L^2(B_{t}(0)\cap\widehat K)} \frac{dt}{t}}
    &\lesssim \int_{0}^{1/2}{t^{-2\theta}  \| u \|^2_{L^2(B_{2t}(0)\cap\widehat K)} \frac{dt}{t}}
=\int_{0}^{1/2}{t^{-2\theta-1 } \int_{0}^{2t} \int_0^{2\pi} \chi_{\widehat K}(r,\varphi)u^2(r,\varphi) r\, d\varphi dr dt}\\
    &\leq\int_{0}^{2\pi}\int_{0}^{1} \chi_{\widehat K}(r,\varphi)u^2(r,\varphi) r \int_{r}^{\infty} t^{-2\theta -1}\,dtdrd\varphi \\
    &\lesssim\int_0^{2\pi}{\int_{0}^{1}{ \chi_{\widehat K}(r,\varphi)u^2(r,\varphi) r^{1-2\theta} dr \, d\varphi}}
    \leq \| d_{\widehat{V}}^{-\theta} u \|_{L^2(\widehat K)}^2.
  \end{align*}
  Finally, the second term in~\eqref{lem:espw:eq3} can be estimated by
  \begin{align*}
    \int_{0}^{\infty}{t^{2-2\theta} \| d_{\widehat{V}}^{-1} u\| ^2_{L^2(B_{t/2}^c(0)\cap\widehat K)} \frac{dt}{t}}
    &\leq\int_{0}^{\infty}{ \int_{\widehat{K} \cap \{\abs{x} > t/2\}} {t^{1-2\theta} \abs{x}^{-2} u^2(x)} dt}  
=\int_{\widehat{K}} { \abs{x}^{-2} u^2(x)} \int_{0}^{2\abs{x}}{ t^{1-2\theta}  dt}  \\
    &\lesssim\int_{\widehat{K}} { \abs{x}^{-2} u^2(x) \abs{x}^{2-2\theta}}
    = \| d_{\widehat{V}}^{-\theta} u \|_{L^2(\widehat{K})}^2.
\qedhere
  \end{align*}
\end{proof}

The following lemma encapsulates a construction needed later on in the proof of 
Lemma~\ref{lemma:weight-localize}. 
  \begin{lemma}
    \label{lemma:constructing_the_cone}
  Let $V \in \allvertices$ and let $\omega_V$ denote its vertex patch. Fix one element $K \subseteq \omega_V$ and define the annuli
  \begin{align*} A^{b,a}(t):=B_{b t}(V) \setminus \overline{B_{ a t}(V)} \qquad \text{ for } t>0 \text{ and } 0<a<b.\end{align*}
  Then there exists an open set $Z \subseteq K$ and a constant $\beta >0 $ such that
  \begin{align}
    \label{eq:proof_weighted_norm_est_on_patch_4_inclusion}
    \bigcup_{z \in {Z}} {\overline{B_{\beta d_{V}(z)}(z)}} \cap \Gamma
    \subseteq K
  \end{align}
  and for all $v \in H^1_h(\omega_V)$ there holds,  
  provided that the right-hand side is finite,
  \begin{align}
      \label{eq:proof_weighted_norm_est_on_patch_4}
      \| d_{V}^{-1} v \|_{L^2(A^{2,1}(t) \cap \omega_V)} \lesssim
      \| d_{V}^{-1} v \|_{L^2(A^{2,1}(t) \cap Z)}
      +  \| v \|_{H^{1}_h(A^{2,1}(t) \cap \omega_V)}. 
  \end{align}
\end{lemma}
\begin{proof}
  We first show
  \begin{align}
     \label{eq:proof_weighted_norm_est_on_patch_1}
    \norm{d_{V}^{-1} v}_{L^2(A^{2,1}(t) \cap \omega_V)} \lesssim \norm{d_{V}^{-1} v}_{L^2(A^{2,1}(t) \cap K)}
    +  \norm{v}_{H^{1}_h(A^{2,1}(t) \cap \omega_V)}.
  \end{align}
  Let $K_1$ and $K_2$ be two elements of $\omega_V$ sharing an edge $e_{1,2}$.
  Applying both estimates from
  Lemma~\ref{lemma:weighted-norm-equivalence}, (\ref{item:lemma:weighted-norm-equivalence-ii}) and scaling arguments
  shows
  \begin{align*}
    \|d_V^{-1}v\|_{L^2(A^{2,1}(t) \cap K_1)} &\lesssim
    \|d_V^{-1/2} v\|_{L^2(A^{2,1}(t) \cap  e_{1,2})} + \|\nabla v\|_{L^2(A^{2,1}(t) \cap K_1)}\\
    &\lesssim
    \|d_V^{-1} v\|_{L^2(A^{2,1}(t) \cap K_2)} +
    \|\nabla v\|_{L^2(A^{2,1}(t) \cap K_2)} + 
    \|\nabla v\|_{L^2(A^{2,1}(t) \cap K_1)}.
  \end{align*}
  Taking a sequence of neighboring elements $K'=K_1,\ldots,\widehat K_j =  K$,
  this argument can be repeated $j$ times and summed up to yield~\eqref{eq:proof_weighted_norm_est_on_patch_1}.

  We now construct the set $Z$.
  Let $\widehat{Z}$ be a cone on the reference element centered at $\widehat{V}:=(1,0)$,
  symmetric with respect to  the diagonal of $\widehat{S}$ with an opening angle in $(0,\pi/4)$.
  We note the existence of a constant $\widehat{\beta} >0 $
  such that for $\widehat{x} \in \widehat{Z}$ with $|\widehat{x}-\widehat{V}| \leq 1/4$,
  we have $\overline{B_{\beta d_{\widehat{V}}(\widehat{x})}(\widehat{x})}\subseteq \widehat{K}.$
  We set $Z:=F_{K}\big(\widehat{Z} \cap B_{1/4}(\widehat{V})\big)$ (where we assume that $\widehat{V}$ is mapped to ${V}$).
  To show~\eqref{eq:proof_weighted_norm_est_on_patch_4_inclusion}, we just have to select $\beta\leq \widehat{\beta}$ sufficiently small
  to compensate for the Lipschitz constants of $F_K$ and $F_K^{-1}$.

  The estimate~\eqref{eq:proof_weighted_norm_est_on_patch_4}
  can then be shown completely analogously to~\eqref{eq:proof_weighted_norm_est_on_patch_1}:
  On the reference element, we can use Lemma~\ref{lemma:weighted-norm-equivalence}, (\ref{item:lemma:weighted-norm-equivalence-ii})
  to reduce the estimate to an $H^1$ contribution and a $d_{\widehat{V}}^{-1/2}$-weighted integral on the boundary of $\widehat{Z}$. This can then in
  turn again be estimated by a weighted $L^2$-term and an $H^1$-term on $\widehat{Z}$, as appears on
  the right-hand side of~\eqref{eq:proof_weighted_norm_est_on_patch_4}.
  The restriction due to the
  condition $\abs{\widehat{x}-\widehat{V}} \leq 1/4$ does not impact the estimate, as $d_{\widehat{V}}^{-1}$ is
  bounded outside of this region.    
\end{proof}

\begin{lemma}\label{lemma:weight-localize}
  Let $V \in \allvertices$ and let $\omega_V$ denote its vertex patch. Fix one element $K \subseteq \omega_V$.
  Then the following estimate holds for all $\theta \in [0,1]$ and for all
  $u \in H^{1}_h(\omega_V)$, provided that the
  right-hand side is finite:
  \begin{align}
    \norm{d_{V}^{-\theta} u}_{L^2(\omega_V)} \leq C \left[ \norm{d_{V}^{-\theta} u}_{L^2(K)} +  \norm{u}_{H^{\theta}_h(\omega_V)}\right].
    \label{eq:weighthed_norm_est_on_patch}
  \end{align}
  The constant depends on $\theta$,  the shape regularity constant $\gamma$, and $\Gamma$.
\end{lemma}
\begin{proof}
  Let $c_1>0$ be such that $B_{c_1 h_{V}}(V) \cap \Gamma$ is contained in 
  a single chart of
  $\Gamma$ and $B_{c_1 h_V}(V) \cap \Gamma \subseteq \omega_V$.
  This constant can be chosen to depend only on $\Gamma$ and the shape regularity.
  Set $\alpha_1:=c_1/4$ and $\alpha_2:=c_1/2$ and let $\chi$ be a cutoff function
  satisfying $\supp(\chi) \subseteq B_{\alpha_2 h_{V}}(V)$ and 
  $\chi \equiv 1$ on $B_{\alpha_1 h_{V}}(V)$.
  Define $u_1:=(1-\chi) u $ and $u_2:=\chi \, u$. Then it holds
  \begin{enumerate}[(i)]
    \item 
    \label{item:weight-localize-i}
    $\supp(u_1) \cap B_{\alpha_1 h_{V}}(V) = \emptyset$,
    \item
    \label{item:weight-localize-ii}
    $\supp(u_2)$ is contained in a single chart parametrizing the surface $\Gamma$,
    \item
    \label{item:weight-localize-iii}
    $\norm{u_2}_{\widetilde{H}_h^{\theta}(B_{\alpha_2 h_V}(V)\cap\Gamma)} \lesssim
      \norm{u}_{H_h^{\theta}(\omega_V)}$.
  \end{enumerate}
    While 
    (\ref{item:weight-localize-i}) and 
    (\ref{item:weight-localize-ii})
   follow directly from the construction and the properties of $\chi$, property
    (\ref{item:weight-localize-iii})
  is seen as follows:
  Note that the map $u \mapsto u_2$ is bounded as $L^2(\omega_V) \to 
  L^2\left(B_{\alpha_2 h_V}(V)\cap \Gamma\right)$ and
  $H^1_h(\omega_V) \to \widetilde{H}^1_h\left(B_{\alpha_2 h_V}(V)\cap \Gamma\right)$. To see the second boundedness, note that
  \begin{align*}
    \norm{ \nabla( \chi u )}_{L^2\left(B_{\alpha_2 h_V}(V)\cap \Gamma\right)}
    &\leq \norm{\chi}_{L^{\infty}(\omega_V)}  \norm{\nabla u }_{L^2\left(B_{\alpha_2 h_V}(V)\cap \Gamma\right)}
      + \norm{\nabla \chi}_{L^{\infty}(\omega_V)}  \norm { u }_{L^2\left(B_{\alpha_2 h_V}(V)\cap\Gamma\right)} 
\lesssim  \norm{u}_{H^1_h(\omega_V)}
  \end{align*}
  since $\norm{\nabla \chi}_{L^{\infty}(\omega_V)}\sim \dist{B_{\alpha_1 h_V}(V)}
  { \partial B_{\alpha_2 h_V}(V) }^{-1} \sim h_V^{-1}$, and that
  $u_2 \in \widetilde{H}^1_h\left(B_{\alpha_2 h_V}(V)\cap \Gamma\right)$.
  An interpolation argument and Lemma~\ref{lemma:K-vs-k} then show
    (\ref{item:weight-localize-iii}). From 
    (\ref{item:weight-localize-i}) it follows that 
  \begin{align*}
    \norm{d_{V}^{-\theta} u}_{L^2(\omega_V)}&\leq \norm{d_{V}^{-\theta} u_1}_{L^2(\omega_V)} + \norm{d_{V}^{-\theta} u_2}_{L^2(\omega_V)} 
\lesssim h_{V}^{-\theta} \norm{u}_{L^2(\omega_V)} +  \norm{d_{V}^{-\theta} u_2}_{L^2(\omega_V)},
  \end{align*}
  and hence it suffices to show~\eqref{eq:weighthed_norm_est_on_patch} for $u_2$.

    Let $F_{\Gamma}: \R^2 \to \R^3$ denote the chart parametrizing a neighborhood of $\support{u_2}$.
    We use the set $Z$ and the parameter $\beta$ from Lemma~\ref{lemma:constructing_the_cone}.    
    Since the map $F_\Gamma$ is bi-Lipschitz, we can assume (after possibly further reducing $\beta$):
    \begin{align}
    \label{eq:proof_weighted_norm_est_on_patch_5}
    \bigcup_{z \in {F^{-1}_{\Gamma}(Z)}} {\overline{B_{\beta d_{F_{\Gamma}^{-1}(V)}(z)}(z)}}
    &\subseteq F_{\Gamma}^{-1}(K).
    \end{align}

    Now we apply Lemma~\ref{lemma:sobolev:decomposition} with
    $\omega = F_{\Gamma}^{-1}(B_{c_1 h_V}(V) \cap \omega_V)$ and $\widetilde\omega = F_{\Gamma}^{-1}\big(B_{\alpha_2h_V}(V)\cap \omega_V\big)$
    to obtain a function $\widehat{w}$. Then we take the push forward to (a subset of) $\omega_V$ via $F_\Gamma$ and extend it by
    zero to $\omega_V$. This function satisfies
  \begin{align}
    \label{eq:proof_weighted_norm_est_on_patch_6}
    \int_{0}^{\infty}{t^{-2\theta}\left(\norm{u_2-w(t)}_{L^2(\omega_V)}^2
        + t^2 \norm{w(t)}_{H^{1}_h(\omega_V)}^2\right) \;\frac{dt}{t}} 
    \lesssim \norm{u_2}_{\widetilde{H}^{\theta}_h(B_{\alpha_2 h_V}(V)\cap \omega_V)}^2
  \end{align}
  and, using $\omega' = F_{K}^{-1}\big(A^{2,1}(t) \cap Z\big)$ in Lemma~\ref{lemma:sobolev:decomposition}
  and~\eqref{eq:proof_weighted_norm_est_on_patch_5} we get (by again taking the push forward and possibly further decreasing $\beta$)
  \begin{align}
    \label{eq:proof_weighted_norm_est_on_patch_7}
    \| w(t) \|_{L^2(A^{2,1}(t)\cap Z)}
    \lesssim     
    \| u_2 \|_{L^2(A^{3,1/2}(t)\cap K)}.
  \end{align}
  Next, note that for a function $v$ we have
  \begin{align*}
    \| v \|_{L^2(\omega_V)}^2
    &\lesssim \int_{\omega_V} \abs{v(x)}^2 \log(2)\,dx
    \lesssim \int_{\omega_V} \abs{v(x)}^2 \int_{d_V(x)/2}^{d_V(x)}\,\frac{dt}{t}dx\\
    &= \int_0^\infty \int_{\omega_V} \abs{v(x)}^2 \chi_{A^{2,1}(t)}(x)\,dx\frac{dt}{t} = \int_0^\infty \| v \|_{L^2(A^{2,1}(t)\cap \omega_V)}^2\,\frac{dt}{t}.
  \end{align*}
  Hence, applying~\eqref{eq:proof_weighted_norm_est_on_patch_4} to $w$, 
  we estimate
  \begin{align*}
    &\| d_{V}^{-\theta} u_2 \|_{L^2(\omega_V)}^2
    \lesssim \int_{t=0}^{\infty}{ \| d_V^{-\theta} u_2 \|_{L^2(A^{2,1}(t) \cap \omega_V)}^2 \;\frac{dt}{t}}\\
    &\qquad \lesssim \int_{t=0}^{\infty}{ \| d_V^{-\theta} (u_2 - w(t)) \|_{L^2(A^{2,1}(t) \cap \omega_V)}^2 \;\frac{dt}{t}}
    + \int_{t=0}^{\infty}{t^{2-2\theta} \| d_V^{-1}  w(t) \|_{L^2(A^{2,1}(t) \cap \omega_V)}^2 \;\frac{dt}{t}}\\
    &\qquad \lesssim \int_{t=0}^{\infty}{t^{-2\theta}
    \left( \|  (u_2 - w(t)) \|_{L^2(A^{2,1}(t) \cap \omega_V)}^2  + 
      t^2 \| d_V^{-1} w(t) \|_{L^2(A^{2,1}(t) \cap K \cap Z)}^2 +
      t^2 \| w(t) \|_{H^{1}_h(A^{2,1}(t) \cap \omega_V)}^2\right) \;\frac{dt}{t}} \nonumber\\
    &\qquad \leq \int_{t=0}^{\infty}{t^{-2\theta}\left( \| u_2 - w(t) \|_{L^2(\omega_V)}^2  + t^2 \| w(t) \|_{H^1_h(\omega_V)}^2\right) \; \frac{dt}{t}} 
    +\int_{t=0}^{\infty}{t^{2-2\theta} \| d_V^{-1} w(t) \|_{L^2(A^{2,1}(t) \cap K \cap Z)}^2  \frac{dt}{t} },
  \end{align*}
  where  we used that $d_V^{-1} \sim t^{-1}$ on $A^{2,1}(t)$. 
  Using additionally~\eqref{eq:proof_weighted_norm_est_on_patch_7}, we obtain
  \begin{align*}
    \int_{t=0}^{\infty} t^{2-2\theta} \| d_V^{-1} w(t) \|_{L^2(A^{2,1}(t) \cap Z)}^2  \frac{dt}{t}
    \lesssim \int_0^\infty t^{2-2\theta} \| d_V^{-1}u_2 \|_{L^2(A^{3,1/2}(t)\cap K)}^2 \frac{dt}{t}
    \lesssim \| d_V^{-\theta} u_2 \|_{L^2(K)}^2,
  \end{align*}
  where the last estimate can be seen using polar coordinates.
  Using also~\eqref{eq:proof_weighted_norm_est_on_patch_6} and point 
  (\ref{item:weight-localize-iii}) above shows 
  \begin{align*}
    \| d_{V}^{-\theta} u_2 \|_{L^2(\omega_V)}
    &\lesssim \| d_V^{-\theta} u \|_{L^2(K)} +\| u \|_{ H^\theta_h(\omega_V)}.
\qedhere
  \end{align*}
\end{proof}
\section{Averaging operators} 
\label{sec:averaging-operators}
A key tool in the definition of the localization procedure are suitable averaging operators. 
Given a closed set $M$ (which will be a corner or an edge of $\Kref$) 
the basic averaging operator is defined in Lemma~\ref{lemma:general-averaging} by locally averaging 
on a length scale of size $\bigO(\operatorname{dist}(\cdot,M))$. In this way function values are reproduced 
on $M$. This basic operator is modified in several steps in order to yield 
functions  that vanish on parts of $\Kref$.  
We study the stability of these operators in weighted and unweighted norms. 

\subsection{Preliminaries}
The proof of stability of the local averaging operators relies on the following covering lemma.
\begin{lemma}[covering lemma]\label{lemma:covering}
  Let $M \ne \emptyset$ be a closed subset of $\R^d$.
  Fix $\beta \in (0,1)$ and $c \in (0,1)$ such that
  \begin{align*}
    (1+\beta)c< 1.
  \end{align*}
  For each $x \in \R^d \setminus M$ denote by $B_x:= \overline{B}_{c d_M(x)}(x)$ and 
  $\widehat B_x:= \overline{B}_{(1+\beta)c d_M(x)}(x)$ closed balls centered at $x$ with radii 
  $c  d_M(x)$, $(1+\beta)c d_M(x)$, respectively.
  Let $\omega\subset \R^d\setminus M$ be open. Then, there
  exist $(x_i)_{i \in \N} \subset \omega$
and a constant $N \in \N$ depending solely on the spatial dimension $d$
  such that:
\begin{enumerate}[(i)]
  \item (covering property) $\omega\subset \cup_i B_{x_i}$;
  \item (finite overlap)
    $\sup_{x\in\R^d}\operatorname*{card} \{i\colon x \in \widehat B_{x_i}\} \leq N$. 
\end{enumerate} 
\end{lemma}
\begin{proof}
The proof is very similar to that of \protect{\cite[Lemma~{A.1}]{melenk-wohlmuth12}}. We have to address
the technical issue that the radii of the balls are unbounded so that it is not {\sl a priori} clear that the 
classical Besicovitch covering theory is applicable. 
Fix $c_0>0$ such that $(1+\beta)c < 1-c_0$
and define: $q:= \max\{1+c_0, 1/(1-c_0)\} > 1$. Define, 
for each $i \in \Z$,  
the bounded sets $\omega_i:= \{x \in \omega\colon q^i \leq d_M(x) \leq q^{i+1}\}$.
For each of these sets,
we can find a cover by balls of the above type. The choice of $q > 1$ is such that balls with centers 
in $\omega_i$ have non-trivial intersection with $\omega_j$ only for $j \in \{i-1,i,i+1\}$.
Hence, the overlap
properties can be ensured. 
\end{proof}
The basic lemma for defining the local averaging operators is the following.
\begin{lemma}[basic averaging operator]
\label{lemma:general-averaging}
  Let $M \ne \emptyset$ be a closed subset of $\R^d$. 
  Fix $\alpha \in (0,1/3)$ and let $\rho \in C^\infty_0(\R^d)$ satisfy 
  $\operatorname*{supp} \rho \subset B_\alpha(0)$ and $\int_{\R^d} \rho(y)\,dy = 1$.
  For $t > 0$, set $\rho_t(y):= t^{-d} \rho(y/t)$. 
  Define the averaging operator $\A^M_\rho:L^1_{loc}(\R^d) \rightarrow C^\infty(\R^d \setminus M)$ by 
  \begin{align*}
    (\A^M_\rho u)(x) := \int_{\R^d} u(y) \rho_{d_M(x)} (x - y)\,dy,
  \end{align*}
  and for a set $\omega\subset\R^d\setminus M$ the domain of influence
\begin{equation}
\label{eq:domain-of-influence}
\omega_\rho^M:= \cup_{y \in \omega} \supp \rho_{d_M(y)}(y-\cdot).
\end{equation}
  Then, the following holds: 
  \begin{enumerate}[(i)]
    \item \label{item:lemma:general-averaging-i} $\A^M_\rho 1 = 1$. 
    \item \label{item:lemma:general-averaging-ii}
      If $\omega \subset \R^d\setminus M$ is open, then the conditions
      $d_M|_{\omega_\rho^M}\in\P_1$ and $u|_{\omega_\rho^M} \in \P_p$ imply
      $(\A^M_{\rho} u)|_{\omega} \in \P_p$. 
  \end{enumerate}
  Let a function $\varphi\in L^\infty(\R^d)$ be given that satisfies,
  for some $c \in (0,1)$, $C > 0$, the conditions
  \begin{equation} \label{eq:lemma:general-averaging-1000}
    C^{-1} \varphi(y) \leq \varphi(x) \leq C \varphi(y) \qquad \forall y \in B_{c d_M(x)}(x)
    \qquad \forall x \in \R^d \setminus M. 
  \end{equation} 
  Then for any open $\omega \subset \R^d\setminus M$, there holds: 
  \begin{enumerate}[(i)]\setcounter{enumi}{2}
    \item \label{item:lemma:general-averaging-iv}
      $\|\varphi \A^M u\|_{L^2(\omega)} \lesssim \|\varphi u\|_{L^2(\omega^M_\rho)}$
      provided the right-hand side is finite.
    \item \label{item:lemma:general-averaging-v}
      $\|\varphi \nabla \A^M u\|_{L^2(\omega)} \lesssim
      \|\varphi \nabla u\|_{L^2(\omega^M_\rho)}$ provided the right-hand side is finite.
    \item \label{item:lemma:general-averaging-iii}
      $\A^M_\rho: L^2(\omega^M_\rho) \rightarrow L^2(\omega)$ and 
      $\A^M_\rho: H^1(\omega^M_\rho) \rightarrow H^1(\omega)$ is bounded.
    \item \label{item:lemma:general-averaging-vii}
      If $u$ is continuous at a point $z \in M$, then $\A^M_\rho u$ is continuous at $z$ and 
      $(\A^M_\rho u)(z) = u(z)$.
    \item \label{item:lemma:general-averaging-viii}
      If $u \in H^\theta(\R^d \setminus M)$, then
      $\|d_M^{-\theta}(u - \A^M_\rho u)\|_{L^2(\R^d\setminus M)} \lesssim
      \|u\|_{H^{\theta}(\R^d\setminus M)}$.
    \item \label{item:lemma:general-averaging-ix}
      $\|\nabla \A^M_\rho u\|_{L^2(\omega)} \lesssim \|d_M^{-1} u\|_{L^2(\omega^M_\rho)}$
      provided the right-hand side is finite.
    \item \label{item:lemma:general-averaging-ixa}
      $\|d_M \nabla \A^M_\rho u\|_{L^2(\omega)} \lesssim \|u\|_{L^2(\omega^M_\rho)}$.
    \item \label{item:lemma:genreal_averaging-d2}
      Assume that  $d_M|_{\omega_\rho^M}\in\P_1$, then
      $\|d_M \nabla^2 \A^M_\rho u\|_{L^2(\omega)} \lesssim \|\nabla u\|_{L^2(\omega^M_\rho)}.$
    \item \label{item:lemma:general-averaging-vi}
      Let $K \subset \R^d$ be compact with $\dist{K}{M} > 0$. Then 
      $\|\A^M_\rho u\|_{W^{j,\infty}(K)} \leq C_{K,\rho,d,M,j} \|u\|_{L^2(K^M_\rho)}$
      for all $j\in\N\cup\left\{ 0 \right\}$.
  \end{enumerate}
  All hidden constants depend solely on $d$, $\rho$, $M$, and $\theta$.
\end{lemma}
\begin{proof}
  We start with some preparatory results.
  \begin{itemize}
    \item Elementary considerations show that $d_M:\R^d \rightarrow \R$
      is Lipschitz continuous with Lipschitz constant $L = 1$.
    \item Let $\varphi$ satisfy (\ref{eq:lemma:general-averaging-1000}). 
      Then for any $\widehat c \in (0,1)$ there exists $\widehat C > 0$ such that 
      \begin{equation}\label{eq:lemma:general-averaging-2000} 
	\widehat C^{-1} \varphi(y) \leq \varphi(x) \leq \widehat C \varphi(y)
	\qquad \forall y \in B_{\widehat c d_M(x)}(x) \quad \forall x \in \R^d \setminus M.
      \end{equation}
      To see this, it suffices to consider the case $1 > \widehat c > c$.
      For fixed $x \in \R^d \setminus M$ and 
      $y \in B_{\widehat c d_M(x)}(x)$, consider a sequence of $L = \lfloor \widehat c/(c(1-\widehat c) )\rfloor +1$ 
      points $x = x_0,\ldots, x_L = y$ on the line connecting
      $x = x_0$ with $y=x_L$ such that $x_i \in B_{c d_M(x_{i-1})}(x_{i-1})$, $i=1,\ldots,L$. 
      For each pair $(x_{i-1}, x_i)$, the assumption (\ref{eq:lemma:general-averaging-1000}) is applicable. 
      Hence, the claim (\ref{eq:lemma:general-averaging-2000}) follows with $\widehat C:= C^L$.
    \item A change of variables yields
      \begin{equation}\label{eq:lemma:general-averaging-100}
	(\A^M_\rho u)(x) = 
	\int_{B_\alpha(0)} u\big(x - d_M(x) y\big) \rho(y)\,dy. 
      \end{equation}
  \end{itemize}

  {\bf Proof of (\ref{item:lemma:general-averaging-i}) and (\ref{item:lemma:general-averaging-ii}):}
  This follows by a direct calculation.

  {\bf Proof of (\ref{item:lemma:general-averaging-iv}):}
  First, we provide a locally finite cover of $\omega$.
  Recall $\alpha\in (0,1/3)$. Fix $c\in(0,1/2)$ such that $\alpha < c/(1+c)$.
  Then fix $\beta\in (0,1)$ such that $\alpha (1+c)/c < \beta$ and observe 
  \begin{align}\label{eq:lemma:general-averaging-4000}
    c < c + (1+c)\alpha < c(1+\beta) < 1.
  \end{align}
  According to Lemma~\ref{lemma:covering}, there are points $x_i\in\omega$, $i\in\N$, such that
  the balls $B_{x_i}:= \overline B_{c d_M(x_i)}(x_i)$
  and the stretched balls $\widehat B_{x_i}:= \overline B_{(1+\beta)c d_M(x_i)}(x_i)$
  fulfill
  \begin{subequations}\label{eq:lemma:general-averaging-5000}
    \begin{align}
      \omega \subset \cup_i B_{x_i}\label{eq:lemma:general-averaging-5001},\\
      \sup_{x \in \R^d} \operatorname*{card} \{i\colon x \in \widehat B_{x_i}\} \leq N \in\N.
      \label{eq:lemma:general-averaging-5002}
    \end{align}
  \end{subequations}
  Setting $\widetilde B_{x_i}:= \overline B_{(c+(1+c)\alpha) d_M(x_i)}(x_i)$ and taking into
  account~\eqref{eq:lemma:general-averaging-4000}, we also conclude
  \begin{align*}
    \sup_{x \in \R^d} \operatorname*{card} \{i\colon x \in \widetilde B_{x_i}\} \leq N.
  \end{align*}
  Furthermore, it follows from the Lipschitz continuity of $d_M$ with constant $L=1$
  that $y\in B_{x_i}\cap \omega$ implies
  $\supp \rho_{d_M(y)}(y-\cdot)\subset \widetilde B_{x_i}$, which yields
  \begin{align*}
    \left( B_{x_i}\cap\omega \right)^M_\rho \subset \widetilde B_{x_i}
  \end{align*}
  and consequently
  \begin{align}\label{eq:lemma:general-averaging-5003}
    \sup_{x \in \R^d} \operatorname*{card} \{i\colon x \in \left( B_{x_i}\cap\omega \right)^M_\rho \}
    \leq N.
  \end{align}
  For any $z\in\omega$, we have the basic estimate
  \begin{align}\label{eq:lemma:general-averaging-basic-1}
    \abs{\A^M_\rho u(z)} \lesssim d_M(z)^{-d}\norm{u}_{L^1(\supp \rho_{d_M(z)}(z-\cdot))}.
  \end{align}
  For $i\in\N$ we conclude, using H\"older's inequality and~\eqref{eq:lemma:general-averaging-2000},
  \begin{align}\label{eq:lemma:general-averaging-6000}
    \norm{\varphi \A^M_\rho u}_{L^2(B_{x_i}\cap\omega)}
    \lesssim \norm{\varphi u}_{L^2((B_{x_i}\cap \omega)^M_\rho)}.
  \end{align}
  The covering property~\eqref{eq:lemma:general-averaging-5001}
  and properties~\eqref{eq:lemma:general-averaging-5003} then yield the result.

  {\bf Proof of (\ref{item:lemma:general-averaging-v}):} By Rademacher's
  Theorem~\cite[Thm.~3.1.6]{federer}, $d_M$ is differentiable almost everywhere
  with $\abs{\nabla d_M(x)}\leq 1$. 
  Using~\eqref{eq:lemma:general-averaging-100}, we then write
  \begin{equation}
    \label{it:general_averaging_diff_formula}
    \nabla \A^M_\rho u(x) = \int_{B_\alpha(0)}
    (I - \nabla d_M(x)  y^T)\nabla u(x - d_M(x) y)  \rho(y)\,dy.
  \end{equation}
  Now we can argue exactly as in the proof of~\eqref{item:lemma:general-averaging-iv}.

  {\bf Proof of (\ref{item:lemma:general-averaging-iii}):}
  follows from (\ref{item:lemma:general-averaging-iv}) and (\ref{item:lemma:general-averaging-v})
  with $\varphi \equiv 1$.
  
  {\bf Proof of (\ref{item:lemma:general-averaging-vii}):} 
  This follows from~\eqref{eq:lemma:general-averaging-100}
  and the observation $\lim_{x' \rightarrow x} d_M(x') = 0$ for $x \in M$.

  {\bf Proof of (\ref{item:lemma:general-averaging-viii}):} 
  We start with the case $\theta=1$.
  Estimate~\eqref{eq:lemma:general-averaging-6000}, (i), and a Poincar\'e inequality show
  $
  \|\A^M_\rho u - u\|_{L^2(B_{x_i})} 
  \lesssim d_M(x_i) \|\nabla u\|_{L^2(\widetilde B_{x_i})}. 
  $
  Dividing by $d_M(x_i)$ and using that $d_M(x_i) \sim d_M$ on $B_{x_i}$ yields
  \begin{equation*}
    \|d_M^{-1} \left( \A^M_\rho u - u\right)\|_{L^2(B_{x_i})} 
    \lesssim \|\nabla u\|_{L^2(\widetilde B_{x_i})}. 
  \end{equation*}
  Now we can argue as in the proof of~\eqref{item:lemma:general-averaging-iv}.
  The case of general $\theta \in (0,1)$ follows by interpolation.
  
  {\bf Proof of (\ref{item:lemma:general-averaging-ix}) and (\ref{item:lemma:general-averaging-ixa}):}
  We calculate 

  \begin{align*}
    \partial_{x_j} \left[ \rho\left( \frac{x-y}{d_M(x)} \right) d_M(x)^{-d} \right]
    = &\rho\left( \frac{x-y}{d_M(x)} \right) (-d)d_M(x)^{-d-1}\partial_{x_j}d_m(x)\\
    &+d_M(x)^{-d}
    \left(- (x-y) \cdot \nabla \rho \left( \frac{x-y}{d_M(x)} \right) \frac{\partial_{x_j}d_M(x)}{d_M(x)^2}
    + \partial_{x_j}\rho \left( \frac{x-y}{d_M(x)} \right) \frac{1}{d_M(x)}
    \right)
  \end{align*}
  and conclude for any $y\in B_{\alpha d_M(x)}(x)$ that 
  \begin{align*}
    \abs{\nabla_x \rho_{d_M(x)}(x-y)} \lesssim d_M(x)^{-d-1}.
  \end{align*}
  Hence, we have the basic estimate 
  \begin{align}\label{eq:lemma:general-averaging-basic-2}
    \abs{\nabla \A^M_\rho u(x)} \lesssim 
    d_M(x)^{-d-1}\norm{u}_{L^1(\supp \rho_{d_M(x)}(x-\cdot))}
  \end{align}
  and can proceed as in the proof of~\eqref{item:lemma:general-averaging-iv}.

    {\bf Proof of~(\ref{item:lemma:genreal_averaging-d2}):}
    Starting from~\eqref{it:general_averaging_diff_formula}, a simple transformation gives for $x \in \omega$
    \begin{align*}
      \nabla \A^M_\rho u(x)
      &=\int_{\omega_{\rho}^{M}}
      \big(I - \nabla d_M(x) d_{M}^{-1}(x) (x-y)^T\big)\nabla u(y)  \rho_{d_{M}(x)}(x-y)\,dy.
    \end{align*}
    The assumption $d_{M} \in \mathcal{P}_1$ implies that $\nabla d_{M}$ is constant.
    Differentiating with respect to $x$, we can continue as in the proof of~(\ref{item:lemma:general-averaging-ixa}).
    In particular, the additional term $(x-y)d_{M}^{-1}(x)$ and its $x$-derivative
      behave like $\bigO(1)$ and $\bigO(d_{M}^{-1})$, leading to a differentiated version of~\eqref{eq:lemma:general-averaging-basic-2}.

  {\bf Proof of (\ref{item:lemma:general-averaging-vi}):}
  This follows immediately from the basic estimates~\eqref{eq:lemma:general-averaging-basic-1}
  and~\eqref{eq:lemma:general-averaging-basic-2} for $j=0$, respectively $j=1$.
  The case $j>1$ can be shown analogously.
\end{proof}
\subsection{Averaging operators for the vertex part}
We first construct the  averaging operators for the vertex contributions.
  Since it already showcases most of the difficulties encountered also for the other contributions,
  we present the construction in more detail.
\subsubsection{Averaging operators on the reference triangle for the vertex part}
We start with a simple averaging operator on the reference triangle, 
which will be the basis for the construction of further operators 
with additional properties. 
\begin{lemma}\label{lemma:vertex-part}
  Let $\Vref$ be a vertex of $\Tref$. 
  There exists a linear operator ${\mathcal A}^{\Vref}:L^1_{loc}(\Tref) \rightarrow C^\infty(\Tref)$ with 
  the following properties: 
  \begin{enumerate} [(i)]
    \item \label{item:lemma:vertex-part-3}
      If $u$ is continuous at $\widehat V$, then $({\mathcal A}^{\Vref} u)(\widehat V) = u(\widehat V)$. 
    \item \label{item:lemma:vertex-part-4}
      ${\mathcal A}^{\Vref} 1 \equiv 1$. Furthermore, 
      $u \in {\mathcal P}_p$ implies ${\mathcal A}^{\Vref} u \in {\mathcal P}_p$. 
    \item \label{item:lemma:vertex-part-1}
      ${\mathcal A}^{\Vref}: H^{\theta}(\Tref) \rightarrow H^{\theta}(\Tref)$
    is bounded and linear for $\theta \in [0,1]$.
    \item \label{item:lemma:vertex-part-2}
      For fixed $\gamma \in \R$, 
      $\displaystyle 
      \|d_{\Vref}^\gamma {\mathcal A}^{\Vref} u\|_{L^2(\Tref)} 
        \leq C_\gamma
	\|d_{\Vref}^\gamma u\|_{L^2(\Tref)} 
      $
      provided the right-hand side is finite. 
    \item \label{item:lemma:vertex-part-6}
      For all $\theta \in [0,1]$ there exists $C>0$ such that for all $u \in H^{\theta}(\Tref)$:
      $$
      \|d_{\Vref}^{-\theta}( u - {\mathcal A}^{\Vref} u)\|_{L^2(\Tref)} \leq C \|u\|_{H^{\theta}(\Tref)}.
      $$
    \item \label{item:lemma:vertex-part-7}
      For every $\varepsilon > 0$ we have 
      $\A^{\Vref} u \in C^\infty(T_\varepsilon)$, where 
      $T_\varepsilon:= \{(\xi,\eta) \in \overline\Tref \,|\, d_{\widehat{V}}(\xi,\eta) \geq \varepsilon\}$. Moreover, 
      for every $j \in \N_0$ and $\varepsilon > 0$, there exists $C_{j,\varepsilon} > 0$ such that 
      $$
        \|\A^{\Vref} u\|_{W^{j,\infty}(T_\varepsilon)} \leq C_{j,\varepsilon} \|u\|_{L^2(\Tref)}. 
      $$
    \item \label{item:lemma:vertex-part-8}
      Let $\widehat{e}$ denote an edge of $\Tref$ with $\Vref \in \overline{\widehat{e}}$, then
      \begin{align}\label{eq:lemma:vertex-part-8}
	\norm{d_{\Vref}^{-1/2} \A^{\Vref}u  }_{L^2(\widehat{e})} \leq C \norm{d_{\Vref}^{-1} u}_{L^2(\Tref)}.
      \end{align}
  \end{enumerate} 
\end{lemma}
\begin{proof}
  To fix ideas, we assume that $\Vref = (0,0)$. We will apply Lemma~\ref{lemma:general-averaging},
  where we assume additionally that $\supp\rho \subset B_\alpha(0)\cap \Tref$ for some $\alpha\in(0,1/3)$.
  We choose $\beta<1$ with $\beta(\sqrt{2}+\alpha)<1$
  and the set $M = \left\{ 0 \right\}\times\R$.
  Define $\widehat u(x):= u(\beta x)$ for $x\in\Tref$, and extend $\widehat u$ to $\R^d$
  using the Stein extension~\cite[Sec.~{VI}]{stein70}, which is simultaneously bounded in $L^2$ and $H^1$.
  Note that $\widehat u(\widehat{V}) = u(\widehat{V})$ if $u$ is continuous at $\widehat{V}$.
  Now we define $\A^{\Vref} u := \A^M_\rho \widehat u$ with the aid of Lemma~\ref{lemma:general-averaging}.
  Note that due to the choice of $\beta$ and the support of $\rho$, $(\A^{\Vref} u)|_{\Tref}$ depends solely on $u|_{\Tref}$.

  {\bf Proof of (\ref{item:lemma:vertex-part-3}):} This follows from
  Lemma~\ref{lemma:general-averaging},~\eqref{item:lemma:general-averaging-vii} and the fact that $\widehat u(\widehat{V}) = u(\widehat{V})$.

  {\bf Proof of (\ref{item:lemma:vertex-part-4}):}
  The choice of $M$ implies $d_M|_{\{x > 0\}} \in \P_1$. Furthermore, $u\in\P_p(\Tref)$
  implies $\widehat u\in\P_p(\beta^{-1}\Tref)$, and we conclude with
  Lemma~\ref{lemma:general-averaging},~\eqref{item:lemma:general-averaging-ii}
  that $\A^{\Vref}u\in \P_p(\Tref)$. Together with
  Lemma~\ref{lemma:general-averaging},~\eqref{item:lemma:general-averaging-i},
  this shows~\eqref{item:lemma:vertex-part-4}.

  {\bf Proof of (\ref{item:lemma:vertex-part-1}):} For $\theta\in\left\{ 0,1 \right\}$, note that
  due to Lemma~\ref{lemma:general-averaging},~\eqref{item:lemma:general-averaging-iii}
  and $\Tref^M_\rho \subset \beta^{-1}\Tref$,
  \begin{align*}
    \|\A^{\Vref}u\|_{H^\theta(\Tref)} \lesssim
    \| \widehat u \|_{H^\theta(\beta^{-1}\Tref)}
    \lesssim \|u\|_{H^\theta(\Tref)}.
  \end{align*}
  An interpolation argument finishes the proof.

  {\bf Proof of (\ref{item:lemma:vertex-part-2}):}
  We note that due to Lipschitz continuity
  of $d_M$, the function $\varphi := d_M^\gamma$ fulfills~\eqref{eq:lemma:general-averaging-1000}.
  Hence, the estimate follows from
  Lemma~\ref{lemma:general-averaging},~\eqref{item:lemma:general-averaging-iv}.

  {\bf Proof of (\ref{item:lemma:vertex-part-6}):}
  We calculate, using $d_{\Vref}\sim d_M$ on $\Tref$,
  $$
    \|d_{\Vref}^{-\theta}( u - \A^{\Vref} u)\|_{L^2(\Tref)} \lesssim
    \|d_M^{-\theta}( u - \A^{\Vref} u)\|_{L^2(\Tref)} \leq 
    \|d_M^{-\theta}( \widehat u - \A^M_\rho \widehat u)\|_{L^2(\R^d\setminus M)} + \| d_M^{-\theta} (u - \widehat u)\|_{L^2(\Tref)}. 
  $$
  Next, 
  \begin{align*}
    \|d_M^{-\theta}( \widehat u - \A^M_\rho \widehat u)\|_{L^2(\R^d\setminus M)}\lesssim \|\widehat u \|_{H^{\theta}(\R^d\setminus M)}
    \lesssim \|u\|_{H^{\theta}(\Tref)},
  \end{align*}
  where the first estimate follows from Lemma~\ref{lemma:general-averaging},~\eqref{item:lemma:general-averaging-viii},
  and the last one is due to the boundedness of the Stein extension operator and the definition of $\widehat u$.
  It remains to consider $\| d_M^{-\theta} (u - \widehat u)\|_{L^2(\Tref)}$. For $\theta=0$,
  this term can be bounded by $\| u \|_{L^2(\Tref)}$. For $\theta=1$, define the function
  $\widetilde u(x,y) = u(x,\beta y)$ and estimate 
  \begin{align*}
     \int_{\Tref} d_M^{-2}& \left( u(x,\beta y)-u(\beta x,\beta y) \right)^2\,dxdy 
      \leq \int_0^1\int_y^1 \left( x^{-1} \int_{\beta x}^x \partial_1 u(s,\beta y)\,ds \right)^2 \!dxdy \\
     &\leq \int_0^1\int_y^1 \left( x^{-1} \int_{0}^x \abs{\partial_1 u(s,\beta y)}\,ds \right)^2 \!dxdy 
    \lesssim \int_0^1 \int_y^1 \abs{\partial_1 u(x,\beta y)}^2\!dxdy \\
     &\lesssim \int_{\Tref} \abs{\partial_1 u(x,y)}^2\,dxdy,
  \end{align*}
  where we used Hardy's inequality in the second step. Likewise,
  \begin{align*}
    & \int_{\Tref} d_M^{-2} \left( u(x,y)-u(x,\beta y) \right)^2\,dxdy
    \leq \int_0^1 \int_0^x \left( x^{-1} \int_0^y \abs{\partial_2 u(x,s)}\,ds \right)^2\,dydx\\
    &\quad \leq \int_0^1 \int_0^x \left( y^{-1} \int_0^y \abs{\partial_2 u(x,s)}\,ds \right)^2\,dydx
    \lesssim \int_0^1\int_0^x \abs{\partial_2 u(x,y)}^2\,dydx.
  \end{align*}
  The triangle inequality then shows that $\| d_M^{-1} (u - \widehat u)\|_{L^2(\Tref)} \lesssim \| u \|_{H^1(\Tref)}$.
  The case of $\theta\in(0,1)$ follows by interpolation.

  {\bf Proof of (\ref{item:lemma:vertex-part-7}):} This follows from
  Lemma~\ref{lemma:general-averaging},~\eqref{item:lemma:general-averaging-vi}.

  {\bf Proof of (\ref{item:lemma:vertex-part-8}):} This follows from
  Lemma~\ref{lemma:weighted-norm-equivalence} and
  Lemma~\ref{lemma:general-averaging},~\eqref{item:lemma:general-averaging-ix}, using that
  $d_{\Vref}\sim d_M$ on $\Tref$.
\end{proof}
We now modify ${\mathcal A}^{\widehat{V}}$ to construct an operator ${\mathcal A}^{\widehat{V}}_0$ that produces functions that vanish on one edge of $\Tref$ and have a
convenient symmetry, making the extension to a patch-supported function easier:  
\begin{lemma}\label{lemma:vertex-part-edge-vanish}
  Let $\Vref$ be a vertex of $\Tref$. There exists a linear operator
  $\A^{\Vref}_0:L^1_{loc}(\Tref) \rightarrow C^\infty(\Tref)$
  with the following properties:
\begin{enumerate}[(i)]
  \item\label{item:lemma:vertex-part-edge-vanish-1}
    If $u$ is continuous at $\widehat V$, then $({\A}^{\Vref}_0 u)(\widehat V) = u(\widehat V)$. 
  \item\label{item:lemma:vertex-part-edge-vanish-2}
    $u \in {\mathcal P}_p$ implies $\A^{\Vref}_0 u \in {\mathcal P}_p$ for all $p\geq 1$.
  \item\label{item:lemma:vertex-part-edge-vanish-3}
    $\A^{\Vref}_0: H^{\theta}(\Tref) \rightarrow H^{\theta}(\Tref)$
    is bounded and linear for $\theta \in [0,1]$.
  \item\label{item:lemma:vertex-part-edge-vanish-4}
    For fixed $\gamma \in [-1,0]$, there holds 
    $\displaystyle   \|d_{\Vref}^\gamma \A^{\Vref}_0 u\|_{L^2(\Tref)} 
      \leq C_\gamma
      \|d_{\Vref}^\gamma u\|_{L^2(\Tref)} 
    $
    provided the right-hand side is finite.
  \item\label{item:lemma:vertex-part-edge-vanish-5}
      For all $\theta \in [0,1]$ there exists $C>0$ such that for $u \in H^{\theta}(\Tref)$
      $$
      \|d_{\Vref}^{-\theta}( u - \A^{\Vref}_0 u)\|_{L^2(\Tref)} \leq C \|u\|_{H^{\theta}(\Tref)}
      $$
  \item\label{item:lemma:vertex-part-edge-vanish-6}
    For every $\varepsilon > 0$ we have 
    $\A^{\Vref}_0 u \in C^\infty(T_\varepsilon)$, where 
    $T_\varepsilon:= \{(\xi,\eta) \in \overline\Tref \,|\, d_{\widehat{V}}(\xi,\eta) \geq \varepsilon\}$. Moreover, 
    for every $j \in \N_0$ and $\varepsilon > 0$, there exists $C_{j,\varepsilon} > 0$ such that 
    $$
      \|\A^{\Vref}_0 u\|_{W^{j,\infty}(T_\varepsilon)} \leq C_{j,\varepsilon} \|u\|_{L^2(\Tref)}. 
    $$
  \item\label{item:lemma:vertex-part-edge-vanish-7}
      Let $e$ denote an edge of $\Tref$ with $\Vref \in \overline{e}$. Then
      \begin{align}\label{eq:lemma:vertex-part-edge-vanish-8}
	\norm{d_{\Vref}^{-1/2} \A^{\Vref}_0 u  }_{L^2(e)} \leq C \norm{d_{\Vref}^{-1} u}_{L^2(\Tref)}.
      \end{align}
  \item\label{item:lemma:vertex-part-edge-vanish-8}
    Let $\widehat{e}$ denote the edge opposite $\Vref$. Then $(\A^{\Vref}_0 u)|_{\widehat e} = 0$. 
  \item\label{item:lemma:vertex-part-edge-vanish-9}
    Let $\Vref = (1,0)$. Then
    $\A^{\Vref}_0 u(\xi,\eta) = \A^{\Vref}_0 u(1-\eta,1-\xi)$ for all $(\xi,\eta) \in \Tref$. 
  \item\label{item:lemma:vertex-part-edge-vanish-10}
    Let $V\neq \widehat V$ be another vertex of $\Tref$. Then
    \begin{align*}
      \| d_{V}^{-1/2} \A^{\Vref}_0 u \|_{L^2(\Tref)} \lesssim \| u \|_{L^2(\Tref)}
      \quad\text{ and }\quad
      \| d_{V}^{-1/2}\nabla \A^{\Vref}_0 u \|_{L^2(\Tref)} \lesssim \| u \|_{H^1(\Tref)}.
    \end{align*}
\end{enumerate}
\end{lemma}
\begin{proof}
  As in the proof of Lemma~\ref{lemma:vertex-part}, we assume $\Vref = (0,0)$.
  The operator $\A^{\Vref}_0$ is constructed by a sequence of modifications 
  of $u_1:= \A^{\Vref} u$, where $\A^{\Vref}$ is the operator from Lemma~\ref{lemma:vertex-part}.
  Set $u_2(\xi,\eta):= u_1(\xi,\eta) - l_1(\xi,\eta) u_1(1,0) - l_2(\xi,\eta) u_1(1,1)$,
  where $l_1$ and $l_2$ are the affine hat functions associated with the vertices $(1,0)$ and $(1,1)$.
  It follows immediately with 
  Lemma~\ref{lemma:vertex-part}, (\ref{item:lemma:vertex-part-7}) that 
  \begin{align}\label{lemma:vertex-part-edge-vanish:eq1000}
    \| u_2 \|_{W^{1,\infty}(T_{1/2})} \lesssim \| u \|_{L^2(\Tref)}. 
  \end{align}
  Define $g(\xi,\eta) = u_2(1,\eta)(\xi-\eta)/(1-\eta)$ and set
  $u_3(\xi,\eta):= u_2(\xi,\eta) - g(\xi,\eta)$. 
  It follows immediately that $u_3(\Vref)=u(\Vref)$ if $u$ is continuous at $\Vref$
  and that $u_3|_{\widehat e}=0$ if $\widehat e$ is the edge opposite to $\Vref$.
  Using the properties of $\A^{\Vref}$, we note
  \begin{align*}
    \|u_3\|_{H^\theta(\Tref)} \lesssim \|\A^{\Vref} u\|_{H^\theta(\Tref)}
    + \|\A^{\Vref} u\|_{L^\infty(T_{1/2})} + \| g \|_{H^\theta(\Tref)}
    \lesssim \|u\|_{H^\theta(\Tref)} + \| g \|_{H^\theta(\Tref)}.
  \end{align*}
  Using~\eqref{lemma:vertex-part-edge-vanish:eq1000}
  we conclude $\|g\|_{L^2(\Tref)} \lesssim \|u\|_{L^2(\Tref)}$, and since
  $u_2$ vanishes to first order in the vertices $(1,0)$ and $(1,1)$, we also conclude
  from~\eqref{lemma:vertex-part-edge-vanish:eq1000} that $\|g\|_{H^1(\Tref)} \lesssim \| u \|_{L^2(\Tref)}$.
  An interpolation argument shows $\| g \|_{H^\theta(\Tref)}\lesssim \| u \|_{L^2(\Tref)}$ and hence
  $\| u_3 \|_{H^\theta(\Tref)}\lesssim\| u \|_{H^\theta(\Tref)}$.
  The results obtained so far and a triangle inequality yield $\| u - u_3 \|_{L^2(\Tref)} \lesssim \| u \|_{L^2(\Tref)}$.
  Next, note that every term of $h:=u_3 - \A^{\Vref} u$ vanishes at least on
  one edge containing $\Vref$. Hence, by Lemma~\ref{lemma:weighted-norm-equivalence}, the triangle inequality, and
  the results obtained so far, we conclude $\| d_{\Vref}^{-1} h \|_{L^2(\Tref)} \lesssim \| u \|_{L^2(\Tref)}$.
  Lemma~\ref{lemma:vertex-part}, (\ref{item:lemma:vertex-part-6}) then implies
  \begin{align*}
    \| d_{\Vref}^{-1} (u-u_3) \|_{L^2(\Tref)} \leq
    \| d_{\Vref}^{-1} (u-\A^{\Vref}u) \|_{L^2(\Tref)} + \| d_{\Vref}^{-1} h \|_{L^2(\Tref)}
    \lesssim \| u \|_{H^1(\Tref)}.
  \end{align*}
  An interpolation argument shows $\| d_{\Vref}^{-\theta} (u-u_3) \|_{L^2(\Tref)}\lesssim\| u \|_{H^\theta(\Tref)}$.
  It is seen immediately that $u_3$ satisfies also the statement of Lemma~\ref{lemma:vertex-part}, (\ref{item:lemma:vertex-part-7}).
  Finally, we will show that $u_3$ fulfills the statement of (\ref{item:lemma:vertex-part-8}). Due to Lemma~\ref{lemma:vertex-part},
  (\ref{item:lemma:vertex-part-7}) and (\ref{item:lemma:vertex-part-8}),
  \begin{align*}
    \| d_{\Vref}^{-1/2} u_2 \|_{L^2(e)} \leq
    \| d_{\Vref}^{-1/2} \A^{\Vref}u \|_{L^2(e)} + 
    \left(\| d_{\Vref}^{-1/2} l_1 \|_{L^2(e)}
    + \| d_{\Vref}^{-1/2} l_2 \|_{L^2(e)} \right) \| \A^{\Vref} u \|_{L^{\infty}(T_{1/2})}
    \lesssim \| u \|_{L^2(\Tref)},
  \end{align*}
  and furthermore
  \begin{align*}
    \| d_{\Vref}^{-1/2} g \|_{L^2(e)} \lesssim \| u_2 \|_{L^{\infty}(T_{1/2})}
    \lesssim \| u \|_{L^2(\Tref)}.
  \end{align*}
  A triangle inequality shows that $u_3$ fulfills the statement of Lemma~\ref{lemma:vertex-part},
  (\ref{item:lemma:vertex-part-8}). To show~\eqref{item:lemma:vertex-part-edge-vanish-4} we note
  \begin{align*}
    \| d_{\Vref}^{\gamma} u_3 \|_{L^2(\Tref)} 
    &\lesssim \| d_{\Vref}^{\gamma} u_2  \|_{L^2(\Tref)} + \| u_2 \|_{L^\infty(T_{1/2})}
    \lesssim \| d_{\Vref}^{\gamma} u_2  \|_{L^2(\Tref)} + \norm{u}_{L^2(\Tref)} \\
    &\lesssim \| d_{\Vref}^\gamma \A^{\Vref}u \|_{L^2(\Tref)}
    + \| \A^{\Vref} u \|_{L^\infty(T_{1/2})}  + \norm{u}_{L^2(\Tref)} 
    \lesssim \| d_{\Vref}^\gamma u \|_{L^2(\Tref)},
  \end{align*}
  where we used~\eqref{lemma:vertex-part-edge-vanish:eq1000} for the second estimate and
  Lemma~\ref{lemma:vertex-part} for the last one. Finally, we set 
  $$
  \A^{{\Vref}}_0 u(\xi,\eta):=u_4(\xi,\eta):= \frac{1}{2} \left( u_3(\xi,\eta) + u_3(\xi,\xi - \eta)\right). 
  $$
  If $\Vref=(1,0)$ we take the average of $u_3(\xi,\eta)$ and $u_3(1-\eta,1-\xi)$ instead. 

  Then, $u_4$ fulfills the same bounds as $u_3$ and satisfies (\ref{item:lemma:vertex-part-edge-vanish-8})
  and (\ref{item:lemma:vertex-part-edge-vanish-9}).
  It is easy to see that $u_4(\widehat V) = u(\widehat V)$ and $u_4$ is continuous at $\widehat V$.
  Finally, inspection shows that if $u$ is a polynomial of degree $p \ge 1$, 
  then $u_4$ is a polynomial of degree $p$.
  In order to prove~\eqref{item:lemma:vertex-part-edge-vanish-10}, we fix a smooth cut-off function $\phi$
  that equals $1$ in a neighborhood of $V$ and has support in a neighborhood of $V$ excluding the other vertices.
  Then,
  \begin{align*}
    \| d_{V}^{-1/2} \A^{\Vref}_0 u \|_{L^2(\Tref)}
    \leq 
    \| d_{V}^{-1/2} \phi \A^{\Vref}_0 u \|_{L^2(\Tref)} + \| d_{V}^{-1/2} (1-\phi)\A^{\Vref}_0 u \|_{L^2(\Tref)}. 
  \end{align*}
  Due to the support properties of $\phi$, the second term on the right-hand side is bounded by
  $\| \A^{\Vref}_0 u \|_{L^2(\Tref)}$ and consequently by $\| u \|_{L^2(\Tref)}$. The first term on the right-hand side
  is bounded by $\| \phi \A^{\Vref}_0 u \|_{L^\infty(\Tref)}$ due to Lemma~\ref{lemma:weighted-infty}, and by the support properties of $\phi$
  and~\eqref{item:lemma:vertex-part-edge-vanish-6} finally by $\| u \|_{L^2(\Tref)}$.
  The estimate involving the gradient is shown analogously.
\end{proof}
\subsubsection{Averaging operators on squares for  the vertex parts}
The vertex averaging operator $\A_{\Sref}$ (see Lemma~\ref{lemma:vertex_averaging_quad_to_trig} below) 
for the square is obtained from an averaging operator based on a triangle contained in $\Sref$ and then extended
to the full square by a Duffy transformation. Prior to defining $\A_{\Sref}$, we therefore study the Duffy transformation 
in the following Lemma~\ref{lemma:duffy_transform}.

\begin{lemma}
  \label{lemma:duffy_transform}
  We introduce the Duffy transform
  $$\displaystyle T_{\mathcal{D}}:\quad \Sref \to \Tref, \quad
    \colvec{\xi\\ \eta} \mapsto \colvec{\eta (1-\xi) + \xi\\ \eta}$$
    and the corresponding linear operator as $\mathcal{D}:u \mapsto u \circ T_{\mathcal{D}}$. 
    Then the following holds:
  \begin{enumerate}[(i)]
    \item
      \label{lemma:duffy_transform-1}
      $\mathcal{D}$ is bounded in the norms
      \begin{align*}
	\| \mathcal{D} u \|_{L^2(\Sref)} &\lesssim \| d_{(1,1)}^{-1/2} u\|_{L^2(\Tref)},\\
	\| \mathcal{D} u \|_{H^1(\Sref)} &\lesssim \| d_{(1,1)}^{-1/2} u\|_{L^2(\Tref)} + \|d_{(1,1)}^{-1/2} \nabla u \|_{L^2(\Tref)}.
      \end{align*}
    \item
      \label{lemma:duffy_transform-2}
      If $u \in \PP(\Tref)$, then $\mathcal{D}u \in \QQ(\Sref)$.
    \item
      \label{lemma:duffy_transform-3}
      If $u$ vanishes on the edge $\widehat{e}_2:=\{(x,x), \; 0< x < 1\}$, then
      $\mathcal{D}u$ vanishes on the edges $e_2:=(0,1) \times \{1\}$ and $e_3:=\{0\}\times (0,1)$,
    \item \label{lemma:duffy_transform-4}
      The values on the edges $e_0:=(0,1) \times \{0\}$ and $e_1:=\{1\} \times (0,1)$ are preserved, i.e., 
      $\mathcal{D}u (\xi,0) = u(\xi ,0)$ and $\mathcal{D}u (1,\eta)=u(1,\eta)$.
    \item \label{lemma:duffy_transform-5}
      Close to the vertex $\widehat{V} := (1,0)$, the Duffy transform is almost the identity in the sense 
      that, provided that the right-hand side is finite, 
      \begin{align}
        \label{eq:lemma:duffy_transform-5}
        \| d_{\widehat{V}}^{-\theta}\left(u-\mathcal{D}u \right)\|_{L^2(\Tref)}&\leq \| u \|_{H^{\theta}(\Tref)}
	+ \| d_{(1,1)}^{-1/2} u\|_{L^2(\Tref)}.
      \end{align}
    \item
      \label{lemma:duffy_transform-6}
      For $\widehat{V} := (1,0)$ and $\gamma \in \R$ there holds 
      \begin{align}
        \label{eq:lemma:duffy_transform-6}
        \| d_{\widehat{V}}^{\gamma} \mathcal{D}u \|_{L^2(\Sref)}
	\leq  \| d_{\widehat{V}}^{\gamma} u \|_{L^2(\Tref)} + \| d_{(1,1)}^{-1/2} u\|_{L^2(\Tref)}. 
      \end{align}
  \end{enumerate} 
\end{lemma}
\begin{proof}
  {\bf Proof of~(\ref{lemma:duffy_transform-1}):} The Jacobi matrix of $T_{\mathcal{D}}$ is
  $dT_{\mathcal{D}}=\left(\begin{smallmatrix} 1-\eta  & 1-\xi\\ 0  & 1\end{smallmatrix}\right)$,
  hence $\abs{\det dT_{\mathcal{D}}} = 1-\eta = 1-y$. Transforming the integral we thus pick up the factor
  $(1-y)^{-1} \sim d_{(1,1)}^{-1}$.

  {\bf Proof of~(\ref{lemma:duffy_transform-2}):} Let $u=\sum_{i=0,j=0}^{p}{\alpha_{i,j} x^i y^{j}}$ with $i+j\leq p$.
  Inserting the definition of $\mathcal{D}$ gives:
  $\mathcal{D}u=\sum_{i=0,j=0}^{p}{\alpha_{i,j} (\eta(1-\xi) + \xi)^i \eta^j}$. Expanding the powers and inspecting
  the highest polynomial degrees, we observe that
  the leading term has the form $\beta_{i,j}\; \eta^i\xi^i \eta^j=\beta_{i,j} \eta^{i+j}\xi^i$. Since $i+j\leq p$ we have $\mathcal{D}u \in \QQ$.

  {\bf Proof of~(\ref{lemma:duffy_transform-3}) and (\ref{lemma:duffy_transform-4}):} The claims follow by inspection.

  {\bf Proof of~(\ref{lemma:duffy_transform-5}):}
  Fix a smooth cut-off function $\phi$ that equals $1$ in a neighborhood of the vertex $(1,1)$ and whose 
support excludes neighborhoods of the other 2 vertices. Then
  \begin{align}\label{lemma:duffy_transform:eq1}
    \| d_{\widehat{V}}^{-\theta}\left(u-\mathcal{D}u \right)\|_{L^2(\Tref)}
    \leq \| d_{\widehat{V}}^{-\theta}\left(\phi u-\mathcal{D}(\phi u) \right)\|_{L^2(\Tref)} + 
    \| d_{\widehat{V}}^{-\theta}\left( (1-\phi)u-\mathcal{D}( (1-\phi)u ) \right)\|_{L^2(\Tref)}. 
  \end{align}
  The first term of the right-hand side can be estimated by
  \begin{align*}
    \| d_{\widehat{V}}^{-\theta}\left(\phi u-\mathcal{D}(\phi u) \right)\|_{L^2(\Tref)}
    \lesssim \| \phi u-\mathcal{D}(\phi u) \|_{L^2(\Tref)}
    \lesssim \| d_{(1,1)}^{-1/2} \phi u\|_{L^2(\Tref)} \lesssim \| d_{(1,1)}^{-1/2} u\|_{L^2(\Tref)},
  \end{align*}
  where we use the support properties of $\phi$ in the first step and~\eqref{lemma:duffy_transform-1} in the second.
  For the second term on the right-hand side of~\eqref{lemma:duffy_transform:eq1} note that for $\theta=0$
  we obtain with the properties of $\phi$
  \begin{align}\label{lemma:duffy_transform:eq2}
    \| (1-\phi)u-\mathcal{D}( (1-\phi)u ) \|_{L^2(\Tref)} \lesssim \| d_{(1,1)}^{-1/2} (1-\phi)u \|_{L^2(\Tref)}
    \lesssim \| u \|_{L^2(\Tref)}.
  \end{align}
  For $\theta=1$, note that $d_{\widehat V}\geq 1-\xi$ and calculate using Hardy's inequality
  \begin{align*}
    \begin{split}
    \int_0^1 \int_\eta^1 d_{\widehat V}^{-2} ( v(\eta(1-\xi)+\xi,\eta) -v(\xi,\eta) )^2 \,d\xi d\eta
    &\leq \int_0^1 \int_\eta^1 \left( (1-\xi)^{-1} \int_{\xi}^{1} \abs{\partial_1 v(s,\eta)} ds \right)^2\,d\xi d\eta\\
    &\lesssim \int_0^1\int_\eta^1\abs{\partial_1 v(\xi,\eta)}^2\,d\xi d\eta.
    \end{split}
  \end{align*}
  Applying this to $v=(1-\phi)u$ shows
  \begin{align*}
    \| d_{\widehat{V}}^{-1}\left( (1-\phi)u-\mathcal{D}( (1-\phi)u ) \right)\|_{L^2(\Tref)} \lesssim \abs{(1-\phi)u}_{H^1(\Tref)}
    \leq \| u \|_{H^1(\Tref)}.
  \end{align*}
  Interpolating this with~\eqref{lemma:duffy_transform:eq2} allows us to bound the
  second term on the right-hand side of~\eqref{lemma:duffy_transform:eq1}.

  {\bf Proof of~(\ref{lemma:duffy_transform-6}):}
  Follows by transforming the integral and using $d_{{\widehat{V}}} \circ T_{\mathcal{D}} \sim d_{{\widehat V}}$.
\end{proof}
When switching between polynomials on the reference square and reference triangle,
the difference in the definition of $\QQ$ and $\PP$
leads to an increase of the degree by a factor of $2$.
The main tool to correct the polynomial degrees on rectangles will be the Gau{\ss}-Lobatto
interpolation operator. We collect its properties in the following lemma.
\begin{lemma}\label{lemma:GL}
\begin{enumerate}[(i)]
  \item \label{item:lemma:GL-i}
    There holds for $\theta \in (0,1)$ and $p\in\N_0$, and $E$ a collection of edges of $\Sref$:
    \begin{align*}
    \left({\mathcal Q}_p,\|\cdot\|_{L^2(\Sref)},{\mathcal Q}_p,\|\cdot\|_{H^1(\Sref)}\right)_{\theta,2} 
    & = \left({\mathcal Q}_p,\|\cdot\|_{H^\theta(\Sref)}\right) 
    \qquad \mbox{(equivalent norms)}, \\
    \left(\widetilde{\mathcal Q}_p(\Sref,E),\|\cdot\|_{L^2(\Sref)},\widetilde{\mathcal Q}_p(\Sref,E),\|\cdot\|_{H^1(\Sref)}\right)_{\theta,2} 
    & = \left(\widetilde{\mathcal Q}_p(\Sref,E),\|\cdot\|_{\widetilde{H}^\theta(\Sref,E)}\right) 
      \qquad \mbox{(equivalent norms)},      
    \end{align*}
    where $\widetilde{\mathcal Q}_p(\Sref,E)\subset \QQ$ denotes the polynomials vanishing on $E$,
      and the $\widetilde{H}^\theta(\Sref,E)$-norm is defined via interpolation of $L^2$ and $H^1(\Sref)\cap \{u : \, u|_{E}=0\}$.
    The constants in the norm equivalences do not depend on $p$.
  \item \label{item:lemma:GL-ii}
    Let $i_p: C(\overline{\Sref}) \rightarrow {\mathcal Q}_p$ be the tensor-product
    Gau{\ss}-Lobatto interpolation
    operator. Then for every $\theta \in [0,1]$ there exists $C > 0$ such that for
    all $p$, $q \in \N_0$ the following stability estimate holds for the operator $i_p$: 
    $$
    \|i_p\|_{({\mathcal Q}_{p},\|\cdot\|_{H^\theta(\Sref)}) \leftarrow ({\mathcal Q}_{q},\|\cdot\|_{H^\theta(\Sref)})}
    \leq C ( 1 + q/(p+1))^{2-\theta}
  $$
  \item \label{item:lemma:GL-iii}
    Let $\widehat{V}$ be a vertex of $\Sref$ and set $d_{\widehat{V}}:=\operatorname{dist}(\cdot,\widehat{V})$.
    Then there exists a constant $C >0$, such that for all $p$, $q \in \N_0$,
    the following estimate holds:
    \begin{align*}
      \norm{d^{-\theta}_{\widehat{V}} \left(u-i_p u\right)}_{L^2(\Sref)}
      \leq C \left(1+q/(p+1) \right)^{2(1-\theta)}\left(q/p\right)^{\theta}\norm{u}_{H^\theta(\Sref)}
      \quad \forall u \in \mathcal{Q}_q(\Sref).
    \end{align*}
\end{enumerate}
\end{lemma}
\begin{proof}
  Statement (\ref{item:lemma:GL-i}) is the assertion of \cite[Theorem 6]{belgacem94}. For the treatment of boundary condition, see also Proposition 5
    and its preceding remark.
  To show (\ref{item:lemma:GL-ii}), the key observation is that the 1D Gau{\ss}-Lobatto interpolation
  operator $i^{GL}_p$ is stable in $H^1(-1,1)$ by~\cite[(13.27)]{bernardi-maday97} and also satisfies
  the stability estimate $\|i^{GL}_p u\|_{L^2(-1,1)} \leq C (1 + q/(p+1)) \|u\|_{L^2(-1,1)}$ for all 
  $u \in {\mathcal P}_q$ by \cite[Rem.~{13.5}]{bernardi-maday97}.
  Tensor product arguments then give for all $u \in {\mathcal Q}_q$ the estimates
  $\|i_p u\|_{L^2(\Sref)} \lesssim (1 + q/(p+1))^2 \|u\|_{L^2(\Sref)}$ and 
  $\|i_p u\|_{H^1(\Sref)} \lesssim (1 + q/(p+1)) \|u\|_{H^1(\Sref)}$. Interpolation, which is possible
  due to (\ref{item:lemma:GL-i}), allows us to conclude the proof.
  To show (\ref{item:lemma:GL-iii}), note first that for $\theta=0$ the statement is equivalent
  to the $L^2$-stability of the Gau{\ss}-Lobatto interpolation.
  For $\theta=1$, we note that by Lemma~\ref{lemma:weighted-norm-equivalence}, it
  is sufficient to bound $\| d_{\widehat{V}}^{-1/2} \left(u-i_p u\right) \|_{L^2(\widehat{e})}$ for the two edges emanating
  from the vertex $\widehat{V}$. On such an edge $\widehat{e}$, the tensor product operator coincides with the 
  1D Gau{\ss}-Lobatto interpolation operator. 
  We combine the approximation estimate \cite[Theorem 13.4]{bernardi-maday97}, an inverse estimate
  (see \cite[Theorem 3.91]{schwab_p_fem} for the $H^1$-$L^2$ case, the $H^1$-$H^{1/2}$ case follows
  by interpolation) and a trace estimate to get 
  \begin{align*}
    \| d_{\widehat{V}}^{-1/2} \left(u-i_p u\right) \|_{L^2(\widehat{e})}
    &\lesssim \| \operatorname{dist}(\cdot,\partial \widehat{e})^{-1/2}\left(u-i_p u\right) \|_{L^2(\widehat{e})}
      \lesssim \frac{1}{p}\norm{u}_{H^1(\widehat{e})} \lesssim \frac{q}{p} \norm{u}_{H^{1/2}(\widehat{e})}
    \lesssim \frac{q}{p}\norm{u}_{H^1(\Sref)}.
  \end{align*}
  The general statement then follows from interpolation and Proposition~\ref{prop:interpolation_weighted_l2}.
\end{proof}
\begin{lemma}
  \label{lemma:vertex_averaging_quad_to_trig}
  There exists an operator $\A_{\Sref}: L^1_{\rm loc}(\Sref)\to C^\infty(\Tref)$ such that  
  \begin{enumerate}[(i)]
    \item \label{item:lemma:averaging_quad_to_trig-1}
      $\A_{\Sref}:H^\theta(\Sref)\rightarrow H^\theta(\Tref)$ is linear and bounded
      for all $\theta \in [0,1]$.
    \item \label{item:lemma:averaging_quad_to_trig-2}      
      $\A_{\Sref}$ reproduces the value at $(1,0)$, i.e.,
      $\left(\A_{\Sref}u \right)(1,0) = u(1,0)$ if $u$ is continuous at $(1,0)$. 
    \item \label{item:lemma:averaging_quad_to_trig-3}
      $\| d_{(1,0)}^{-\theta}\left(u - \A_{\Sref} u\right)\|_{L^2(\Tref)} \lesssim \norm{u}_{H^{\theta}(\Sref)}$.
    \item \label{item:lemma:averaging_quad_to_trig-4}
      $\A_{\Sref} u$ vanishes on the edge of $\Tref$ opposite to $(1,0)$.
    \item \label{item:lemma:averaging_quad_to_trig-5}
      Let $e_1$, $e_2$ be the edges of $\Sref$ with $(1,0) \in \overline{e_1} \cap \overline{e_2}$, 
      and let $F_{e_1}^{e_2}$ be the affine map, mapping $e_1$ to $e_2$ with $F(1,0)=(1,0)$.
      Then the following holds for $\hat{u}:=\A_{\Sref} u$:
      \begin{align*}
        \hat{u}(z)=\hat{u} \circ F_{e_1}^{e_2}(z) \quad \forall z \in e_1.
      \end{align*}
    \item \label{item:lemma:averaging_quad_to_trig-6}
      If $u\in \QQ(\Sref)$, then $\A_{\Sref} u\in\P_{2p}(\Tref)$.
    \item\label{item:lemma:averaging_quad_to_trig-7}
      Let $V \neq (1,0)$ be a vertex of $\Tref$. Then
      \begin{align*}
        \| d_{V}^{-1/2} \A_{\Sref} u \|_{L^2(\Tref)} \lesssim \| u \|_{L^2(\Sref)}
        \quad\text{ and }\quad
        \| d_{V}^{-1/2}\nabla \A_{\Sref} u \|_{L^2(\Tref)} \lesssim \| u \|_{H^1(\Sref)}.
      \end{align*}
  \end{enumerate}
\end{lemma}
\begin{proof}  
  We define the operator $\widetilde{\A}_{\Sref}$ as
  \begin{align*}
    \widetilde\A_{\Sref}:=\A^{(1,0)}_0 \circ \mathcal{R}_{\Tref}
  \end{align*}
  where $\mathcal{R}_{\Tref}$ denotes the restriction operator to the triangle $\Tref$
  and $\A^{(1,0)}_0$ is the operator from Lemma~\ref{lemma:vertex-part-edge-vanish} with $\widehat V=(1,0)$.
  As $\mathcal{R}_{\Tref} \Q_p(\Sref) \subset \P_{2p}(\Tref)$, we
  obtain~(\ref{item:lemma:averaging_quad_to_trig-6}).
  The properties~(\ref{item:lemma:averaging_quad_to_trig-1})
  and~(\ref{item:lemma:averaging_quad_to_trig-2}) are direct
  consequence of Lemma~\ref{lemma:vertex-part-edge-vanish} and
  the properties of the restriction operator.
  For (\ref{item:lemma:averaging_quad_to_trig-3}), we note that on $\Tref$
  the restriction operator is the identity.
  Lemma~\ref{lemma:vertex-part-edge-vanish} then yields
  \begin{align}\label{lemma:vertex_averaging_quad_to_trig:1}
    \begin{split}
    \| d_{(1,0)}^{-\theta}(u - \widetilde\A_{\Sref} u)\|_{L^2(\Tref)}
    \lesssim \norm{\mathcal{R}_{\Tref}u}_{H^\theta(\Tref)}
    \lesssim \norm{u}_{H^\theta(\Sref)}.
    \end{split}
  \end{align}
  Lemma~\ref{lemma:vertex-part-edge-vanish} also yields (\ref{item:lemma:averaging_quad_to_trig-4}).
  What is left to do, is to ensure~(\ref{item:lemma:averaging_quad_to_trig-5}).
  To that end, we introduce the notation $v^{\rm flip}(x,y) = v(1-y,1-x)$ 
  and set 
  \begin{align*}
    \A_{\Sref} u:=
    \frac{ \widetilde{\A}_{\Sref} u + (\widetilde{\A}_{\Sref} u)^{\rm flip}}{2}.
  \end{align*}
  The operator $\A_{\Sref}$ clearly fulfills~(\ref{item:lemma:averaging_quad_to_trig-1}),~(\ref{item:lemma:averaging_quad_to_trig-2}),
  and~(\ref{item:lemma:averaging_quad_to_trig-4}). To  show~(\ref{item:lemma:averaging_quad_to_trig-3}),
  note first that
  \begin{align*}
    \| d_{(1,0)}^{-\theta}(u - \A_{\Sref}u)\|_{L^2(\Tref)}
    &\lesssim \| d_{(1,0)}^{-\theta}(u - \widetilde{\A}_{\Sref}u)\|_{L^2(\Tref)}
    + \| d_{(1,0)}^{-\theta}(u - (\widetilde{\A}_{\Sref}u)^{\rm flip})\|_{L^2(\Tref)}\\
    &= \| d_{(1,0)}^{-\theta}(u - \widetilde{\A}_{\Sref}u)\|_{L^2(\Tref)}
    +\| d_{(1,0)}^{-\theta}(u^{\rm flip} - \widetilde{\A}_{\Sref}u)\|_{L^2(\Tref)}\\
    &\lesssim \| d_{(1,0)}^{-\theta}(u - \widetilde{\A}_{\Sref}u)\|_{L^2(\Tref)}
    + \| d_{(1,0)}^{-\theta} (u-u^{\rm flip}) \|_{L^2(\Tref)}.
  \end{align*}
  The first term on the right-hand side is bounded by~\eqref{lemma:vertex_averaging_quad_to_trig:1},
  and it suffices to bound the second term. As
  \begin{align*}
    \abs{(x,y)-(1-y,1-x)} \leq 2 d_{(1,0)}(x,y),
  \end{align*}
  a simple argument based on the fundamental theorem of calculus and Hardy's inequality shows that
  \begin{align*}
    \| d_{(1,0)}^{-\theta} (u-u^{\rm flip}) \|_{L^2(\Tref)} \leq \| u \|_{H^\theta(\Tref)}
  \end{align*}
  for $\theta=1$. The same estimate for $\theta=0$ is trivial. An interpolation argument finishes the proof.
  The property~\eqref{item:lemma:averaging_quad_to_trig-7} also follows from 
  Lemma~\ref{lemma:vertex-part-edge-vanish}. 
\end{proof}
In order to preserve the polynomial degree also for quadrilaterals, one can modify the construction of $\A_{\Sref}$.
  The only concession one needs to make is the restriction of the domain of definition.
\begin{corollary}\label{cor:vertex_averaging_quad_to_trig}
  If $p\geq 2$, one can define $\A_{\Sref}^p: \QQ(\Sref) \to \PP(\Tref)$,
  preserving the  properties from Lemma~\ref{lemma:vertex_averaging_quad_to_trig}.
\end{corollary}
\begin{proof}
  We can mimic the proof of Lemma~\ref{lemma:vertex_averaging_quad_to_trig}, defining
  \begin{align*}
    \widetilde\A_{\Sref}^p:=\A^{(1,0)}_0 \circ \mathcal{R}_{\Tref} \circ i_{\lfloor p/2\rfloor}.
  \end{align*}
  As $\mathcal{Q}_{\lfloor p/2\rfloor}(\Tref) \subseteq \PP(\Tref)$, we see that $\A_{\Sref}^p$
  maps $\QQ(\Sref)$ to $\PP(\Tref)$. To show boundedness of $\A_{\Sref}$, we additionally use
  the boundedness of the Gau{\ss}-Lobatto interpolation operator (see Lemma~\ref{lemma:GL}).
  Note also that $(1,0)$ is a Gau{\ss}-Lobatto node so that $i_{\lfloor p/2 \rfloor}$ reproduces
  point values.
  Estimate~\eqref{lemma:vertex_averaging_quad_to_trig:1} then becomes
  \begin{align*}
    \| d_{(1,0)}^{-\theta}\left(u - \widetilde\A_{\Sref}^p u\right)\|_{L^2(\Tref)}
    &= \;\;\|d_{(1,0)}^{-\theta}\left(u -
    \A^{(1,0)}_0 \circ \mathcal{R}_{\Tref} \circ i_{\lfloor p/2\rfloor} u\right)\|_{L^2(\Tref)} \\
    &\leq \;\; \|d_{(1,0)}^{-\theta}\left(u - \A^{(1,0)}_0 u \right)\|_{L^2(\Tref)} +
    \|d_{(1,0)}^{-\theta} \A^{(1,0)}_0 \left( u - i_{\lfloor p/2\rfloor}u \right)\|_{L^2(\Tref)} \\
    &\stackrel{\mathclap{\text{Lem.~\ref{lemma:vertex-part-edge-vanish}}}}{\lesssim} \;\; \;\norm{u}_{H^\theta(\Tref)} + \norm{d_{(1,0)}^{-\theta} \left( u - i_{\lfloor p/2\rfloor}u \right)}_{L^2(\Tref)}
    \stackrel{\text{Lem.~\ref{lemma:GL}, (\ref{item:lemma:GL-iii})}}{\lesssim} \norm{u}_{H^\theta(\Sref)}.
\qedhere    
  \end{align*}
\end{proof}

\subsection{Averaging operators for the edge parts}
In this section, we develop averaging operators on the reference triangle that reproduce
  values on one of the edges and produces functions which vanish on the others. They will
  be crucial for defining the edge contributions of the decomposition.
We will need an averaging operator for edges.
\begin{lemma}\label{lemma:edge-part}
  Let $\widehat e$ be an edge of $\Tref$.
  Then there exists a linear operator
  $\A^{\widehat e}:L^1_{loc}(\Tref) \rightarrow C^\infty(\Tref)$ with the following properties: 
\begin{enumerate}[(i)]
  \item \label{item:lemma:edge-part-1}
    If $u$ is a polynomial of degree $p\ge 0$, then $\A^{\widehat e} u$ is a polynomial
    of degree $p$. 
  \item \label{item:lemma:edge-part-2}
    If $u$ is continuous at a point $z \in \widehat e$, then $(\A^{\widehat e} u)(z) = u(z)$. 
  \item\label{item:lemma:edge-part-3}
    $\A^{\widehat{e}}: H^{\theta}(\Tref) \rightarrow H^{\theta}(\Tref)$ is bounded for all $\theta \in [0,1]$.
  \item \label{item:lemma:edge-part-4}
    For every $T^\prime \subset \Tref$ with $\operatorname*{dist}(T^\prime,\overline{\widehat e}) > 0$
    there exists $C_{T^\prime}$ such that
    $\|\A^{\widehat e} u\|_{L^\infty(T^\prime)} \leq C_{T^\prime}\|u\|_{L^2(\Tref)}$. 
  \item\label{item:lemma:edge-part-5}
Let $e_1$, $e_2$  be the two other edges of $\widehat T$. 
Introduce for the two endpoints $V_1$, $V_2$ of $\widehat e$ 
the distance functions $d_i:= d_{V_i}$.
    Then for  $i=1$, $2$: 
    \begin{align}
      \|d_i^{1/2} {\mathcal A}^{\widehat e} u\|_{L^2(e_i)} &\lesssim \|u\|_{L^2(\Tref)}, 
      \label{eq:item:lemma:edge-part-5} \\
      \|d_i^{-1/2} {\mathcal A}^{\widehat e} u\|_{L^2(e_i)}  
      &\lesssim  \|d_i^{-1} u\|_{L^2(\Tref)}, \label{eq:item:lemma:edge-part-6}\\
      \|d_i^{1/2} \nabla {\mathcal A}^{\widehat e} u\|_{L^2(e_i)} &\lesssim \|\nabla u\|_{L^2(\Tref)}.
      \label{eq:item:lemma:edge-part-8}
    \end{align}
  \item\label{item:lemma:edge-part-7}
    There holds 
    $$
    \|d_{\widehat e}^{-1} (u - \A^{\widehat e} u)\|_{L^2(\Tref)} \lesssim \|\nabla u\|_{L^2(\Tref)}. 
    $$
\end{enumerate}
\end{lemma}
\begin{proof}
Structurally, the procedure is similar to that in the proof of Lemma~\ref{lemma:vertex-part}. 
To keep the notation simple, we use the reference triangle
$\mTref:=\{(\xi,\eta)\,:\, -1 < \xi < 1, \quad 0 <  \eta < 1 - |\xi|\}$ and 
suppose $\widehat e:=(-1,1)\times\{0\}$.
Let $M = \R \times \{0\}\supset \widehat e$. Select $\rho$ with $\operatorname*{supp} \rho \subset \alpha (\mTref - (0,1))$ for 
some $\alpha \in (0,1/\sqrt{2})$. Note that this implies $\operatorname*{supp} \rho \subset B_1(0)$ 
and $\operatorname*{supp} \rho_{d_M(x)}(x - \cdot) \subset \mTref$ for any $x \in \mTref$. 
Again, with the aid of Stein's extension operator \cite[Sec.~{VI}]{stein70} we may assume 
that $u$ is defined on $\R^d$.  
Items (\ref{item:lemma:edge-part-1})--(\ref{item:lemma:edge-part-4}) and (\ref{item:lemma:edge-part-7})
follow from Lemma~\ref{lemma:general-averaging}.
 In order to derive Item (\ref{item:lemma:edge-part-5}),
localized versions of the trace inequalities and covering arguments show for sufficiently smooth functions
$v$
\begin{align}
\|d_i^{1/2} v\|_{L^2(\widehat e_i)} & \lesssim \| v\|_{L^2(\widetilde{T})} + \|d_{\widehat{e}} \nabla v\|_{L^2(\widetilde T)}, \\
\|d_i^{-1/2} v\|_{L^2(\widehat e_i)} & \lesssim \|d_{\widehat{e}}^{-1} v\|_{L^2(\widetilde T)} + \|\nabla v\|_{L^2(\widetilde T)}, \\
\|d_i^{1/2} \nabla v\|_{L^2(\widehat e_i)} & \lesssim \|\nabla v\|_{L^2(\widetilde T)} + \|d_{\widehat{e}} \nabla^2 v\|_{L^2(\widetilde T)}.
\end{align}
Applying the corresponding estimates from
Lemma~\ref{lemma:general-averaging} finishes the proof.
\end{proof}
The operator ${\mathcal A}^{\widehat e}$ can be modified so as to enforce
homogeneous boundary conditions on 
$\partial\widehat T\setminus \widehat e$: 
\begin{lemma}\label{lemma:edge-part-homogeneous-bc}
  Let $d_1$ and $d_2$ be the distances from the two vertices of $\widehat e$.
  There exists a linear operator ${\mathcal A}^{\widehat e}_0:L^1_{loc}(\Tref) \rightarrow C^\infty(\Tref)$ with the 
  following properties: 
\begin{enumerate}[(i)]
\item
\label{item:lemma:edge-part-homogeneous-bc-1}
${\mathcal A}^{\widehat e}_0$ is bounded in $L^2(\Tref)$. 
\item
\label{item:lemma:edge-part-homogeneous-bc-2} 
For each $\theta \in [0,1]$ there is $C>0$ such that, provided the right-hand side is finite,
$$
\|{\mathcal A}^{\widehat e}_0 u\|_{H^\theta(\Tref)} \leq C \left[ \|u\|_{H^\theta(\Tref)} + 
\|d_1^{-\theta} u\|_{L^2(\Tref)} + 
\|d_2^{-\theta} u\|_{L^2(\Tref)}\right]. 
$$
\item 
\label{item:lemma:edge-part-homogeneous-bc-3}
$({\mathcal A}^{\widehat e}_0 u)|_{\partial\Tref\setminus \widehat e} = 0$. 
\item
\label{item:lemma:edge-part-homogeneous-bc-4}
If $u$ is continuous at $z \in \widehat e$, then $({\mathcal A}^{\widehat e}_0 u)(z) = u(z)$. 
\item 
\label{item:lemma:edge-part-homogeneous-bc-5}
If $u$ is a polynomial degree $p$ {\em and} $u$ vanishes in the endpoints of $\widehat e$, then 
${\mathcal A}^{\widehat e}_0 u$ is a polynomial of degree $p$. 

\item 
\label{item:lemma:edge-part-homogeneous-bc-6}
It holds
$$
\|d_{\widehat e}^{-\theta} (u - {\mathcal A}^{\widehat e}_0 u)\|_{L^2(\Tref)} \leq C \|u\|_{H^{\theta}(\Tref)} 
+\|d_1^{-\theta} u\|_{L^2(\Tref)}  
+\|d_2^{-\theta} u\|_{L^2(\Tref)}.
$$

\item\label{item:lemma:edge-part-homogeneous-bc-7}
  Let $\widehat V$ be the vertex opposite of $\widehat e$. Then
  \begin{align*}
    \| d_{\widehat V}^{-1/2} \A^{\widehat e}_0 u \|_{L^2(\Tref)} &\lesssim \| u \|_{L^2(\Tref)}  \quad \text{ and }\\    
    \| d_{\widehat V}^{-1/2}\nabla \A^{\widehat e}_0 u \|_{L^2(\Tref)} &\lesssim \| u \|_{H^1(\Tref)}
    + \|d_1^{-1} u\|_{L^2(\Tref)} + \|d_2^{-1} u\|_{L^2(\Tref)}.
  \end{align*}
\end{enumerate}
\end{lemma}
\begin{proof}
  We assume $\widehat V = (1,1)$ and $\widehat e = \{(x,0)\,|\, 0 < x < 1\}$.
  We construct ${\mathcal A}^{\widehat e}_0$ explicitly with the aid of ${\mathcal A}^{\widehat e}$. 
  For simplicity of notation, we write $u_1:= {\mathcal A}^{\widehat e} u$. As a first step, we set
$u_2:= u_1 - y u_1(1,1)$. Then $u_2$ vanishes in $(1,1)$ and $u_2$ has the desired stability properties
in $L^2$ and $H^1$ by those of ${\mathcal A}^{\widehat e} u$ given in Lemma~\ref{lemma:edge-part}. Next, we correct 
the behavior of $u_2$ on the edge $\widehat e_1 = \{(x,x)\,|\, x \in (0,1)\}$. We set 
$$
u_3(\xi,\eta):= u_2(\xi,\eta) - \frac{\eta}{\xi} u_2(\xi,\xi). 
$$
Then $u_3$ coincides with $u_2$ on $\widehat e$, it coincides with $u_2$ on $\widehat e_2 = \{(1,\eta)\,|\,0 < \eta < 1\}$ 
in view of $u_2(1,1) = 0$, and it vanishes on $\widehat e_1$. When trying to control the $L^2$-norm of $u_3$ the 
correction part of $u_3$ leads to having to control 
$\|d_1^{1/2} u_2 \|_{L^2(\widehat e_1)}$, where $d_1$ denotes the distance from $V_1 = (0,0)$
This is estimated 
by $\|u\|_{L^2(\Tref)}$ in view of \eqref{eq:item:lemma:edge-part-5} in Lemma~\ref{lemma:edge-part}. 
For the $H^1$-norm, the correction part of $u_3$ leads to several contributions which can all be
controlled by Lemma~\ref{lemma:edge-part},~(\ref{item:lemma:edge-part-5}). In particular, 
it is responsible for a term $\|d_1^{-1/2} u_2\|_{L^2(\widehat e_1)}$, which 
can be estimated by $\|d_1^{-1} u\|_{L^2(\Tref)}$ by \eqref{eq:item:lemma:edge-part-6}, 
and a term $\|d_1^{1/2} \nabla u_2\|_{L^2(\widehat e_1)}$, which is controlled by \eqref{eq:item:lemma:edge-part-8}.
The intermediate cases $\theta \in (0,1)$ follow by interpolation using Lemma~\ref{lemma:equivalence-sobolev-plus-weight}.
We still need to show the weighted $L^2$ estimate  (\ref{item:lemma:edge-part-homogeneous-bc-6}). We again restrict ourselves to the case $\theta=1$,
the general result follows via interpolation using Lemma~\ref{lemma:equivalence-sobolev-plus-weight}.
Since $u_2$ satisfies this estimate, we only need to estimate the correction term. A simple calculation
yields a term of the form $\|d_1^{-1/2} u_2\|_{L^2(\widehat e_1)}$, which can again
be controlled by $\|d_1^{-1} u\|_{L^2(\Tref)}$ via (\ref{eq:item:lemma:edge-part-8}). 
Next, a correction for the edge $e_2$ is performed that is completely analogous to that for the edge $\widehat e_1$. 
This produces the final function $u_4 =: {\mathcal A}^{\widehat e}_0 u$. 
It remains to see that $u_4$ is a polynomial if $u$ is a polynomial that vanishes in the two endpoints of $\widehat e$. 
This follows by inspection and the fact that ${\mathcal A}^{\widehat e} u$ is a polynomial if $u$ is a polynomial 
and that $({\mathcal A}^{\widehat e}  u)|_{\widehat e} = u|_{\widehat e}$.
The property~\eqref{item:lemma:edge-part-homogeneous-bc-7} is shown analogously as before.
\end{proof}

\section{Decomposition of FEM spaces}
\label{sec:decomposition-general-triangulations}
In this section, we use the averaging operators
defined on the reference elements to construct the vertex-, edge- and element contributions
on the respective patches. The basic structure is to single out from the patch one element 
on which to perform the averaging. In the discrete setting, this should be the element with lowest 
polynomial degree. The averaged function on that element is then extended to the whole patch
by ``rotation''. 
\subsection{The vertex parts}
\begin{lemma}\label{lemma:averaging-vertex_patch-general}
  Fix $V \in \innervertices$ and let $\omega_V$ be its vertex patch as defined in
  Definition~\ref{def:patches}.
  Then there exists an operator $\AVP:L^2(\Gamma)\to L^2(\omega_V)$ with
  the following properties and implied constants depending only on $\Gamma$ 
  and the shape-regularity of $\TT$:
  \begin{enumerate}[(i)]
    \item \label{lemma:averaging-vertex_patch-general-1}
      For all $\theta \in [0,1]$, there holds 
$\displaystyle 
	\norm{\AVP u}_{\widetilde{H}^\theta_h(\omega_V)}
	\lesssim \norm{u}_{H^\theta_h(\omega_V)}.
$
    \item \label{lemma:averaging-vertex_patch-general-3}
      The following weighted norm estimates hold:
      \begin{align}
	\norm{d_{V}^{-\theta} \left(u-\AVP u\right)}_{L^2(\omega_{V})}
	&\lesssim \norm{u}_{H^\theta_h(\omega_V)},
	\label{eq:lemma:averaging-vertex_patch-general-3-v} \\
	\norm{d_{V'}^{-\theta} \AVP u}_{L^2(\omega_{V})}
	&\lesssim \norm{u}_{H^\theta_h(\omega_V)}
	\qquad \mbox{for vertices $V' \neq V$.}
	\label{eq:lemma:averaging-vertex_patch-general-3-vprime}
      \end{align}
    \item \label{lemma:averaging-vertex_patch-general-4}
      If ${\TT|_{\omega_V}}$ consists of triangles only, 
      then $\AVP: \SS^{\pp,1}(\TT) \to \SS^{\pp,1}(\TT|_{\omega_V})$.
      If ${\TT|_{\omega_V}}$ consists of triangles and quadrilaterals, then
      $\AVP: \SS^{\pp,1}(\TT) \to \SS^{2\pp,1}(\TT|_{\omega_V})$.
  \end{enumerate}
\end{lemma}
\begin{proof}
  Select an element $K\subset \omega_V$ with the lowest polynomial degree.
  Let $\Kref$ be the corresponding reference element with element map $F_K$,
  with (for notational convenience) the additional assumption that $F_K(1,0) = V$.
  Depending on whether $K$ is a triangle $\Kref = \Tref$ or square $\Kref = \Sref$,
  we define
  \begin{align*}
    \widetilde{u}:=
    \begin{cases}
      \A^{(1,0)}_0 \widehat{u}& \text{ if $\widehat{K}=\Tref$}, \\
      \A_{\Sref}  \widehat{u}& \text{ if $\widehat{K}=\Sref$},
    \end{cases}
  \end{align*}
  where $\widehat{u}:=u\circ F_K$ denotes the pullback of $u$ to the reference element
  and $\A_{\Sref}$ is the operator from Lemma~\ref{lemma:vertex_averaging_quad_to_trig}.
  For all elements $K'\subset \omega_V$, let $F_{K'}  : \widehat{K'}  \to K'$ denote the element map,
  with the additional assumption that $F_{K'}(1,0)=V$.
  We define $\AVP$ by ``rotating'' $\widetilde{u}$, i.e.,
  \begin{align*}
    (\AVP u)|_{K'}:=
    \begin{cases}
      \widetilde{u} \circ F_{K'}^{-1}  & \text{if $\widehat{K'}=\Tref$,} \\
      \left(\mathcal{D} \widetilde{ u} \right) \circ F_{K'}^{-1}  & \text{if $\widehat{K'}=\Sref$}.
    \end{cases}
  \end{align*}
  If $K'$ is a triangle, scaling arguments and Lemma~\ref{lemma:vertex-part-edge-vanish} 
(if $K$ is a triangle)
  or Lemma~\ref{lemma:vertex_averaging_quad_to_trig} (if $K$ is a quadrilateral) show for $\theta\in\left\{ 0,1 \right\}$
  \begin{align}\label{eq:averaging_vertex_patch_stability}
    \|\AVP u \|_{H^\theta_h(K')} \lesssim \|u\|_{H^\theta_h(K)}.
  \end{align}
  If $K'$ is a quadrilateral, we additionally use boundedness of the Duffy operator $\mathcal{D}$
  from Lemma~\ref{lemma:duffy_transform},~\eqref{lemma:duffy_transform-1},
  \begin{align}\label{eq:averaging_vertex_patch_stability-2}
    \begin{split}
    \|\AVP u \|_{L^2(K')} &\lesssim h_{K'}\| \mathcal{D}\widetilde u \|_{L^2(\Sref)} \lesssim
    h_{K'} \| d^{-1/2}_{(1,1)}\widetilde u \|_{L^2(\Tref)}
    \lesssim \| u \|_{L^2(K)},\\
    \|\AVP u \|_{H^1_h(K')} &\lesssim \| \mathcal{D}\widetilde u \|_{H^1(\Sref)}
    \lesssim \| d^{-1/2}_{(1,1)}\widetilde u \|_{L^2(\Tref)} + \| d^{-1/2}_{(1,1)}\nabla \widetilde u \|_{L^2(\Tref)}
    \lesssim \| u \|_{H^1_h(K)}.
    \end{split}
  \end{align}
  The properties of 
Lemmas~\ref{lemma:vertex_averaging_quad_to_trig}, (\ref{item:lemma:averaging_quad_to_trig-5})
  and \ref{lemma:vertex-part-edge-vanish}, (\ref{item:lemma:vertex-part-edge-vanish-9})
  ensure global continuity of $\AVP u$.
  The Duffy transform maps the edge $e_1:=\{0\}\times (0,1)$ to
  the edge $e:=\{(x,x),\; 0< x < 1\}$, and the edge $e_2:=(0,1) \times \{1\}$ to
  the vertex $(1,1)$. As $\widetilde u$ vanishes on $e$ and at $(1,1)$, 
  we get that $\AVP u$ vanishes on all of $\partial \omega_V$.   
  Summing the estimates~\eqref{eq:averaging_vertex_patch_stability} and~\eqref{eq:averaging_vertex_patch_stability-2}
  over all elements $K'$ of the patch $\omega_V$, 
  we get that the operator $\AVP$ is bounded as
  \begin{align}\label{lemma:averaging-vertex_patch-general:1}
    L^2(K) &\to L^2(\omega_V) \qquad \mbox{ and } \qquad 
    H^1_h(K) \to \widetilde H^1_h(\omega_V).
  \end{align}
  An interpolation argument yields (\ref{lemma:averaging-vertex_patch-general-1}).
  It remains to show~(\ref{lemma:averaging-vertex_patch-general-3}). 
  Using Lemma~\ref{lemma:weight-localize} and the statement~(\ref{lemma:averaging-vertex_patch-general-1})
  of the current lemma, we obtain
  \begin{align*}
    \| d_{V}^{-\theta} (u - \AVP u) \|_{L^2(\omega_V)}
    \lesssim \| d_V^{-\theta} (u - \AVP u) \|_{L^2(K)}
    + \| u \|_{H^\theta_h(\omega_V)}.
  \end{align*}
  It remains to bound the first term on the right-hand side.
  For notational simplicity assume $\widehat{K}=\Sref$, the case of $\widehat{K}=\Tref$
  follows along the same lines but is even simpler, as the Duffy transform can be omitted.
  A scaling argument and the triangle inequality yield
  \begin{align*}
    \| d_{V}^{-\theta} (u - \AVP u) \|_{L^2(K)}
    &\lesssim h_{K}^{1-\theta}
    \left(
    \| d_{(1,0)}^{-\theta} (\widehat u - \D \widetilde u) \|_{L^2(\Tref)}
    + \| d_{(1,0)}^{-\theta} (\widehat u - \D \widetilde u) \|_{L^2(\Sref\setminus\Tref)} \right).
  \end{align*}
  Since $d_{(1,0)}$ is bounded uniformly from below on $\Sref\setminus\Tref$,
  the same arguments as in~\eqref{eq:averaging_vertex_patch_stability-2} show
  \begin{align*}
    \| d_{(1,0)}^{-\theta} (\widehat u - \D \widetilde u) \|_{L^2(\Sref\setminus\Tref)}
    \lesssim 
    \| \widehat u \|_{L^2(\Sref\setminus\Tref)}
    + \| \D \widetilde u \|_{L^2(\Sref\setminus\Tref)}
    \lesssim 
    \| \widehat u \|_{L^2(\Sref)}
    \lesssim h_K^{-1} \| u \|_{L^2(K)},
  \end{align*}
  where we have used the boundedness of $\D$ and $\A_{\Sref}$.
  Using the same arguments, we see
  \begin{align*}
    \|d_{(1,0)}^{-\theta} \left(\widehat{u} - \D\widetilde{u}\right)\|_{L^2(\Tref)} 
    &\leq \;\,\| d_{(1,0)}^{-\theta} \left(\widehat{u} - \widetilde{u}\right) \|_{L^2(\Tref)} 
    + \| d_{(1,0)}^{-\theta} \left( \widetilde{u} - \D\widetilde{u}\right) \|_{L^2(\Tref)}\\
    &\!\!\!\!\stackrel{(\ref{eq:lemma:duffy_transform-5})}{\lesssim}\! \| d_{(1,0)}^{-\theta} \left(\widehat{u} - \widetilde{u}\right) \|_{L^2(\Tref)}
    + \| \widetilde u \|_{H^\theta(\Tref)} + \| d^{-1/2}_{(1,1)} \widetilde u \|_{L^2(\Tref)}\\
    &\lesssim\;\, \| \widehat u \|_{H^\theta(\Tref)}.
  \end{align*}
  Combining these estimates and using Lemma~\ref{lemma:K-vs-k} gives
  $\| d_{V}^{-\theta}(u-\AVP u)\|_{L^2(K)} \lesssim\norm{u}_{H^\theta_h(K)}$.
%

  To show \eqref{eq:lemma:averaging-vertex_patch-general-3-vprime}, we again only consider $\theta=1$.
  On the elements $K'$ with $\dist{V'}{K} > 0$ the estimate is just $L^2$ stability, therefore we may only consider
  elements $K'$ with $V' \in \overline{K'}$. 
  We note that $V' \in \overline{\omega_V}$ but $V \neq V'$ implies that there exists
  an edge $e \subset \partial \omega_V \cap \overline{K'}$ with $V' \in \overline{e}$. By construction, $\AVP u$ vanishes on $e$.
  Transforming to the reference element and applying Lemma~\ref{lemma:weighted-norm-equivalence} gives the stated estimate.

  Property~(\ref{lemma:averaging-vertex_patch-general-4}) follows immediately from
  Lemmas~\ref{lemma:vertex-part-edge-vanish} and~\ref{lemma:vertex_averaging_quad_to_trig}
  and the fact that the Duffy transform maps $\PP(\Tref) \to \QQ(\Sref)$.
\end{proof}
If $\TT$ consists of triangles and quadrilaterals, an  operator analogous to $\AVP$ from Lemma~\ref{lemma:averaging-vertex_patch-general}
  can be defined which preserves the polynomial degree.
\begin{corollary}\label{cor:averaging-vertex_patch-general}
  There exists an operator $\AVP^p: \SS^{\pp,1}(\TT)\to \SS^{\pp,1}(\TT|_{\omega_V})$,
  retaining the stability properties of Lemma~\ref{lemma:averaging-vertex_patch-general}
\end{corollary}
\begin{proof}
  We can mimic the proof of Lemma~\ref{lemma:averaging-vertex_patch-general}.
  For the definition of $\widetilde u$, we use the operator $\A_{\Sref}^p$ from
  Corollary~\ref{cor:vertex_averaging_quad_to_trig}.
  Then, if $K$ happens to be a square, instead of~\eqref{lemma:averaging-vertex_patch-general:1}
  we get that $\AVP^p$ is bounded as
  \begin{align*}
    \left( \QQ(K), \| \cdot \|_{L^2(K)} \right) \to L^2(\omega_V)
\qquad \mbox{ and } \qquad 
    \left( \QQ(K), \| \cdot \|_{H^1_h(K)} \right) \to \widetilde H^1_h(\omega_V).
  \end{align*}
  An interpolation argument and Lemma~\ref{lemma:GL}~(\ref{item:lemma:GL-i}) 
  show the result.
\end{proof}
\subsection{The edge parts}

\begin{lemma}
\label{lemma:averaging-edge_patch-general}
Let $\omega_e$ denote an edge patch.
Then there exists an operator 
$\AEP: L^2(\Gamma)\to L^2(\omega_e)$
with the following properties:
\begin{enumerate}[(i)]
  \item \label{it:lemma:averaging-edge_patch-general-1}
    For all $\theta \in [0,1]$ 
    and $d_{1}$ and $d_2$ denoting the distances to the endpoints of $e$, there
holds (provided the right-hand side is finite)
    \begin{align*}
      \|\AEP u\|_{\widetilde{H}^{\theta}_h(\omega_e)}
    &\lesssim \|u\|_{H^{\theta}_h(\omega_e)}
    + \|d_{1}^{-\theta} u\|_{L^2(\omega_e)} +  \| d_{2}^{-\theta} u \|_{L^2(\omega_e)}. 
    \end{align*}
  \item \label{it:lemma:averaging-edge_patch-general-3}    
    For all $\theta \in [0,1]$ and all other edges 
$e' \subset \overline{\omega_{e}}$ there holds (provided the right-hand sides are finite)
    \begin{align*}
      \|d_{e}^{-\theta} \left(u-\AEP u\right)\|_{L^2(\omega_{e})}
        &\lesssim \| u \|_{H^{\theta}_h(\omega_e)}
	+ \| d_{1}^{-\theta} u \|_{L^2(\omega_e)} + \| d_{2}^{-\theta} u \|_{L^2(\omega_e)}, \\
      \| d_{e'}^{-\theta} \AEP u \|_{L^2(\omega_{e})}
      &\lesssim \| u \|_{H^{\theta}_h(\omega_e)}
      + \| d_{1}^{-\theta} u \|_{L^2(\omega_e)} +  \| d_{2}^{-\theta} u \|_{L^2(\omega_e)}.
    \end{align*}
  \item \label{it:lemma:averaging-edge_patch-general-4}
    Let $u\in \SS^{\pp,1}(\TT)$ vanish in all vertices of $\omega_e$.
      If $\TT|_{\omega_e}$ consists of triangles only then $\AEP u \in \SS^{\pp,1}(\TT|_{\omega_e})$.
      If $\TT|_{\omega_e}$ consists of triangles and quadrilaterals, then
      $\AEP u \in \SS^{2\pp,1}(\TT|_{\omega_e})$.
\end{enumerate}
\end{lemma}
\begin{proof}
  Let $K$ be the element of $\omega_e$ with the lowest polynomial degree,
  and $K'$ the other element of $\omega_e$.
  Suppose that the corresponding element maps fulfill
  $F_K(\widehat e) = F_{K'}(\widehat e) = e$,
  where $\widehat e := ( 0,1 )\times \left\{ 0 \right\}$, but
  $F_K$ and $F_{K'}$ having different orientation on $e$.
  We distinguish four cases:
  \begin{enumerate}
    \item \label{it:lemma:averaging-edge_patch-general-case-T}
      $K$ is a triangle, i.e., $\widehat{K}=\Tref$:
      Define $\widetilde{u}:=\mathcal{A}^{\widehat{e}}_0 \left(u\circ F_K\right)$,
      $(\AEP u)|_{K}:= \widetilde{u} \circ F_{K}^{-1}$, and
      \begin{align*}
        (\AEP u)|_{K'}:=
        \begin{cases}
          \widetilde{u} \circ F_{K'}^{-1}  & \text{if $\widehat{K'}=\Tref$},\\
          \left(\mathcal{D} \widetilde{ u} \right) \circ F_{K'}^{-1}  & \text{if $\widehat{K'}=\Sref$}.
        \end{cases}
      \end{align*}
  \item \label{it:lemma:averaging-edge_patch-general-case-ST1}
    $K$ is a square, $K'$ is a triangle, and $p_{K} \leq p_{K'} \leq 2 p_{K}$:
    Define $\widetilde{u}:=\mathcal{A}^{\widehat{e}}_0 \left(u\circ F_{K'}\right)$ and
    $(\AEP u)|_{K'}:= \widetilde{u} \circ F_{K'}^{-1}$. On $K$, define
    $\left(\AEP u\right)|_{K}:=\left(\mathcal{D} \widetilde{u} \right) \circ F_K^{-1}$.
  \item \label{it:lemma:averaging-edge_patch-general-case-ST2}
    $K$ is a square, $K'$ is a triangle, and $p_{K'} > 2 p_{K}$:
    Define $\widetilde{u}:=\mathcal{A}^{\widehat{e}}_0 \left(u\circ F_{K}\right)$,
    $(\AEP u)|_{K'}:= \widetilde{u} \circ F_{K'}^{-1}$, and
    $(\AEP u)|_K := \left(\mathcal{D} \widetilde{u} \right) \circ F_K^{-1}$
  \item \label{it:lemma:averaging-edge_patch-general-case-SS}
    $K$ and $K'$ are squares:
    Define $\widetilde{u}:=\mathcal{A}^{\widehat{e}}_0 \left(u\circ F_{K}\right)$,
    $(\AEP u)|_{K'} := \left(\mathcal{D} \widetilde{u} \right) \circ F_{K'}^{-1}$,
    and $(\AEP u)|_K := \left(\mathcal{D} \widetilde{u} \right) \circ F_K^{-1}$.
  \end{enumerate}
  Let us show the boundedness in case (1), the other cases follow by the same argument.
  First, writing $\widehat u := u\circ F_K$, scaling arguments, and
  Lemma~\ref{lemma:edge-part-homogeneous-bc},~\eqref{item:lemma:edge-part-homogeneous-bc-2}, show for $\theta\in\left\{ 0,1 \right\}$
  \begin{align*}
    \|\AEP u\|_{H^\theta_h(K)} 
    &\lesssim h_K^{1-\theta}
    \left( 
    \| \A^{\widehat{e}}_0 \widehat u \|_{L^2(\Tref)} + \abs{\A^{\widehat{e}}_0 \widehat u}_{H^\theta(\Tref)}
    \right)
\lesssim h_K^{1-\theta}
    \left( 
    \| \widehat u \|_{H^\theta(\widehat K)} + \| d_{\widehat 1}^{-\theta} \widehat u \|_{L^2(\widehat K)}
    + \| d_{\widehat 2}^{-\theta} \widehat u \|_{L^2(\widehat K)}
    \right)\\
    &\lesssim 
    \| u \|_{H^\theta_h(K)} + \| d_1^{-\theta} u \|_{L^2(K)}
    + \| d_2^{-\theta} u \|_{L^2(K)}.
  \end{align*}
  If $K'$ is a triangle, the same estimate holds with $K'$ instead of $K$ 
  on the left-hand side.  If $K'$ is a square, we additionally invoke 
  Lemma~\ref{lemma:duffy_transform},~\eqref{lemma:duffy_transform-1}, and
  Lemma~\ref{lemma:edge-part-homogeneous-bc},~\eqref{item:lemma:edge-part-homogeneous-bc-7}, to show for $\theta=1$
  \begin{align*}
    \|\AEP u\|_{H^\theta_h(K')}
    &\lesssim h_K^{1-\theta}
    \left( 
    \| \mathcal{D}\A^{\widehat{e}}_0 \widehat u \|_{L^2(\Sref)} + \abs{\mathcal{D}\A^{\widehat{e}}_0 \widehat u}_{H^\theta(\Sref)}
    \right)\\
    &\lesssim h_K^{1-\theta}
    \left( 
    \| d^{-1/2}_{(1,1)} \A^{\widehat{e}}_0 \widehat u \|_{L^2(\Tref)} + \| d^{-1/2}_{(1,1)} \nabla \A^{\widehat{e}}_0 \widehat u \|_{L^2(\Tref)}
    \right)\\
    &\lesssim h_K^{1-\theta}
    \left( 
    \| \widehat u \|_{H^\theta(\widehat K)} + \| d_{\widehat 1}^{-\theta} \widehat u \|_{L^2(\widehat K)}
    + \| d_{\widehat 2}^{-\theta} \widehat u \|_{L^2(\widehat K)}
    \right)\\
    &\lesssim 
    \| u \|_{H^\theta_h(K)} + \| d_1^{-\theta} u \|_{L^2(K)}
    + \| d_2^{-\theta} u \|_{L^2(K)}.
  \end{align*}
  The same estimate is true for $\theta=0$. These arguments can be used for the cases (2), (3), and (4).
  Furthermore, as $\widetilde u$ vanishes on $\Tref \setminus \widehat e$, we conclude that
  $\mathcal{D}\widetilde u$ and hence $\AEP u$ vanish on $\Sref\setminus\widehat e$.
  Summing the last estimates over all elements in $\omega_e$, we
  get that the operator $\AEP$ is bounded uniformly as
  \begin{align}\label{lemma:averaging-edge_patch-general:2}
    \begin{split}
      L^2(K) &\to L^2(\omega_e)\\
      H^1_h(K,d_1^{-1}+d_2^{-1}) &\to \H^1_h(\omega_e).\\
    \end{split}
  \end{align}
  The statement~(\ref{it:lemma:averaging-edge_patch-general-1}) now follows
  by an interpolation argument and the following considerations:
  By Corollary~\ref{cor:reference-patch-theta-equivalence}
  and Lemma~\ref{lemma:equivalence-sobolev-plus-weight} we have for $\theta\in[0,1]$
  \begin{align*}
    h_K^{\theta-1}\| u \|_{\bigl( L_2(K), H^1_h(K,d_1^{-1}+d_2^{-1}) \bigr)_{\theta,2}}
    \sim
    \| \widehat u \|_{\bigl( L_2(\widehat K),
    H^1(\widehat K,d_{\widehat 1}^{-1} + d_{\widehat 2}^{-1}) \bigr)_{\theta,2}}
    \sim \| \widehat u \|_{H^\theta(\widehat K)}
    + \| (d_{\widehat 1}^{-\theta} + d_{\widehat 2}^{-\theta}) \widehat u \|_{L^2(\widehat K)},
  \end{align*}
  and according to Corollary~\ref{cor:reference-patch-theta-equivalence} we have 
  \begin{align*}
    \| \widehat u \|_{H^\theta(\widehat K)}
    + \| (d_{\widehat 1}^{-\theta} + d_{\widehat 2}^{-\theta}) \widehat u \|_{L^2(\widehat K)}
    \sim h_K^{\theta-1}\left[ \| u \|_{H_h^\theta(K)} +
      \| d_1^{-\theta} u \|_{L^2(K)} + \| d_2^{-\theta} u \|_{L^2(K)}
    \right].
  \end{align*}
  This finishes the proof of~(\ref{it:lemma:averaging-edge_patch-general-1})
  in the cases (1), (3), and (4). The same argument can be used
  in the case (2), only the notation $K'$ and $K$ has to be adapted correspondingly.
  In all cases, $\AEP$ reproduces the values on $e$, i.e., $u- \AEP u$ vanishes on $e$,
  and we can use
  Lemma~\ref{lemma:weighted-norm-equivalence},~(\ref{item:lemma:weighted-norm-equivalence-edge})
  and $1\lesssim h_{\omega_e}^{-1}$ to estimate
  \begin{align*}
    \norm{d_{e}^{-1} \left(u-\AEP u\right)}_{L^2(\omega_e)} &\lesssim \norm{u - \AEP u}_{H^1_h(\omega_e)}.
  \end{align*}
  Therefore,
  \begin{align*}
    \| d_{e}^{-\theta} \left(u-\AEP u\right) \|_{L^2(\omega_e)}
    \lesssim \norm{u - \AEP u}_{H^\theta_h(\omega_e)},
  \end{align*}
  and with the triangle inequality and~(\ref{it:lemma:averaging-edge_patch-general-1})
  we conclude~(\ref{it:lemma:averaging-edge_patch-general-3}).
  The proof of the estimate for $e'$
  follows along the same lines, but using that $\AEP u$ vanishes on $e'$.
  Property~(\ref{it:lemma:averaging-edge_patch-general-4}) follows immediately from
  Lemma~\ref{lemma:edge-part-homogeneous-bc},
  and the fact that the Duffy transform maps $\PP(\Tref) \to \QQ(\Sref)$.
\end{proof}

%
  \begin{corollary}\label{cor:averaging-edge_patch-general}
    There
    exists a modified operator $\mathcal{A}_{\omega_e}^p$, defined for
    all $u \in \SS^{\pp,1}(\TT)$ that vanish in all vertices of $\TT|_{\omega_e}$,
    such that $u$ is mapped to $\mathcal{A}_{\omega_e}^p u \in  \SS^{\pp,1}(\TT|_{\omega_e})$.
    If Assumption~\ref{ass:degree_distribution} holds, then the
    bounds from Lemma~\ref{lemma:averaging-edge_patch-general}~(\ref{it:lemma:averaging-edge_patch-general-1}),
    (\ref{it:lemma:averaging-edge_patch-general-3}) hold. If Assumption~\ref{ass:degree_distribution} is violated, the bounds only hold for $\theta\in \{0,1\}$.
\end{corollary}
\begin{proof}
  
    We  mimic the proof of Lemma~\ref{lemma:averaging-edge_patch-general}.
    In the case (1), the operators from Lemma~\ref{lemma:averaging-edge_patch-general} already have all the necessary mapping properties.
    In the  cases (3) and (4), we define $$\AEP^p u:=\widetilde{\mathcal{A}}^{p}_{\omega_e} \circ i_{p_K} \circ \mathcal{D} \circ \mathcal{A}^{\widehat{e}}_0 [u \circ F_K],$$
    with an auxiliary operator $\widetilde{A}^p_{\omega_e}$ on the reference square. $\widetilde{A}^p_{\omega_e} v$ is defined using $\widetilde{v}:=\mathcal{A}^{\widehat{e}}_0 {v}$ and distinguishing the cases
  \begin{enumerate}    
  \item[(3)]
    $K$ is a square, $K'$ is a triangle, and $p_{K'} > 2 p_{K}$:
    Define 
    $(\widetilde{A}_{\omega_e}^p v)|_{K'}:= \widetilde{v} \circ F_{K'}^{-1}$, and
    $(\widetilde{A}_{\omega_e}^p v)|_K := \left(i_{p_K} \circ \mathcal{D} \widetilde{v} \right) \circ F_K^{-1}$
  \item[(4)]
    $K$ and $K'$ are squares:
    Set $(\widetilde{A}_{\omega_e}^p v)|_{K'} := \left(i_{p_K} \circ \mathcal{D} \widetilde{v} \right) \circ F_{K'}^{-1}$,
    and $(\widetilde{A}_{\omega_e}^p v)|_K := \left(i_{p_K} \circ \mathcal{D} \widetilde{v} \right) \circ F_K^{-1}$.
  \end{enumerate}
  For functions that vanish on $\partial \Sref\setminus  \overline{\widehat{e}}$, the operator $\widetilde{A}^{p}_{\omega_e}$ satisfies the
  following stability conditions:
  \begin{align}
    \label{eq:dege_patch_stability_auxiliary}
      \Big(\widetilde{\mathcal{Q}}_{p}(\Sref,{\partial \Sref \setminus \overline{\widehat{e}}}) , \norm{\cdot}_{L^2(\Sref)}\Big) &\to L^2(\omega_e) \qquad \text{ and } \qquad 
      \Big(\widetilde{\mathcal{Q}}_{p}(\Sref,{\partial \Sref \setminus \overline{\widehat{e}}}),\norm{\cdot}_{H^1(\Sref)}\Big) \to \H^1_h(\omega_e).
  \end{align}  
  This can be seen similarly to Lemma~\ref{lemma:averaging-edge_patch-general} using the following insights:
  \begin{itemize}
    \item We additionally have to invoke the stability of the 
      Gau{\ss}-Lobatto operator given in Lemma~\ref{lemma:GL},~(\ref{item:lemma:GL-ii}).
    \item Via Lemma~\ref{lemma:weighted-norm-equivalence}~(\ref{item:lemma:weighted-norm-equivalence-i}),
      we can estimate $\|{d^{-1}_{\widehat{V}} v}\|_{L^2(\Sref)} \lesssim \|v\|_{H^{1}(\widehat{S})}$ for both vertices $\widehat{V}$ of $\widehat{e}$.
    \end{itemize}
  Interpolation of~\eqref{eq:dege_patch_stability_auxiliary} via Lemma~\ref{lemma:GL},~(\ref{item:lemma:GL-i}) gives the stability
   $ \|{\widetilde{\mathcal{A}}^p_{\omega_e} v}\|_{\widetilde{H}_h^{\theta}(\omega_e)}\lesssim \norm{v}_{\widetilde{H}^{\theta}(\Sref,\partial \Sref \setminus \widehat{e})}.$

   We note that the averaging $u \mapsto \mathcal{D} \circ \mathcal{A}^{\widehat{e}}_0 [u \circ F_K]$
  satisfies via Lemma~\ref{lemma:edge-part-homogeneous-bc},~(\ref{item:lemma:edge-part-homogeneous-bc-2}), (\ref{item:lemma:edge-part-homogeneous-bc-7}) and Lemma~\ref{lemma:duffy_transform}:
  \begin{align*}
    \norm{\mathcal{D} \circ \mathcal{A}^{\widehat{e}}_0 [u \circ F_K]}_{\widetilde{H}^{\theta}(\Sref,\partial \Sref \setminus \widehat{e})}
    &\lesssim \norm{u}_{H^{\theta}_{h}(K)} + \|d_{1}^{-\theta}u\|_{L^2(K)} + \|d_{2}^{-\theta} u \|_{L^2(K)}
  \end{align*}  
  and increases the polynomial degree at most by a factor of two.
  This and the stability of $i_{p_K}$
  implies that $\AEP$ is a stable operator satisfying the bound from Lemma~\ref{lemma:averaging-edge_patch-general},~(\ref{it:lemma:averaging-edge_patch-general-1}).

	The case (2), which is the case excluded by Assumption~\ref{ass:degree_distribution}, needs special considerations. If
    $K$ is a square, $K'$ is a triangle, and $p_{K} \leq p_{K'} \leq 2 p_{K}$,
    define $\widetilde{u}:=\mathcal{A}^e_0 \left(u\circ F_{K'}\right)$ and
    $(\AEP^p u)|_{K'}:= \widetilde{u} \circ F_{K'}^{-1}$. On $K$, define
    $\left(\AEP^p u\right)|_{K}:=\left(i_{p_K} \circ \mathcal{D} \widetilde{u} \right) \circ F_K^{-1}$.
    This operator is then bounded in the $L^2$- and $H^1_h$-norms, but the continuity of the resulting function 
    relies on the fact that $u|_{e} \in \mathcal{P}_{p_K}$, i.e., is of lower degree than the volume function used for the averaging.
    In order to derive stability in $H^{\theta}_h$-norms, we would need a result analogous to Lemma~\ref{lemma:GL},~(\ref{item:lemma:GL-i})
    for the space $\{u \in \PP(\Tref), u|_{e} \in \mathcal{P}_{q} \wedge u|_{\partial \Tref \setminus \widehat{e} }=0\}$, which is not available in the literature.

  The proof of Lemma~\ref{lemma:averaging-edge_patch-general},~(\ref{it:lemma:averaging-edge_patch-general-3}) can be adapted to the present case in a straight forward way.
\end{proof}
\subsection{Proof of Theorems~\ref{thm:main_decomposition} and~\ref{thm:decomposition-FEM-space}}
\label{sec:proofs}
We only show the case for general meshes
with all contributions in $\SS^{\pp,1}(\TT)$, the other cases follow analogously.
\begin{proof}[Proof of Theorem~\ref{thm:decomposition-FEM-space}]
  Let $I^1_h$ denote an operator of Scott-Zhang type and $u_1:=u-I^1_h u$. Then
  \begin{align}\label{eq:sz}
    \| u_1 \|_{\H^\theta(\Gamma)}+\| h^{-\theta} u_1 \|_{L^2(\Gamma)} \lesssim \norm{u}_{\H^{\theta}(\Gamma)}.
  \end{align}
  For every vertex $V \in \innervertices$ define $u_V:=\AVP^p u_1$ (see Corollary~\ref{cor:averaging-vertex_patch-general}).
  By Lemma~\ref{lemma:averaging-vertex_patch-general},  (\ref{lemma:averaging-vertex_patch-general-1})
  \begin{align}\label{eq:vertexpart}
    \sum_{V \in \innervertices}{\norm{u_V}_{\H^\theta_h(\omega_V)}^2}
      &\lesssim \sum_{V \in \innervertices}{\norm{ u_1 }^2_{H_h^{\theta}(\omega_V)}  }
    \lesssim  \norm{u}_{\H^{\theta}(\Gamma)}^2,
  \end{align}
  where in the last step we used that the patches $(\omega_V)_{V\in \innervertices}$ have finite overlap
  and~\eqref{eq:sz}.
  Extend $u_V$ by zero outside of $\omega_V$ and define
  $u_{\mathcal{V}}:=\sum_{V \in \innervertices}{u_V}$.
  Set $u_2:=u_1 - u_{\mathcal{V}}$. Let $d_{\mathcal{V}}:=\operatorname{dist}(\cdot,\allvertices)$
  denote the distance to all vertices of $\TT$.
  We claim the following estimates:
  \begin{align}
    \norm{u_2}_{\widetilde{H}^{\theta}(\Gamma)} + \| h^{-\theta} u_2 \|_{L^2(\Gamma)}
    &\lesssim \| u \|_{\widetilde{H}^{\theta}(\Gamma)},
    \label{eq:proof_decomposition_general_u3_sobolev}\\
    \| d_{\mathcal{V}}^{-\theta} u_2 \|_{L^2(\Gamma)}
    &\lesssim \norm{u}_{\widetilde{H}^{\theta}(\Gamma)} \label{eq:proof_decomposition_general_u3_wighted_l2}.
  \end{align}
  To show \eqref{eq:proof_decomposition_general_u3_sobolev},
  we note that since the contributions $u_V$ have local support, we can apply a coloring argument to
  localize the norms and apply the estimates~\eqref{eq:sz} and~\eqref{eq:vertexpart}
  \begin{align*}
    \| u_2 \|_{\widetilde{H}^{\theta}(\Gamma)}^2 + \| h^{-\theta} u_2 \|_{L^2(\Gamma)}^2
    &\lesssim \| u_1 \|^2_{\widetilde{H}^{\theta}(\Gamma)} + \| h^{-\theta} u_1 \|^2_{L^2(\Gamma)}
    + \sum_{V \in \innervertices}{\norm{u_{V}}^2_{\widetilde{H}_h^{\theta}(\omega_V)}} 
    \lesssim \norm{u}^2_{\H^{\theta}(\Gamma)}.
  \end{align*}
  To show \eqref{eq:proof_decomposition_general_u3_wighted_l2}, we consider
  a single element $K$ with vertices $ \mathcal{V}_K:=\left\{ V_1, \dots, V_k \right\}$.
  By shape-regularity, $d_{\mathcal{V}}\sim d_{\mathcal{V}_K}$ on $K$,
  and we also have
  $d_{\mathcal{V}_K}^{-\theta} = \max_{j=1,\dots,k} d_{V_j}^{-\theta} \leq \sum_{j=1}^k d_{V_j}^{-\theta}$.
  Using Lemma~\ref{lemma:averaging-vertex_patch-general}, (\ref{lemma:averaging-vertex_patch-general-3})
  and~\eqref{eq:sz}, we conclude
  \begin{align*}
    \| d_{\mathcal{V}}^{-\theta} u_2 \|_{L^2(K)}^2
    &\lesssim \sum_{j=1}^k \| d_{V_j}^{-\theta} (u_1 - \sum_{\ell=1}^k u_{V_\ell}) \|_{L^2(K)}^2
\\ &
\lesssim \sum_{j=1}^k
    \left( 
    \| d_{V_j}^{-\theta} (u_1-u_{V_j}) \|_{L^2(K)}^2
    + \sum_{\ell\neq j} \| d_{V_j}^{-\theta} u_{V_\ell} \|_{L^2(K)}^2
    \right)
    \lesssim \sum_{j=1}^k \| u_1 \|_{H_h^{\theta}(\omega_{V_j})}^2.
  \end{align*}
  Summing this estimate over all $K\in\TT$ and using the second estimate 
of~\eqref{eq:vertexpart}
  shows~\eqref{eq:proof_decomposition_general_u3_wighted_l2}.
  Next, define $u_{e}:=\AEP^p u_{2}$ for all edges $e \in \inneredges$, using $\AEP^p$ from Corollary~\ref{cor:averaging-edge_patch-general}. 
  Lemma~\ref{lemma:averaging-edge_patch-general}, (\ref{it:lemma:averaging-edge_patch-general-1})
  and~\eqref{eq:proof_decomposition_general_u3_sobolev},~\eqref{eq:proof_decomposition_general_u3_wighted_l2} then imply
  \begin{align}\label{eq:edgepart}
    \sum_{e \in \inneredges } \norm{u_e}_{\widetilde{H}_h^{\theta}(\omega_e)}^{2}    
    &\lesssim \norm{u}_{\widetilde{H}^{\theta}(\omega)}^2.
  \end{align}
  Set $u_{\mathcal{E}}:=\sum_{e\in \inneredges}{u_e}$ and $u_3:=u_2-u_{\mathcal{E}}$.
  We claim the estimates
  \begin{align}
\label{eq:claim-1}
    \| u_3 \|^2_{\widetilde{H}^{\theta}(\Gamma)} + \| h^{-\theta} u_3 \|^2_{L^2(\Gamma)}
    &\lesssim \norm{u}_{\widetilde{H}^{\theta}(\Gamma)}^2, \\
\label{eq:claim-2}
    \| d_{\mathcal{E}}^{-\theta} u_3 \|^2_{L^2(\Gamma)}
    &\lesssim \| u \|^2_{\widetilde{H}^{\theta}}(\Gamma).
  \end{align}
  (\ref{eq:claim-1}) follows from~\eqref{eq:edgepart} via the usual coloring argument to localize the norm.
  The weighted estimate (\ref{eq:claim-2}) follows analogously to the proof of~\eqref{eq:proof_decomposition_general_u3_wighted_l2},
  using Lemma~\ref{lemma:averaging-edge_patch-general},~(\ref{it:lemma:averaging-edge_patch-general-3}).
  Furthermore, (\ref{eq:claim-2}) implies that for the restrictions $u_K:=u_3|_{K}$ we can estimate
  $\norm{u_K}^2_{\widetilde{H}_h^{\theta}(K)}\leq \norm{u}_{\H^{\theta}(\Gamma)}^2$.
\end{proof}

\begin{proof}[Proof of Theorem~\ref{thm:main_decomposition}]
    Follows along the same lines as Theorem~\ref{thm:decomposition-FEM-space},
    one only has to replace the operators $\AVP^p$ and $\AEP^p$ with their continuous counterparts
    from Lemmas~\ref{lemma:averaging-vertex_patch-general} and~\ref{lemma:averaging-edge_patch-general} respectively.
\end{proof}

\textbf{Acknowledgments:} Financial support of A.R. by the Austrian Science Fund (FWF) through the 
doctoral school W1245 ``Dissipation and Dispersion in Nonlinear PDEs'' and the SFB 65 ``Taming complexity
in partial differential systems'' is gratefully acknowledged. The research of the author M.K. is supported
by Conicyt Chile through project FONDECYT 1170672 ``Fast space-time discretizations for fractional evolution equations''.

\bibliographystyle{amsalpha}
\bibliography{literature}

\end{document}